\documentclass[reqno]{amsart}
\usepackage[a4paper, margin=2cm]{geometry}
\usepackage{graphicx}
\usepackage{amssymb,amsmath}
\usepackage{amsthm,thmtools}
\usepackage{mathrsfs}
\usepackage{array}
\usepackage{comment}
\usepackage{mathtools}
\usepackage{stmaryrd}
\usepackage{tikz}
\usepackage{subfiles}
\usepackage{bbm}
\usetikzlibrary{arrows,arrows.meta,decorations.pathreplacing,decorations.markings,shapes,calc,cd,nfold}
\tikzset{labelsize/.style={font=\scriptsize}}
\tikzset{2cell/.style={-implies,double,double equal sign distance}}
\tikzset{3cell/.style={double equal sign distance, nfold = 3, arrows=-Implies}}
\usepackage[colorlinks=true,pagebackref]{hyperref}
\usepackage{cleveref}

\hypersetup{%
	colorlinks=true,
	allcolors=violet
}

\newcolumntype{N}{@{}m{0pt}@{}}

\usepackage{etoolbox}
\makeatletter
\patchcmd{\@counteralias}
 {\@ifdefinable{c@#1}}
 {\expandafter\@ifdefinable\csname c@#1\endcsname}
 {}{}
\makeatother

\makeatletter
\newcommand{\labeleditem}[1]{
\item[\text{#1}]\protected@edef\@currentlabel{\text{#1}}\phantomsection
}
\makeatother
\numberwithin{equation}{subsection}

\declaretheorem[style=plain,sibling=equation,name=Theorem]{theorem}
\declaretheorem[style=plain,sibling=theorem,name=Lemma]{lemma}
\declaretheorem[style=plain,sibling=theorem,name=Proposition]{proposition}
\declaretheorem[style=plain,sibling=theorem,name=Corollary]{corollary}

\declaretheorem[style=definition,qed=$\blacksquare$,sibling=theorem,name=Definition]{definition}

\declaretheorem[style=definition,qed=$\blacksquare$,sibling=theorem,name=Remark]{remark}

\declaretheorem[style=definition,sibling=theorem,name=Claim]{claim}

\crefname{theorem}{Theorem}{Theorems}
\crefname{section}{Section}{Sections}
\crefname{subsection}{Subsection}{Subsections}
\crefname{definition}{Definition}{Definitions}
\crefname{notation}{Notation}{Notations}
\crefname{example}{Example}{Examples}
\crefname{remark}{Remark}{Remarks}
\crefname{equation}{}{}
\crefname{construction}{Construction}{Constructions}
\crefname{corollary}{Corollary}{Corollaries}
\crefname{proposition}{Proposition}{Propositions}
\crefname{lemma}{Lemma}{Lemmas}
\crefname{appendix}{Appendix}{Appendices}
\crefname{claim}{Claim}{Claims}

\mathchardef\mhyphen="2D
\newcommand{\Set}{\mathbf{Set}}

\newcommand{\N}{\mathbb{N}}
\newcommand{\G}{\mathbb{G}}
\newcommand{\Pow}{\mathcal{P}}
\newcommand{\op}{\mathrm{op}}
\newcommand{\ar}{\mathop{\mathrm{ar}}}
\newcommand{\arity}{\mathop{\mathrm{ar}}}
\newcommand{\StrCats}[1]{{\mathbf{Str}\mhyphen{#1}\mhyphen\mathbf{Cat}}} % strict n-categories and strict functors
\newcommand{\WkCats}[1]{{\mathbf{Wk}\mhyphen{#1}\mhyphen\mathbf{Cat}}} % weak n-categories and strict functors
\newcommand{\mWkCats}[1]{{\mathbf{m}\mhyphen\mathbf{Wk}\mhyphen{#1}\mhyphen\mathbf{Cat}}} % marked weak n-categories and marking-preserving strict functors
\newcommand{\enGph}[1]{{#1}\mhyphen\mathbf{Gph}}
\newcommand{\GSet}{\mathbf{GSet}}
\newcommand{\id}[3]{\mathrm{id}_{#1}^{#2}\left(#3\right)} % identity cells in a weak omega-category; 1st argument: dimension of the identity cell; 2nd argument: name of the weak omega-category; 3rd argument: name of the source/target cell of the identity cell

\newcommand{\comp}[2]{\mathbin{{\ast}_{#1}^{#2}}}  % composition in a weak omega-category; 1st argument: dimension of the boundary; 2nd argument: name of the weak omega-category

\newcommand{\OO}{G}
\newcommand{\kk}{\boldsymbol{k}}

\newcommand{\uu}{\boldsymbol{u}}

\newcommand{\phiinv}{\phi^{\mathrm{rinv}}}
\newcommand{\uurinv}{\boldsymbol{u^{\mathrm{rinv}}}}
\newcommand{\uulinv}{\boldsymbol{u^{\mathrm{linv}}}}

\newcommand{\igen}[1]{i^{#1}} % generic i
\newcommand{\rgen}[1]{r^{#1}} % generic i
\newcommand{\igenb}[1]{i_\partial^{#1}} % generic i into boundary
\newcommand{\sourceb}[1]{\sigma_\partial^{#1}} % source into boundary

\newcommand{\C}[1]{\mathcal{C}^{#1}}
\newcommand{\mC}[1]{\mathrm{m}\mathcal{C}^{#1}} % C^n with c_n marked 
\newcommand{\Cst}[1]{\mathsf{C}^{#1}}
\newcommand{\E}[1]{\mathcal{E}^{#1}} % coinductive half-adjoint equivalence in the weak world
\newcommand{\EF}[1]{\mathcal{E}_{\mathcal{F}}^{#1}} % coinductive flat equivalence in the weak world
\newcommand{\iF}[1]{i_\FF^{#1}} % unit of algebraically injective replacement
\newcommand{\eomega}[1]{\lambda^{#1}} % leg of defining colimiting cocone of EF{n,omega}
\newcommand{\EFst}[1]{\mathsf{E}_{\mathsf{F}}^{#1}} % coinductive flat equivalence in the strict world
\newcommand{\iFst}[1]{i_{\mathsf{F}}^{#1}} % reflection of unit of algebraically injective replacement
\newcommand{\Est}[1]{\mathsf{E}^{#1}} % LMW-style model for coherent equivalence
\newcommand{\Ept}[1]{\mathsf{E}^{#1}_\star} % pointed LMW-style model for coherent equivalence

\newcommand{\ipt}[1]{i_\star^{#1}} % folk cofibration into pointed E
\newcommand{\rpt}[1]{r_\star^{#1}} % folk trivial fibration out of pointed E
\newcommand{\FF}{\mathcal{F}}
\newcommand{\HH}{\mathcal{H}}

\newcommand{\lunit}[3]{\lambda_{#1}^{#2}(#3)}
\newcommand{\runit}[3]{\rho_{#1}^{#2}(#3)}
\newcommand{\assoc}[5]{\alpha_{#1}^{#2}(#3,#4,#5)}
\newcommand{\lunitinv}[3]{\check\lambda_{#1}^{#2}(#3)}
\newcommand{\runitinv}[3]{\check\rho_{#1}^{#2}(#3)}
\newcommand{\associnv}[5]{\check\alpha_{#1}^{#2}(#3,#4,#5)}
\newcommand{\eqclass}[1]{\left\llbracket#1\right\rrbracket}

\newcommand{\cylarrow}{\curvearrowright}
\newcommand{\proj}[2]{\pi_{#1}^{#2}}

\newcommand{\bridge}[5]{{U}_{#1}^{#2}\left({#3},{#4},{#5}\right)}
\newcommand{\OR}{\mathcal{E}^1_\mathrm{OR}}
\newcommand{\ORst}{\widehat{\omega\mathcal{E}}}
\newcommand{\ORm}[1]{\mathcal{E}^{1,#1}_\mathrm{OR}}
\newcommand{\ORstm}[1]{\widehat{\omega\mathcal{E}}^{(#1)}}
\newcommand{\cosep}[1]{\mathsf{W}^{#1}}
\newcommand{\AInj}{\mathbf{AInj}} % the category of algebraically injective objects
\newcommand{\marking}[1]{T^{#1}} % marked cells in E^{n,m} that are not in the image of E^{n,m-1}
\newcommand{\zig}[1]{\underline{#1}}
\newcommand{\zag}[1]{\overline{#1}}
\newcommand*{\sys}{\Xi}
\newcommand*{\sysm}[1]{\Xi^{#1}}

\newcommand*{\concat}{\mathbin{+\!\!\!+}}
\newcommand{\fulllabel}{\mathrm{fdl}}
\newcommand{\Pre}{\mathrm{Pre}}
\newcommand{\pst}{\operatorname{pst}}
\newcommand{\rinv}{\operatorname{rinv}}
\newcommand{\linv}{\operatorname{linv}}
\newcommand{\ppair}[1]{\langle #1\rangle} % parallel pair of morphisms
\newcommand{\Inj}[1]{{#1}\mhyphen\mathbf{Inj}}
\newcommand{\cof}[1]{{#1}\mhyphen\mathrm{cof}}

\newcommand*{\Scell}{\mathtt{\Sigma}} % suspension of a cell (unit)
\newcommand{\tw}[1]{{\textstyle\frac{C}{F}^{#1}}}
\newcommand*{\twp}{{\textstyle\frac{E}{E}}} % system for OR
\newcommand*{\twpn}[1]{{\textstyle\frac{E}{E}^{#1}}} % sequence of systems for OR
\DeclareMathOperator{\colim}{\operatorname{colim}} % colimit
\newcommand*{\FFp}{{\mathcal{D}}} %{\overline{\mathcal{F}}}
\newcommand{\xF}{x_\FF}
\newcommand{\yF}{y_\FF}
\newcommand{\uF}{u_\FF}
\newcommand{\vF}{v_\FF}
\newcommand{\wF}{w_\FF}
\newcommand{\pF}{p_\FF}
\newcommand{\qF}{q_\FF}
\newcommand{\KF}{K_\FF}
\newcommand{\JF}{J_\FF}
\newcommand{\kF}{k_\FF}
\newcommand{\JFst}{J_{\mathsf{F}}}
\newcommand{\KH}{K_\HH}
\newcommand{\xH}{x_\HH}
\newcommand{\yH}{y_\HH}
\newcommand{\uH}{u_\HH}
\newcommand{\vH}{v_\HH}
\newcommand{\pH}{p_\HH}
\newcommand{\qH}{q_\HH}
\newcommand{\rH}{r_\HH}
\newcommand{\kH}{k_\HH}
\newcommand{\ORx}{{x}_\mathrm{OR}}
\newcommand{\ORy}{{y}_\mathrm{OR}}
\newcommand{\ORu}{{u}_\mathrm{OR}}
\newcommand{\ORv}{{v}_\mathrm{OR}}
\newcommand{\ORw}{{w}_\mathrm{OR}}
\newcommand{\ORe}[1]{e^{#1}}
\newcommand{\ORp}[1]{f_p^{#1}}
\newcommand{\ORq}[1]{f_{q}^{#1}}
\newcommand{\ORr}[1]{f_r^{#1}}
\newcommand{\ORep}[1]{g^{#1}}
\newcommand{\ORpp}[1]{h_p^{#1}}
\newcommand{\ORqp}[1]{h_q^{#1}}
\newcommand{\ORrp}[1]{h_r^{#1}}

\newcommand{\unit}{p}
\newcommand{\counit}{q}
\newcommand{\tri}{r}

\begin{document}

\title{\texorpdfstring{$\omega$}{ω}-equifibrations between strict and weak \texorpdfstring{$\omega$}{ω}-categories}
\author{Soichiro Fujii}
\address{National Institute of Informatics, Tokyo, Japan}
\email{s.fujii.math@gmail.com}

\author{Keisuke Hoshino}
  \address{Research Institute for Mathematical Sciences, Kyoto University, Kyoto, Japan}
  \email{hoshinok@kurims.kyoto-u.ac.jp}

\author{Yuki Maehara}
  \address{Tokyo Metropolitan University, Tokyo, Japan}
  \email{ymaehar@tmu.ac.jp}

\date{November 13, 2025}


\begin{abstract}
We study \textit{$\omega$-equifibrations} between weak $\omega$-categories in the sense of Batanin--Leinster.
We define $\omega$-equifibrations as a natural weak $\omega$-categorical analogue of isofibrations between categories, and 
show that they can be characterised via the right lifting property with respect to a suitable set $J$ of strict $\omega$-functors. 
The definition of $J$ involves the construction of a certain weak $\omega$-category $\mathcal{E}^1$ which, roughly speaking, is freely generated by an equivalence $1$-cell in a ``coherent'' manner. 
We show that the strict version of $\mathcal{E}^1$ coincides with Ozornova and Rovelli's \textit{coherent walking $\omega$-equivalence} $\widehat{\omega\mathcal{E}}$. 
The $\omega$-equifibrations between strict $\omega$-categories coincide with the fibrations in the folk model structure.
\end{abstract}

\maketitle

\section{Introduction}\label{sec:intro}
The purpose of the present paper is to introduce and study the notion of \emph{$\omega$-equifibration}, which is a weak $\omega$-categorical analogue of isofibrations between ordinary categories.
Here we are adopting the algebraic approach to weak $\omega$-categories pioneered by Batanin \cite{Batanin_98} and Leinster \cite{Leinster_book}, in which they are defined as the Eilenberg--Moore algebras for a suitable monad on the category of globular sets.

The definition of $\omega$-equifibration is the ``obvious'' one, namely a strict $\omega$-functor $f \colon X \to Y$ between weak $\omega$-categories is an $\omega$-equifibration if and only if, for any $(n-1)$-cell $x$ in $X$ and for any equivalence $n$-cell (in the coinductive sense recalled below) $u \colon fx \to y$ in $Y$, there exists an equivalence $n$-cell $\bar u \colon x \to \bar y$ in $X$ such that $f \bar u = u$.
In the special case where both $X$ and $Y$ are ordinary categories, this recovers precisely the isofibrations, or equivalently, precisely the fibrations in the folk model structure on $\mathbf{Cat}$.
With an eye towards constructing a weak $\omega$-categorical analogue of the folk model structure (cf.\ \cite{Lafont_Metayer_Worytkiewicz_folk_model_str_omega_cat,Lack-bicat,Lack-Gray}), our main result provides a characterisation of $\omega$-equifibrations between weak $\omega$-categories via the right lifting property.

A subtle point here is that the above definition of $\omega$-equifibration refers to the \emph{property}, as opposed to a \emph{structure}, of $u$ (or $\bar u$) being an equivalence.
Recall (from \cite{FHM1}, or from \cite{Lafont_Metayer_Worytkiewicz_folk_model_str_omega_cat} in the case of strict $\omega$-categories) that an $n$-cell $u \colon a \to b$ in a weak $\omega$-category is an \emph{equivalence} if and only if there exist an $n$-cell $v \colon b \to a$ and \emph{equivalence} $(n+1)$-cells $p \colon u \comp{}{} v \to \id{}{}{a}$ and $q \colon v \comp{}{} u \to \id{}{}{b}$; this seemingly circular definition is formalised using coinduction (see \cref{def:spherical-eq}).
Although it is not difficult to construct a weak $\omega$-category $\E{n}$ with a specified $n$-cell $i_n$ such that
\begin{center}
	an $n$-cell $u$ in a weak $\omega$-category $Y$ is an equivalence if and only if\\ 
    there exists a strict $\omega$-functor $g\colon \E n\to Y$ with $g(i_n)=u$,
\end{center}
specifying such $g$ necessarily involves \emph{choosing} a witness to the fact that $u$ is an equivalence.
Therefore the lifting property of the form
\[
\begin{tikzcd}[row sep = large]
	\C{n-1}
	\arrow [r]
	\arrow [d, "{s(i_n)}", swap] &
	X
	\arrow [d, "f"]\\
	\E{n}
	\arrow [r, "g", swap] 
	\arrow [ur, dashed] &
	Y
\end{tikzcd}
\]
(where $\C{n-1}$ is the weak $\omega$-category freely generated by a single $(n-1)$-cell and $s(i_n)$ is the strict $\omega$-functor picking out the $(n-1)$-dimensional source of $i_n$) encodes that we can lift not only the equivalence $n$-cell, but also the \emph{chosen} witness (specified by $g$) through $f$, which is in general a stronger condition than $f$ being an $\omega$-equifibration.
In order for this lifting property to characterise precisely the $\omega$-equifibrations, intuitively the weak $\omega$-category $\E{n}$ must be ``coherent'' so that for any equivalence $n$-cell $u$, the choice of corresponding $g$ is essentially unique (although the precise statement of this essential uniqueness is beyond the scope of this paper).

Inspired by various known results (and conjectures) \cite{Lack-bicat,Lack-Gray,hottbook,Rice-inv,HLOR}, we provide two solutions to this problem, or more precisely two characterisations of equivalences that lead to concrete models for such ``coherent'' $\E{n}$.
The first of these two apparently weakens the original definition by allowing the left and right inverses to be distinct.
It still captures the same notion of equivalence essentially for the familiar reason that, if a cell $u \colon a \to b$ has a right inverse $v$ and a left inverse $w$, then we have
\[
v \sim \id{}{}{b} \comp{}{} v \sim (w \comp{}{} u) \comp{}{} v \sim w \comp{}{} (u \comp{}{} v) \sim w \comp{}{} \id{}{}{a} \sim w
\]
which implies that both $v$ and $w$ qualify as a two-sided inverse of $u$.
In contrast, the second characterisation apparently strengthens the original definition by asking for an additional witness for one of the triangular identities.
That this strengthening is only superficial can be deduced from the well-known special case in dimension $2$, namely that any equivalence in a bicategory can be upgraded to an adjoint equivalence.

Our homotopical intuition for why these two characterisations may yield models for ``coherent'' $\E{n}$ is as follows.
In the original definition, where we ask for the equivalence $u$ to have a single inverse $v$ and two witnesses $p \colon u \comp{}{} v \to \id{}{}{a}$ and $q \colon v \comp{}{} u \to \id{}{}{b}$, topologically these cells form a sphere:\footnote{The diagrams \cref{eqn:spherical-eq-diag,eqn:flat-eq-diag} are admittedly somewhat ambiguous. For example, the boundary of $p$ would be more clearly represented as 
\[
\begin{tikzpicture}[baseline = 0]
		\node (1) at (0,0.5) {$x$};
		\node (2) at (3,0) {$y$.};
		\node (3) at (0,-0.5) {$x$};
		
		\draw[->] (1) to node [labelsize, swap, auto] {$\id{1}{}{x}$} (3);
		\draw[->, bend left=10] (1) to node [labelsize, auto] {$u$} (2);
		\draw[->,bend left=10] (2) to node [labelsize, auto] {$v$} (3);
		\draw[2cell] (2,0) to node [labelsize, swap, auto] {$p$} (1,0);
	\end{tikzpicture}
\]
The sole purpose of \cref{eqn:spherical-eq-diag,eqn:flat-eq-diag} is to convey the geometric intuition.}
\begin{equation}\label{eqn:spherical-eq-diag}
\begin{tikzpicture}[baseline = 0]
		\node (1) at (0,0) {$a$};
		\node (2) at (3,0) {$b$};

		\draw[->, bend left=60] (1) to node [labelsize, auto] {$u$} (2);
		\draw[->, bend left=60] (2) to node [labelsize, auto] {$v$} (1);
		\draw[2cell, bend left=20] (2,-0.2) to node [labelsize, auto] {$p$} (1,-0.2);
		\draw[2cell,dashed, bend left=20] (1,0.2) to node [labelsize, auto] {$q$} (2,0.2);
	\end{tikzpicture}
\end{equation}
which is not contractible, and it is this homotopical non-triviality that spoils the essential uniqueness of $g\colon \E{n}\to Y$ corresponding to $u$.
The first characterisation cuts open this sphere by separating the left and right inverses:
\begin{equation}\label{eqn:flat-eq-diag}
\begin{tikzpicture}[baseline = 0]
		\node (1) at (0,0) {$a$};
		\node (2) at (3,0) {$b$};

		\draw[->] (1) to node [labelsize, midway, fill=white] {$u$} (2);
		\draw[->, bend right=60] (2) to node [labelsize, swap, auto] {$w$} (1);
		\draw[->, bend left=60] (2) to node [labelsize, auto] {$v$} (1);
		\draw[2cell] (1,0.4) to node [labelsize, auto] {$q$} (2,0.4);
		\draw[2cell] (2,-0.4) to node [labelsize, auto] {$p$} (1,-0.4);
	\end{tikzpicture}
\end{equation}
whereas the second characterisation attaches an extra $3$-dimensional cell in the interior of the sphere, both of which produce a contractible space.

We then investigate how our results relate to what is known about \emph{strict} $\omega$-categories.
In particular, we prove that a strict $\omega$-functor between strict $\omega$-categories is an $\omega$-equifibration if and only if it is a fibration in the folk model structure constructed by Lafont, M\'etayer, and Worytkiewicz  \cite{Lafont_Metayer_Worytkiewicz_folk_model_str_omega_cat}.
Consequently, by taking either of the two families of models for ``coherent'' $\E{n}$ and reflecting them to the strict $\omega$-categories, we obtain an explicit generating set of folk trivial cofibrations.
Moreover, for the version with separate left and right inverses, the strict $\omega$-categorical reflection of (a certain universal model of) $\E{1}$ turns out to be isomorphic to Ozornova and Rovelli's \textit{coherent walking $\omega$-equivalence} $\widehat{\omega\mathcal{E}}$ \cite{OR}.
As a corollary, we recover the main theorem of \cite{HLOR}, namely that $\ORst$ is weakly equivalent to the terminal strict $\omega$-category in the folk model structure.

\subsection*{Outline of the paper}
After recalling the notion of weak (and strict) $\omega$-category in \cref{sec:prelim}, we show in \cref{sec:equivalence-cells} that two definitions of equivalence cells in a weak $\omega$-category---one via a two-sided inverse and the other via separate left and right inverses---coincide. 
Then in \cref{sec:omega-equifibrations}, we define $\omega$-equifibrations and show that they can be characterised by the right lifting property with respect to a suitable set $J$ of strict $\omega$-functors. 
More precisely, we give a general procedure to turn a suitable set $K$ of morphisms between \emph{marked} weak $\omega$-categories into such a set $J$, and then provide an example $K_\mathcal{F}$ of $K$ based on the definition of equivalences via separate left and right inverses.
(Another example $K_\mathcal{H}$ of $K$ based on ``half-adjoint'' equivalences is given in \cref{appendix-on-halfadjoint-lift-one-step}.)
Because this procedure involves choices of factorisations, the set $J$ is not completely determined by $K$.
However, the \emph{algebraic} small object argument \cite{Garner_understanding} provides a universal one among all possible $J$, and in particular we obtain a universal set $J_\mathcal{F}$ of strict $\omega$-functors from $K_\mathcal{F}$. 
In \cref{presentation-of-En}, we give two explicit presentations of the morphisms in $J_\mathcal{F}$. 
The second of these allows us to show that the strict $\omega$-categorical reflection of the weak $\omega$-category $\mathcal{E}_\mathcal{F}^1$ appearing in $J_\mathcal{F}$ is isomorphic to Ozornova and Rovelli's coherent walking $\omega$-equivalence \cite{OR}. 
Finally, in \cref{sec:fibrations-between-strict}, we show that the $\omega$-equifibrations between strict $\omega$-categories coincide with the fibrations in the folk model structure \cite{Lafont_Metayer_Worytkiewicz_folk_model_str_omega_cat}.

\section{Preliminaries}\label{sec:prelim}
In this section, we briefly review the notions of strict and weak $\omega$-category (the latter in the sense of Leinster~\cite{Leinster_book}, inspired by an earlier work by Batanin \cite{Batanin_98}) and collect necessary results.

\subsection{Globular sets}
\label{subsec:globular-sets}
We start with the notion of globular set, which forms a basis of the notions of strict and weak $\omega$-category.

A \emph{globular set} $X$ consists of 
\begin{itemize}
    \item a set $X_n$ for each natural number $n\geq 0$, and 
    \item functions $s_n,t_n\colon X_{n+1}\to X_n$ for each $n\geq 0$,
\end{itemize}
satisfying 
\begin{equation}\label{eqn:globular-identities}
    s_n\circ s_{n+1}=s_n\circ t_{n+1} \quad \text{and} \quad t_n\circ s_{n+1}=t_n\circ t_{n+1}
\end{equation}
for each $n\geq 0$.
An element of $X_n$ is called an \emph{$n$-cell} of $X$. Given an $n$-cell $u \in X_n$ with $n\geq 1$, we write $u\colon x\to y$ to mean $s_{n-1}(u)=x$ and $t_{n-1}(u)=y$.
Two $n$-cells $x$ and $y$ of $X$ are said to be \emph{parallel} if
\begin{itemize}
    \item $n=0$, or
    \item $n \ge 1$, $s_{n-1}(x)=s_{n-1}(y)$ and $t_{n-1}(x)=t_{n-1}(y)$.
\end{itemize}
Note that \cref{eqn:globular-identities} is equivalent to saying that, for each $u\in X_n$ with $n\geq 2$, the $(n-1)$-cells $s_{n-1}(u)$ and $t_{n-1}(u)$ are parallel.

A \emph{globular map} $f\colon X\to Y$ between globular sets is a family $(f_n\colon X_n\to Y_n)_{n\geq 0}$ of functions commuting with $s_n$ and $t_n$. We often omit the subscript in $f_n$, writing $fu$ for the image of an $n$-cell $u\in X_n$ under $f_n\colon X_n\to Y_n$. 
We denote by $\GSet$ the category of globular sets and globular maps between them. It is clear that $\GSet$ is the presheaf category $[\G^\op,\Set]$ over a suitable category $\G$ generated by the graph 
\[
\begin{tikzpicture}[baseline=-\the\dimexpr\fontdimen22\textfont2\relax ]
      \node(0) at (0,0) {$0$};
      \node(1) at (2,0) {$1$};
      \node(d) at (4,0) {$\cdots$};
      \node(n) at (6,0) {$n$};
      \node(d2) at (8,0) {$\cdots$;};
      
      \draw [->,transform canvas={yshift=3pt}] (0) to node[auto, labelsize] 
      {$\sigma^0$} (1); 
      \draw [->,transform canvas={yshift=-3pt}] (0) to node[auto, swap,labelsize] 
      {$\tau^0$} (1); 
      \draw [->,transform canvas={yshift=3pt}] (1) to node[auto, labelsize] 
      {$\sigma^1$} (d); 
      \draw [->,transform canvas={yshift=-3pt}] (1) to node[auto, 
      swap,labelsize] 
      {$\tau^1$} (d); 
      \draw [->,transform canvas={yshift=3pt}] (d) to node[auto, labelsize] 
      {$\sigma^{n-1}$} (n); 
      \draw [->,transform canvas={yshift=-3pt}] (d) to node[auto, 
      swap,labelsize] 
      {$\tau^{n-1}$} (n); 
      \draw [->,transform canvas={yshift=3pt}] (n) to node[auto, labelsize] 
      {$\sigma^{n}$} (d2); 
      \draw [->,transform canvas={yshift=-3pt}] (n) to node[auto, 
      swap,labelsize] 
      {$\tau^{n}$} (d2);
\end{tikzpicture}
\]
see e.g.\ \cite[Section~2.1]{FHM1}.

For each $n\geq 0$, we write $G^n$ for the representable globular set $\G(-,n)$ and $\iota^n \colon \partial \OO^n \to \OO^n$ for the inclusion of its boundary (i.e., the largest proper globular subset).
The globular set $\partial \OO^n$ represents parallel pairs of $(n-1)$-cells for $n \ge 1$, whereas $\partial \OO^0$ is the initial globular set $\emptyset$.
For any parallel pair $(x,y)$ of $(n-1)$-cells in $X$, we denote the corresponding map by $\ppair{x,y} \colon \partial \OO^n \to X$. 
Note that for $n\geq 1$, an $n$-cell $u$, and a parallel pair $(x,y)$ of $(n-1)$-cells in $X$, we have $u\colon x\to y$ in $X$ if and only if the diagram 
\[
\begin{tikzpicture}[baseline=-\the\dimexpr\fontdimen22\textfont2\relax ]
      \node(00) at (0,1) {$\partial\OO^n$};
      \node(01) at (2,1) {$X$};
      \node(10) at (0,-1) {$\OO^n$};
      
      \draw [->] (00) to node[auto, labelsize] {$\ppair{x,y}$} (01); 
      \draw [->] (00) to node[auto,swap,labelsize] {$\iota^n$} (10);   
      \draw [->] (10) to node[swap,auto,labelsize] {$u$} (01);
\end{tikzpicture}
\]
in $\GSet$ commutes.

We denote the globular map $\G(-,\sigma^{n-1})\colon G^{n-1}\to G^n$ by $\sigma^{n-1}$. Its factorization through $\iota^n$ is written as $\sourceb{n-1}$:
\begin{equation}\label{eqn:sigma-sigma-prime}
    \begin{tikzpicture}[baseline=-\the\dimexpr\fontdimen22\textfont2\relax ]
      \node(00) at (2,1) {$\partial\OO^n$};
      \node(01) at (0,-1) {$\OO^{n-1}$};
      \node(10) at (2,-1) {$\OO^n.$};
      
      \draw [<-] (00) to node[swap,auto, labelsize] {$\sourceb{n-1}$} (01); 
      \draw [->] (00) to node[auto,labelsize] {$\iota^n$} (10);   
      \draw [<-] (10) to node[auto,labelsize] {$\sigma^{n-1}$} (01);
\end{tikzpicture}
\end{equation}

\subsection{Strict and weak \texorpdfstring{$\omega$}{ω}-categories}\label{subsec:strict-and-weak-omega-cats}
For detailed discussions of strict and weak $\omega$-categories, see e.g.\ \cite[Sections~1.4 and 9.2]{Leinster_book}. 
Instead of repeating their definitions, here we will introduce notations and state facts used in this paper.

There are (at least) two different notions of map between strict or weak $\omega$-categories (see e.g.\ \cite{Garner_homomorphisms}), but in this paper we will only use that of \emph{strict $\omega$-functor}.
Whereas this is arguably not the most natural notion of map, 
our primary objective is an extension of the folk model structures to weak $\omega$-categories, and we aim to achieve this in the category of weak $\omega$-categories and strict $\omega$-functors. 
(In the strict \cite{Lafont_Metayer_Worytkiewicz_folk_model_str_omega_cat} or low-dimensional \cite{Lack-bicat,Lack-Gray} cases, the folk model structures have been constructed on categories with strict functors as morphisms.)

The category $\StrCats{\omega}$ (resp.\ $\WkCats{\omega}$)
of \emph{strict $\omega$-categories} and \emph{strict $\omega$-functors} (resp.\ \emph{weak $\omega$-categories} and \emph{strict $\omega$-functors}\footnote{In \cite{Cottrell_Fujii_hom,FHM1,FHM2}, the category of weak $\omega$-categories and strict $\omega$-functors was called 
$\mathbf{Wk}\mhyphen\omega\mhyphen\mathbf{Cat}_\mathrm{s}$, and $\WkCats{\omega}$ referred to the category of weak $\omega$-categories and \emph{weak $\omega$-functors} \cite{Garner_homomorphisms}. Because we will not discuss weak $\omega$-functors in this paper, we have adopted the simpler notation.}) is monadic over $\GSet$, with the corresponding monad $T$ (resp.\ $L$).
By definition of weak $\omega$-category \cite{Leinster_book}, there is a monad morphism $\ar\colon L\to T$ (for \emph{arity}; see e.g.\ \cite[Subsection~2.5]{FHM1} or \cite[Subsection~2.3]{FHM2} for more explanation), and hence we have a commutative triangle 
\begin{equation}\label{eqn:triangle-over-GSet}
\begin{tikzpicture}[baseline=-\the\dimexpr\fontdimen22\textfont2\relax ]
      \node(00) at (0,1) {$\StrCats\omega$};
      \node(01) at (4,1) {$\WkCats{\omega}$};
      \node(10) at (2,-1) {$\GSet$};
      
      \draw [->] (00) to node[auto, labelsize] {inclusion} (01); 
      \draw [->] (00) to node[auto,swap,labelsize] {$U^T$} (10);   
      \draw [<-] (10) to node[swap,auto,labelsize] {$U^L$} (01);
\end{tikzpicture}
\end{equation}
of right adjoint functors. 
Here, the inclusion functor $\StrCats{\omega}\to \WkCats\omega$ is induced by the monad morphism $\ar$ (i.e., it sends $(X,TX\xrightarrow{\xi}X)\in \StrCats{\omega}$ to $(X,LX\xrightarrow{\ar_X}TX\xrightarrow{\xi}X)\in \WkCats{\omega}$), and is 
not only faithful but also full, since each component $\ar_X\colon LX\to TX$ of the monad morphism $\ar$ is an epimorphism in $\GSet$.
(To verify the last statement, observe that $\ar_X$ has, by definition, the right lifting property with respect to $\iota^n\colon \partial G^n\to G^n$ for each $n\geq 1$, and it also has the right lifting property with respect to $\iota^0\colon \partial G^0\to G^0$ because the unit of $T$ induces a bijection $X_0\cong (TX)_0$ and $\ar_X$ commutes with the units. This implies that $\ar_X$ is surjective in every dimension.)
\begin{proposition}\label{Wk-omega-Cat-LFP}
    All categories in \eqref{eqn:triangle-over-GSet} are locally finitely presentable, and all functors in it are finitary right adjoints.
\end{proposition}
\begin{proof}
    The category $\GSet$ is a presheaf category and hence is locally finitely presentable. The remaining claims follow from the facts that 
    the monads $T$ and $L$ are finitary; the former is \cite[Theorem~F.2.2]{Leinster_book} and the latter is  \cite[Proposition~2.2.5]{FHM2}.
\end{proof}

The left adjoints of $U^T$ and $U^L$ are denoted by $F^T$ and $F^L$, respectively.
The left adjoint of the inclusion functor $\StrCats\omega\to\WkCats{\omega}$ will be referred to as the \emph{strict $\omega$-categorical reflection}.
For each $n\geq 0$, we write 
\begin{itemize}
    \item $\Cst n$ for the free strict $\omega$-category $F^TG^n$ on $G^n$,
    \item $\partial \Cst n$ for the free strict $\omega$-category $F^T\partial G^n$ on the boundary of $G^n$,
    \item $\C n$ for the free weak $\omega$-category $F^LG^n$ on $G^n$, and
    \item $\partial \C n$  for the free weak $\omega$-category $F^L\partial G^n$ on the boundary of $G^n$.
\end{itemize}
(The letter C stands for ``cell''.)
The images of \eqref{eqn:sigma-sigma-prime} under $F^T$ and $F^L$ are written as 
\begin{equation}\label{eqn:sigma-sigma-prime-cat}
\begin{tikzpicture}[baseline=-\the\dimexpr\fontdimen22\textfont2\relax ]
      \node(00) at (2,1) {$\partial\Cst n$};
      \node(01) at (0,-1) {$\Cst{n-1}$};
      \node(10) at (2,-1) {$\Cst n,$};
      
      \draw [<-] (00) to node[swap,auto, labelsize] {$\sourceb{n-1}$} (01); 
      \draw [->] (00) to node[auto,labelsize] {$\iota^n$} (10);   
      \draw [<-] (10) to node[auto,labelsize] {$\sigma^{n-1}$} (01);
\end{tikzpicture}
\qquad\text{and}\qquad
\begin{tikzpicture}[baseline=-\the\dimexpr\fontdimen22\textfont2\relax ]
      \node(00) at (2,1) {$\partial\C n$};
      \node(01) at (0,-1) {$\C{n-1}$};
      \node(10) at (2,-1) {$\C n$};
      
      \draw [<-] (00) to node[swap,auto, labelsize] {$\sourceb{n-1}$} (01); 
      \draw [->] (00) to node[auto,labelsize] {$\iota^n$} (10);   
      \draw [<-] (10) to node[auto,labelsize] {$\sigma^{n-1}$} (01);
\end{tikzpicture}
\end{equation}
respectively.

The inclusion functor $\StrCats\omega\to \WkCats\omega$ allows us to regard 
strict $\omega$-categories as a special kind of weak $\omega$-categories.
In particular, the operations in a weak $\omega$-category introduced below are also available in a strict $\omega$-category (and are in fact usually included as basic operations in its definition).
Similarly, various facts about these operations in a weak $\omega$-category stated below are also valid in a strict $\omega$-category, 
often for much simpler reasons. 

Thus a \emph{weak $\omega$-category} is an Eilenberg--Moore algebra 
$(X,\xi)$ of the monad $L$, consisting of a globular set $X$ and a globular map $\xi\colon LX\to X$ satisfying the axioms for Eilenberg--Moore algebra. 
(Of course, this just reduces the definition of a weak $\omega$-category to that of the monad $L$; for more details of the latter, see \cite[Section~9.2]{Leinster_book} or \cite[Section~2]{FHM1}.)
We will often denote the weak $\omega$-category $(X,\xi)$ simply by $X$.
As explained in \cite[Definition~2.5.2]{FHM1}, for each weak $\omega$-category $X$, we have the following operations (encoded in the monad $L$). 
\begin{itemize}
    \item Given a natural number $n\geq 1$ and an $(n-1)$-cell $x$ of $X$, we have an $n$-cell $\id{n}{X}{x}\colon x\to x$ of $X$.
    \item Given a natural number $n\geq 1$ and $n$-cells $u\colon x\to y$ and $v\colon y\to z$ of $X$, we have an $n$-cell $u\comp{n-1}{X}v\colon x\to z$ of $X$. 
\end{itemize}
(That is, each weak $\omega$-category has the structure of an \emph{$\omega$-precategory} in the sense of \cite[Definition~1]{Cheng_dual}.)
When $X$ is clear from the context, we omit the superscript.
Since these operations are derived from the monad $L$, any strict $\omega$-functor preserves them (strictly):
\begin{proposition}\label{str-fun-pres-comp-and-id}
    Let $f\colon X\to Y$ be a strict $\omega$-functor between weak $\omega$-categories. Then we have $f\bigl(\id{n}{X}{x}\bigr)=\id{n}{Y}{fx}$ and $f(u\comp{n-1}{X}v)=fu\comp{n-1}{Y}fv$. 
\end{proposition}
\begin{proof}
    See the discussion above Definition~2.5.2 of \cite{FHM1}.
\end{proof}

Moreover, in any weak $\omega$-category $X$, we have a class of \emph{equivalence} cells.\footnote{Equivalence cells are called (weakly or pseudo) \emph{invertible} cells in \cite{Cheng_dual,Cottrell_Fujii_hom,FHM1,FHM2} and \emph{reversible} cells in \cite{Lafont_Metayer_Worytkiewicz_folk_model_str_omega_cat}.}
Since equivalence cells will be studied in detail in \cref{sec:equivalence-cells}, we defer their definition. 
For now, it suffices to state that equivalence cells in a weak $\omega$-category $X$ form a subset of $\coprod_{n\geq 1}X_n$, and writing $x\sim y$ if $x$ and $y$ are (necessarily parallel) $n$-cells of $X$ such that there exists an equivalence $(n+1)$-cell $u\colon x\to y$ in $X$, we have the following:

\begin{proposition}[{\cite[Corollary~3.3.15]{FHM1}; see \cite[Proposition~4.4]{Lafont_Metayer_Worytkiewicz_folk_model_str_omega_cat} for the strict case}]\label{sim-is-eq-rel}
    For any weak $\omega$-category $X$, $\sim$ is an equivalence relation on the set $\coprod_{n\geq 0}X_n$ of its cells. 
\end{proposition}

\begin{proposition}[{\cite[Theorem~3.3.7]{FHM1}; see \cite[Proposition~4.4]{Lafont_Metayer_Worytkiewicz_folk_model_str_omega_cat} for the strict case}]\label{sim-is-congruence}
    Let $X$ be a weak $\omega$-category. 
    Suppose that $u,u'\colon x\to y$ and $v,v'\colon y\to z$ are $n$-cells of $X$ with $n\geq 1$. 
    If $u\sim u'$ and $v\sim v'$, then we have $u\comp{n-1}{}v\sim u'\comp{n-1}{}v'$. 
\end{proposition}

\begin{proposition}[\textnormal{\cite[Corollary~3.3.17]{FHM1}}]
	\label{prop:invariance-of-equivalence}
	Let $X$ be a weak $\omega$-category, $n\geq 1$, and $u,v\colon x\to y$ be a parallel pair of $n$-cells in $X$ such that $u\sim v$. Suppose that $u$ is an equivalence. Then $v$ is an equivalence too. 
\end{proposition}

\begin{proposition}[{\cite[Proposition~3.2.5]{FHM1}}]\label{associativity}
    Let $X$ be a weak $\omega$-category and $p\colon x\to y$, $q\colon y\to z$, and $r\colon z\to w$ be $n$-cells of $X$ with $n\geq 1$. Then we have $(p\comp{n-1}{}q)\comp{n-1}{}r\sim p\comp{n-1}{}(q\comp{n-1}{}r)$. 
\end{proposition}

\begin{proposition}[{\cite[Proposition~3.3.5]{FHM1}}]\label{unit-law}
    Let $X$ be a weak $\omega$-category and $u\colon x\to y$ be an $n$-cell of $X$ with $n\geq 1$. Then we have $\id{n}{}{x}\comp{n-1}{}u\sim u\sim u\comp{n-1}{}\id{n}{}{y}$. 
\end{proposition}

\begin{proposition}[{\cite[Proposition~3.2.1]{FHM1}; see \cite[Lemma~4.3]{Lafont_Metayer_Worytkiewicz_folk_model_str_omega_cat} for the strict case}]\label{str-functor-pres-eq}
    Any strict $\omega$-functor between weak $\omega$-categories preserves equivalence cells.
\end{proposition}

\subsection{Whiskering operations}\label{subsec:whiskering}

Whereas the above simple operations in a weak (or strict) $\omega$-category (endowing it with an $\omega$-precategory structure) usually suffice for the purpose of this paper, sometimes we use the following additional operations in a weak $\omega$-category. 

Let $X$ be a globular set. 
For natural numbers $0\leq n<m$, we write $s_n\colon X_m\to X_n$ for the composite 
\[
\begin{tikzpicture}[baseline=-\the\dimexpr\fontdimen22\textfont2\relax ]
      \node(n) at (0,0) {$X_m$};
      \node(n-1) at (2,0) {$X_{m-1}$};
      \node(d) at (4,0) {$\cdots$};
      \node(m) at (6,0) {$X_n$};
      
      \draw [->] (n) to node[auto, labelsize] 
      {$s_{m-1}$} (n-1);
      \draw [->] (n-1) to node[auto, labelsize] 
      {$s_{m-2}$} (d);
      \draw [->] (d) to node[auto, labelsize] 
      {$s_{n}$} (m);
\end{tikzpicture}
\]
and similarly for $t_n\colon X_m\to X_n$.

Now suppose that $X$ is a weak $\omega$-category. 
We have the following operations, which are special cases of \cite[Definition~2.3.5]{FHM2}. 
\begin{itemize}
    \item Let $1\leq n< m$ be natural numbers. Given an $n$-cell $u$ and an $m$-cell $v\colon s\to t$ of $X$ such that $t_{n-1}(u)=s_{n-1}(v)$, we have an $m$-cell 
    \[u\comp{n-1}{}v\colon u\comp{n-1}{}s\to u\comp{n-1}{}t\]
    of $X$. 
    Similarly, given an $m$-cell $u\colon s\to t$ and an $n$-cell $v$ of $X$ such that $t_{n-1}(u)=s_{n-1}(v)$, we have an $m$-cell 
    \[u\comp{n-1}{}v\colon s\comp{n-1}{}v\to t\comp{n-1}{}v\] 
    of $X$. 
\end{itemize}
Intuitively, these are \emph{whiskering} operations composing e.g.
\[
\begin{tikzpicture}
		\node (0) at (0,0) {$\bullet$};
		\node (1) at (2,0) {$\bullet$};
		\node (2) at (4,0) {$\bullet$.};

		\draw[->] (0) to node [labelsize, auto] {$u$} (1);
		\draw[->, bend right = 30] (1) to node [labelsize, swap,auto] {$t$} (2);
		\draw[->, bend left = 30] (1) to node [labelsize, auto] {$s$} (2);
		\draw[->, 2cell] (3,0.2) to node [labelsize, auto] {$v$} (3,-0.2);
	\end{tikzpicture}
\]
Analogously to \cref{str-fun-pres-comp-and-id,sim-is-congruence}, strict $\omega$-functors preserve whiskerings, and whiskering an $n$-cell to a pair of $m$-cells connected by an equivalence $(m+1)$-cell yields again a pair connected by an equivalence $(m+1)$-cell.
Moreover, we have the following.

\begin{proposition}\label{whiskering-ess-0-surj}
    Let $X$ be a weak $\omega$-category and $u\colon x\to y$ be an equivalence $n$-cell of $X$ with $n\geq 1$.
    \begin{enumerate}
        \item For any $n$-cell $w\colon x\to z$ in $X$, there exists an $n$-cell $v\colon y\to z$ in $X$ such that $w\sim u\comp{n-1}{}v$.
        \item For any $m>n$, any parallel pair of $(m-1)$-cells $s$ and $t$ in $X$ with $s_{n-1}(s)=s_{n-1}(t)=y$, and any $m$-cell $w\colon u\comp{n-1}{}s\to u\comp{n-1}{}t$ in $X$, there exists an $m$-cell $v\colon s\to t$ in $X$ such that $w\sim u\comp{n-1}{}v$.
    \end{enumerate}
    Moreover, in both of these situations, $v$ is an equivalence if and only if $w$ is.
\end{proposition}
\begin{proof}
    Both (1) and (2) follow from \cite[Theorem~3.2.14]{FHM2}, in view of \cite[Remark~2.4.4]{FHM2} (or from \cite[Lemma~4.6]{Lafont_Metayer_Worytkiewicz_folk_model_str_omega_cat} in the strict case).

    If $v$ is an equivalence, then so is $w$ by \cite[Theorem~3.3.7]{FHM1}. The converse is shown in the last paragraph of the proof of \cite[Proposition~3.2.22]{FHM2}, given \cref{prop:invariance-of-equivalence}.
\end{proof}
The dual statements of \cref{whiskering-ess-0-surj} (concerning post-composition of an equivalence cell) also hold.

\section{Two definitions of equivalence cells}\label{sec:equivalence-cells}

There are two seemingly different ways to define equivalence cells in a weak $\omega$-category $X$. Informally, these two definitions can be stated as follows.
\begin{itemize}
	\item First definition \cite{Cheng_dual,Lafont_Metayer_Worytkiewicz_folk_model_str_omega_cat,FHM1,FHM2}: An $n$-cell $u\colon x\to y$ in $X$ is an equivalence if there exists an $n$-cell $v\colon y\to x$ (a two-sided inverse) in $X$ such that $u\comp{n-1}{}v\sim \id{n}{}{x}$ and $v\comp{n-1}{}u\sim\id{n}{}{y}$. 
	\item Second definition \cite{HLOR}: An $n$-cell $u\colon x\to y$ in $X$ is an equivalence if there exist $n$-cells $v\colon y\to x$ (a right inverse) and $w\colon y\to x$ (a left inverse) in $X$ such that $u\comp{n-1}{}v\sim \id{n}{}{x}$ and $w\comp{n-1}{}u\sim\id{n}{}{y}$. 
\end{itemize}
In this section, we adapt an argument from \cite{Rice-inv} and show that these two definitions in fact define the same class of cells in $X$.
In order to avoid confusion, \emph{in this section only} we use the terms \emph{spherical equivalence} and \emph{flat equivalence} to refer to the notions of equivalence defined in these two manners (see \cref{def:spherical-eq,def:flat-eq}).

\begin{remark}
    Whereas the fact that the two notions of equivalence coincide would probably follow from \cite[Theorem~18]{Rice-inv}, we decided to include a proof for the following reasons.
    \begin{itemize}
        \item In order to apply \cite[Theorem~18]{Rice-inv} to weak $\omega$-categories, one has to show that weak $\omega$-categories \emph{respect the graphical calculus} in the sense of \cite[Definition~7]{Rice-inv}, which we find highly non-trivial.
        Although it is sufficient to only verify the special cases of \cite[Definition~7]{Rice-inv} that are actually used in the proof of \cite[Theorem~18]{Rice-inv}, the reader would then have to go through that proof in order to check that we have indeed covered all the necessary cases, which seems burdensome.
        \item We found the proofs of \cite{Rice-inv} rather hard to follow for those not familiar with type-theoretic treatment of coinduction (such as ourselves). 
        In \cref{subsec:coinduction}, we will state and justify the necessary coinductive proof principles in ordinary mathematical terms.
    \end{itemize}
    Nevertheless, a quick inspection would reveal the similarity between our argument below and that in \cite{Rice-inv}.
\end{remark}

\subsection{Formalising coinduction}\label{subsec:coinduction}
The above two definitions of equivalence cells are \emph{coinductive} in nature (equivalence $n$-cells are defined in terms of the relation $\sim$ between $n$-cells, which are in turn defined via equivalence $(n+1)$-cells), and we will formalise such concepts in lattice-theoretic terms.

Let us fix a complete lattice $A$ throughout this subsection.  
Then any monotone map $F\colon A\to A$ has the largest post-fixed point $\nu F\in A$, that is,
\begin{itemize}
    \item $\nu F\leq F(\nu F)$ holds, and
    \item for any $a \in A$ with $a \leq Fa $, we have $a \leq \nu F$,
\end{itemize}
which is in fact a fixed point $\nu F=F(\nu F)$ (see e.g.\ \cite[Remark~3.1.2]{FHM1}).

In the case of the two definitions of equivalence cells in a weak $\omega$-category $X$, we take $A$ to be the powerset lattice $\Pow\bigl(\coprod_{n\geq 0}X_n\bigr)$ of the set of all cells in $X$.

\begin{definition}[\textnormal{\cite[Remark~3.1.2]{FHM1}}]
	\label{def:spherical-eq}
	Let $X$ be a weak $\omega$-category. Define the monotone map $\Phi\colon \Pow\bigl(\coprod_{n\geq 0}X_n\bigr)\to \Pow\bigl(\coprod_{n\geq 0}X_n\bigr)$ by 
	\begin{multline*}
		\Phi(S)=
		\bigl\{%
			\,(u\colon x\to y)\in X_n\,%
			\big\vert\, n\geq 1,\ \ 
			\exists (v\colon y\to x)\in X_n,
		\\%
			\exists \bigl({p}\colon u\comp{n-1}{}v\to \id{n}{}{x}\bigr)\in S\cap X_{n+1},\ \ 
			\exists \bigl({q}\colon v\comp{n-1}{}u\to \id{n}{}{y}\bigr)\in S\cap X_{n+1}\,
		\bigr\}.
	\end{multline*}
	We say that a cell in $X$ is a \emph{spherical equivalence} if it is in $\nu\Phi\subseteq \coprod_{n\geq 0}X_n$. 
\end{definition}

The adjective ``spherical'' derives from the fact that, for example, a witness for a $1$-cell $u\colon x\to y$ in a weak $\omega$-category $X$ to be a spherical equivalence consists of a $1$-cell $v\colon y\to x$ together with spherical equivalence $2$-cells $p\colon u\comp{0}{}v\to\id{1}{}{x}$ and $q\colon v\comp{0}{}u\to\id{1}{}{y}$, which constitute a spherical shape; see \cref{eqn:spherical-eq-diag}.

\begin{definition}
	\label{def:flat-eq}
	Let $X$ be a weak $\omega$-category. Define the monotone map $\Psi\colon \Pow\bigl(\coprod_{n\geq 0}X_n\bigr)\to \Pow\bigl(\coprod_{n\geq 0}X_n\bigr)$ by 
	\begin{multline*}
		\Psi(S)=
		\bigl\{%
			\,(u\colon x\to y)\in X_n\,%
			\big\vert\, n\geq 1,\ \ 
			\exists (v\colon y\to x)\in X_n,\ \ 
			\exists (w\colon y\to x)\in X_n,
		\\%
			\exists \bigl({p}\colon u\comp{n-1}{}v\to \id{n}{}{x}\bigr)\in S\cap X_{n+1},\ \ 
			\exists \bigl({q}\colon w\comp{n-1}{}u\to \id{n}{}{y}\bigr)\in S\cap X_{n+1}\,
		\bigr\}.
	\end{multline*}
	We say that a cell in $X$ is a \emph{flat equivalence} if it is in $\nu\Psi\subseteq \coprod_{n\geq 0}X_n$. 
\end{definition}

Clearly we have $\Phi(S)\subseteq \Psi(S)$ for each $S\in \Pow\bigl(\coprod_{n\geq 0}X_n\bigr)$, from which it follows $\nu\Phi\subseteq \nu\Psi$, i.e., that every spherical equivalence is a flat equivalence. 

In our previous work \cite{FHM1,FHM2}, we always considered spherical equivalences.
In particular, various facts mentioned in \cref{subsec:strict-and-weak-omega-cats} (including \cref{unit-law,associativity} which will be used below) are to be interpreted as statements regarding spherical equivalences until we prove the two notions to coincide.

In the next subsection, we will exhibit two closure properties that the class of flat equivalences enjoys.
In the language of lattices, these properties can be formulated as the relation $G(\nu F) \le \nu F$ where $G \colon A \to A$ is a suitable monotone map encoding the construction under which the class is to be closed.
The rest of this subsection is devoted to establishing a couple of general results which will be useful when proving such closure properties.

The relation $G(\nu F) \le \nu F$ says that the largest \emph{post}-fixed point $\nu F$ of $F$ is also a \emph{pre}-fixed point of $G$.
It turns out to be useful to consider more general pre-fixed points of $G$.
More precisely, for any element $a \in A$ and monotone map $G\colon A\to A$, we define 
	\[
		\Pre_{a }(G)=\left\{\,
			b\in A
			\,|\,
			a \leq b,\ 
			Gb\leq b\,
		\right\}.
	\]
    
The following proposition formalises an argument used in the proof of \cite[Theorem~3.3.7]{FHM1}.
\begin{proposition}
	\label{prop:nuF-is-G-closed-F}
	Let $F,G\colon A\to A$ be monotone maps.
	If $\Pre_{\nu F}(G)$ is closed under $F$, then we have $G(\nu F)\leq \nu F$.
\end{proposition}
\begin{proof}
	The subset $\Pre_{\nu F}(G)\subseteq A$ is closed under arbitrary meets in $A$, and hence in particular has the least element $b$. We have $\nu F\leq b$ by the definition of $\Pre_{\nu F}(G)$.
    By assumption, we have $Fb\in \Pre_{\nu F}(G)$. The minimality of $b$ then implies $b\leq Fb$, i.e., that $b$ is a post-fixed point of $F$. It follows that we have $b\leq \nu F$.
    Therefore $\nu F=b\in \Pre_{\nu F}(G)$ and hence in particular $G(\nu F)\leq \nu F$.
\end{proof}
\begin{remark}
	\label{rmk:first-condition-is-redundant}
	In order to check that the assumption of \cref{prop:nuF-is-G-closed-F} is satisfied, we only need to show that for any $b\in \Pre_{\nu F}(G)$, we have $GFb\leq Fb$. The condition $\nu F\leq Fb$ follows from $\nu F\leq b$ since we have $\nu F=F(\nu F)$. 
\end{remark}
We shall also use the following variant. 

\begin{corollary}
	\label{cor:nuF-is-G-closed-F-wedge-id}
	Let $F,G\colon A\to A$ be monotone maps.
	If $\Pre_{\nu F}(G)$ is closed under $F\wedge \mathrm{id}_A\colon A\to A$, then we have $G(\nu F)\leq \nu F$.
\end{corollary}
\begin{proof}
	This is a special case of \cref{prop:nuF-is-G-closed-F} because we have $\nu F=\nu (F\wedge \mathrm{id}_A)$, the post-fixed points of $F$ being the same as the post-fixed points of $F\wedge \mathrm{id}_A$.
\end{proof}

\subsection{Closure properties of flat equivalences}
    Here we establish various closure properties of flat equivalences. To this end, we need to recall more details about the monad $L$ on $\GSet$ for weak $\omega$-categories. 
	Recall from \cite[Section~2]{FHM1} that, for any globular set $X$, an $n$-cell of the globular set $LX$ is a pair $(\phi,\uu)$ consisting of an $n$-cell $\phi$ of $L1$ (where $1$ is the terminal globular set) and an \emph{$n$-dimensional pasting diagram} $\uu$ in $X$; the latter is given by a suitable list of cells of dimension $\leq n$ in $X$, in an appropriate configuration (see \cite[Definition~2.2.1]{FHM1}). 
    As in \cite[Definition~3.3.6]{FHM1}, for an $n$-dimensional pasting diagram $\uu$ in $X$,
    we denote by $\fulllabel(\uu)$ the set of all $n$-cells in $X$ appearing in the list $\uu$ (called the \emph{full-dimensional labels} in \cite{FHM1}). 
	If $(X,\xi)$ is a weak $\omega$-category, then the structure map $\xi\colon LX\to X$ maps each $(\phi,\uu)\in (LX)_n$ to $\xi(\phi,\uu)\in X_n$. Intuitively, $\xi(\phi,\uu)$ is the (pasting) composite of the cells in $\uu$ according to the \emph{pasting instruction} $\phi$. 

	Throughout the rest of this subsection, we fix a weak $\omega$-category $(X,\xi)$.

	In \cite[Theorem~3.3.7]{FHM1}, it is shown that the class $\nu\Phi$ of all spherical equivalences in $X$ is closed under pasting. More precisely, if we define the monotone map $\pst\colon\mathcal{P}(\coprod_{n\geq 0} X_n)\to\mathcal{P}(\coprod_{n\geq 0} X_n)$ by 
	\[
	\pst(S) = \bigl\{\,\xi(\phi,\uu) \,\big\vert\, n\geq 0,\quad (\phi,\uu) \in (LX)_n, \quad \fulllabel(\uu) \subseteq S\,\bigr\},
	\]
	then we have $\pst(\nu\Phi)\subseteq \nu\Phi$.
	This is proved by showing that $\Pre_{\nu\Phi}(\pst)$ is closed under $\Phi$;  see \cref{prop:nuF-is-G-closed-F}. 

	We first prove that the class $\nu\Psi$ of all flat equivalences in $X$ is also closed under pasting. In light of \cref{prop:nuF-is-G-closed-F}, it suffices to show the following.
\begin{proposition}
	\label{prop:pstclosed}
	$\Pre_{\nu\Psi}(\pst)$ is closed under $\Psi$.
\end{proposition}
\begin{proof}
		We can adapt the proof of \cite[Theorem~3.3.7]{FHM1} to the current situation; the details (using the notation of \cite{FHM1}) are as follows.

		We first modify \cite[Definition~3.3.8]{FHM1}.
		Let $S$ be a set of cells of $X$.
		\begin{itemize}
			\item Given any $n\geq 1$ and any $n$-cell $u\colon x\to y$ of $X$, an $n$-cell $v\colon y\to x$ of $X$ is called a \emph{right $S$-inverse} of $u$ if there exist an $n$-cell $w\colon y\to x$ and $(n+1)$-cells ${p}\colon u\comp{n-1}{}v\to \id{n}{}{x}$ and ${q}\colon w\comp{n-1}{}u\to\id{n}{}{y}$ in $X$, with $p,q\in S$. 
			\item Similarly, an $n$-cell $w\colon y\to x$ of $X$ is called a \emph{left $S$-inverse} of $u$ if there exist an $n$-cell $v\colon y\to x$ and $(n+1)$-cells ${p}\colon u\comp{n-1}{}v\to \id{n}{}{x}$ and ${q}\colon w\comp{n-1}{}u\to\id{n}{}{y}$ in $X$, with $p,q\in S$. 
			\item Let $n\geq 1$, $(\phi,\uu)\in(LX)_n$, and $\kk=\ar(\phi)$. A \emph{right $S$-inverse instruction} of $(\phi,\uu)$ is an $n$-cell $(\phiinv,\uurinv)$ of $LX$ satisfying the following conditions.
			\begin{itemize}
				\item $s^{L1}_{n-1}(\phiinv)=t^{L1}_{n-1}(\phi)$, 
				\item $t^{L1}_{n-1}(\phiinv)=s^{L1}_{n-1}(\phi)$, and
				\item $\uurinv\in(TX)_n$ is obtained from $\uu$ by replacing, for each $(n-1)$-transversal component \cite[Subsection~2.2]{FHM1} $0 \le i \le j \le r$ of $\kk$, the corresponding segment
				\[
					\begin{bmatrix}
						u_{i} & & \dots & & u_{j}\\
						& \underline u_{i+1} & \dots & \underline u_{j} &
					\end{bmatrix}
				\]
				with
				\[
					\begin{bmatrix}
					v_{j} & & \dots & & v_{i}\\
					& \underline u_{j} & \dots & \underline u_{i+1} &
					\end{bmatrix},
				\]
				where $v_{\ell}$ is a right $S$-inverse of $u_{\ell}$ for each $i\le \ell\le j$.
			\end{itemize}
			\item Similarly, define the notion of \emph{left $S$-inverse instruction}.
		\end{itemize}
		Just like the remark immediately after \cite[Definition~3.3.8]{FHM1}, we note that an $n$-cell $u$ in $X$ admits a right $S$-inverse if and only if it admits a left $S$-inverse, if and only if $u\in\Psi(S)$; and an $n$-cell $(\phi,\uu)$ in $LX$ admits a right $S$-inverse instruction if and only if it admits a left $S$-inverse instruction, if and only if $\fulllabel(\uu)\subseteq \Psi(S)$. 

		Now, as in the third paragraph of the proof of \cite[Theorem~3.3.7]{FHM1}, we take $S\in \Pre_{\nu\Psi}(\pst)$. 
		In view of \cref{rmk:first-condition-is-redundant}, it suffices to show that $\pst(\Psi(S))\subseteq \Psi(S)$. 
		To this end, instead of (3.3.12) in \cite{FHM1}, we prove the following statement by induction on the number of $n$-cells in the list $\uu$ (denoted by $\|{\ar(\phi)}\|^{(n)}$ in \cite{FHM1}):
		\begin{equation*}
			\parbox{\dimexpr\linewidth-5em}{for each $n\geq 1$, each $(\phi,\uu)\in(LX)_n$,
			each right $S$-inverse instruction $(\phiinv,\uurinv)$ of $(\phi,\uu)$, and each left $S$-inverse instruction $(\phiinv,\uulinv)$ of $(\phi,\uu)$, there exist $(n+1)$-cells
			\[
				\xi(\phi,\uu)\comp{n-1}{X}\xi(\phiinv,\uurinv)\to \id{n}{X}{s^X_{n-1}\xi(\phi,\uu)}
			\]
			and 
			\[
				\xi(\phiinv,\uulinv)\comp{n-1}{X}\xi(\phi,\uu)\to \id{n}{X}{t^X_{n-1}\xi(\phi,\uu)}
			\]
			in $S$ (or equivalently in $\pst(S)$).}
		\end{equation*}
		This can be done exactly as in the proof of \cite[Theorem~3.3.7]{FHM1}, bearing in mind that because we already know that spherical equivalences (called \emph{invertible cells} in \cite{FHM1}) are flat equivalences, all cells induced by the coherence (\cite[Proposition~3.2.5]{FHM1}) or the unit law (\cite[Proposition~3.3.5]{FHM1}) are flat equivalences and hence in $S$.
\end{proof}

	To show that all flat equivalences are spherical equivalences, we also need to show that every right inverse of a flat equivalence is again a flat equivalence. (Note that the analogous statement for spherical equivalences is trivial.) To state this more precisely, we define the monotone map $\rinv\colon\mathcal{P}(\coprod_{n\geq 0} X_n)\to\mathcal{P}(\coprod_{n\geq 0} X_n)$ as follows.
	\begin{multline*}
		\rinv(S)=
		\bigl\{%
			\,(v\colon y\to x)\in X_n\,%
			\big\vert\, n\geq 1,\ \ 
			\exists (u\colon x\to y)\in X_n,\ 
			\exists (w\colon y\to x)\in X_n,
		\\%
			\exists \bigl({p}\colon u\comp{n-1}{}v\to \id{n}{}{x}\bigr)\in S\cap X_{n+1},\ \ 
			\exists \bigl({q}\colon w\comp{n-1}{}u\to \id{n}{}{y}\bigr)\in S\cap X_{n+1}\,
		\bigr\}
	\end{multline*}
	In the terminology introduced in the proof of \cref{prop:pstclosed}, $\rinv(S)$ is the set of all cells in $X$ which are right $S$-inverses of some cells in $X$.
	(One can also define the monotone map $\linv$ for left inverses similarly, but we do not need it.)
	The claim is that we have $\rinv(\nu\Psi)\subseteq \nu\Psi$. In light of \cref{cor:nuF-is-G-closed-F-wedge-id}, it suffices to show the following.
\begin{proposition}
	\label{prop:linvclosed}
	$\Pre_{\nu\Psi}(\pst\lor\rinv)$ is closed under $\Psi\land\mathrm{id}$.
\end{proposition}
\begin{proof}
	Let $S\in\Pre_{\nu\Psi}(\pst\lor\rinv)= \Pre_{\nu\Psi}(\pst)\cap \Pre_{\nu\Psi}(\rinv)$;
		we want to show $\Psi(S)\cap S\in \Pre_{\nu\Psi}(\pst\lor\rinv)$. 
		By a reason similar to \cref{rmk:first-condition-is-redundant}, it suffices to show
        \[
        \pst(\Psi(S)\cap S)\cup \rinv(\Psi(S)\cap S)\subseteq \Psi(S)\cap S.
        \]
		
		We first show $\pst(\Psi(S)\cap S)\subseteq \Psi(S)\cap S$.
	Since $S$ is in $\Pre_{\nu\Psi}(\pst)$, we have $\pst(S)\subseteq S$, and hence $\pst(\Psi(S)\cap S)\subseteq S$. Now observe that  \cref{prop:pstclosed} implies $\Psi(S)\in \Pre_{\nu\Psi}(\pst)$. In particular, we have $\pst(\Psi(S))\subseteq\Psi(S)$.
	Thus we obtain $\pst(\Psi(S)\cap S)\subseteq\Psi(S)\cap S$.

		Next observe that $\rinv(\Psi(S)\cap S)\subseteq S$ follows from $S\in \Pre_{\nu\Psi}(\rinv)$. 
		
	It remains to show $\rinv(\Psi(S)\cap S)\subseteq\Psi(S)$.
	Suppose we are given an $n$-cell $v\colon y\to x$ in $\rinv(\Psi(S)\cap S)$. By the definition of $\rinv$, this means that we have 
	\begin{multline*}
		(u\colon x\to y)\in X_n,\ \ 
		(w\colon y\to x)\in X_n,
		\\
		\bigl({p}\colon u\comp{n-1}{}v\to \id{n}{}{x}\bigr)\in \Psi(S)\cap S\cap X_{n+1},\ \ \text{and}\ \ 
		\bigl({q}\colon w\comp{n-1}{}u\to \id{n}{}{y}\bigr)\in \Psi(S)\cap S\cap X_{n+1}.
	\end{multline*}
	Our aim is to show $v\in\Psi(S)$.
	Since we already have $p\in S$, it remains to construct $q'\colon v\comp{n-1}{} u \to \id{n}{}{x}$ in $S$.
	Thanks to $\pst(S)\subseteq S$, it suffices to find $r\colon v\to w$ in $S$, since we can then set
	$q'=(r\comp{n-1}{}u)\comp{n}{}q$. We define $r$ as the composite of the following string of $(n+1)$-cells in $X$.
	\[
		\begin{tikzcd}
			v
			\ar[r,"\text{(unit)}"]
			&
			{\id{n}{}{y}\comp{n-1}{}v}
			\ar[r,"(i)"]
			&
			{(w\comp{n-1}{}u)\comp{n-1}{}v}
			\ar[r,"\text{(assoc)}"]
			&
			{w\comp{n-1}{}(u\comp{n-1}{}v)}
			\ar[r,"(ii)"]
			&
			{w\comp{n-1}{}\id{n}{}{x}}
			\ar[r,"\text{(unit)}"]
			&
			w
		\end{tikzcd}
	\]
	Here, the cells labeled by $\text{(unit)}$ are spherical equivalences witnessing \cref{unit-law}, whereas the cell labeled by $\text{(assoc)}$ is a spherical equivalence witnessing \cref{associativity}.
	They are in $S$ because all spherical equivalences are flat equivalences and we have $\nu \Psi\subseteq S$ by the assumption.
		The cell  $(ii)$ is $w\comp{n-1}{}p$, which is in $S$ by $p\in S$ and $\pst(S)\subseteq S$.
		Finally, to define the cell $(i)$, observe that $q$ admits a right $S$-inverse (in the sense defined in the proof of \cref{prop:pstclosed}) $q_R\colon \id{n}{}{y}\to w\comp{n-1}{}u$ since $q\in \Psi(S)$.
		We have $q_R\in \rinv(S)\subseteq S$. 
		The cell $(i)$ is defined to be $q_{R}\comp{n-1}{}v$, which is in $S$ by $\pst(S)\subseteq S$. 
		Using $\pst(S)\subseteq S$ once again, we see that the composite $(n+1)$-cell $r$ is in $S$.
\end{proof}
We record the closure properties of the class of flat equivalences we have established. 
\begin{corollary}
	\label{cor:bi-inv-closed-under-pst-and-linv}
	We have $\pst(\nu\Psi)\subseteq \nu\Psi$ and $\rinv(\nu\Psi)\subseteq \nu\Psi$.
\end{corollary}
\begin{proof}
	By \cref{cor:nuF-is-G-closed-F-wedge-id,prop:linvclosed}.
\end{proof}

\subsection{Flat equivalences are spherical equivalences}
We are now ready to show that the two definitions of equivalence cells agree.
\begin{proposition}
	\label{prop:bi-inv-are-inv}
	A cell in a weak $\omega$-category is a flat equivalence if and only if it is a spherical equivalence.
\end{proposition}
\begin{proof}
	It suffices to show that all flat equivalences in a weak $\omega$-category $X$ are spherical equivalences. For this, it suffices to show $\nu\Psi\subseteq\Phi(\nu\Psi)$. 
	Suppose that $u\colon x\to y$ is a flat equivalence $n$-cell in $X$. Thus we have $n$-cells $v,w\colon y\to x$ and flat equivalence $(n+1)$-cells $p\colon u\comp{n-1}{}v\to \id{n}{}{x}$ and $q\colon w\comp{n-1}{}u\to \id{n}{}{y}$ in $X$. 
	We now construct a flat equivalence $(n+1)$-cell $v\comp{n-1}{}u\to \id{n}{}{y}$. 
	The $(n+1)$-cell $r\colon v\to w$ constructed as in the proof of \cref{prop:linvclosed} is a flat equivalence by \cref{cor:bi-inv-closed-under-pst-and-linv}. Hence so is $(r\comp{n-1}{}u)\comp{n}{}q\colon v\comp{n-1}{}u\to \id{n}{}{y}$, again by \cref{cor:bi-inv-closed-under-pst-and-linv}. Therefore we have $u\in \Phi(\nu\Psi)$ as desired. 
\end{proof}

Having now shown that spherical equivalences and flat equivalences coincide, we may call them simply \emph{equivalences} in what follows. 

Finally we define the notion of \emph{inverse} of an equivalence cell.

\begin{proposition}
	\label{prop:right-inv-iff-left-inv}
	Let $X$ be a weak $\omega$-category, $n\geq 1$, and $u\colon x\to y$ be an equivalence $n$-cell in $X$.
	Then, for any $n$-cell $v\colon y\to x$ in $X$, the following conditions are equivalent.
	\begin{enumerate}
		\item We have $u\comp{n-1}{}v\sim\id{n}{}{x}$ in $X$.
		\item We have $v\comp{n-1}{}u\sim\id{n}{}{y}$ in $X$.
	\end{enumerate}
	Moreover, these conditions imply that $v$ is an equivalence in $X$.
\end{proposition}
\begin{proof}
	By \cref{def:spherical-eq}, there exists an $n$-cell $v'\colon y\to x$ with $u\comp{n-1}{}v'\sim\id{n}{}{x}$ and $v'\comp{n-1}{}u\sim\id{n}{}{y}$. 
	Given $u\comp{n-1}{}v\sim\id{n}{}{x}$, we can derive $v\comp{n-1}{}u\sim \id{n}{}{y}$ as follows.
	\begin{flalign*}
		&&v\comp{n-1}{}u 
		&\sim \id{n}{}{y}\comp{n-1}{}(v\comp{n-1}{}u) &\text{(by \cref{unit-law})}\\
		&&&\sim (v'\comp{n-1}{}u)\comp{n-1}{}(v\comp{n-1}{}u)&\text{(by \cref{sim-is-congruence})} \\
		&&&\sim v'\comp{n-1}{}\bigl((u\comp{n-1}{}v)\comp{n-1}{}u\bigr)& \text{(by \cref{associativity,sim-is-congruence,sim-is-eq-rel})}\\
		&&&\sim v'\comp{n-1}{}\bigl(\id{n}{}{x}\comp{n-1}{}u\bigr)& \text{(by \cref{sim-is-congruence})} \\
		&&&\sim v'\comp{n-1}{}u& \text{(by \cref{unit-law,sim-is-congruence})}\\
		&&&\sim \id{n}{}{y}.& 
	\end{flalign*}
	The converse direction is similar.

	The second statement is immediate from \cref{def:spherical-eq}.
\end{proof}
\begin{definition}
	Let $X$ be a weak $\omega$-category and $u\colon x\to y$ be an equivalence $n$-cell in $X$, with $n\geq 1$. An $n$-cell $v\colon y\to x$ in $X$ is called an \emph{inverse} of $u$ if the equivalent conditions of \cref{prop:right-inv-iff-left-inv} are satisfied.
\end{definition}

\begin{proposition}[\textnormal{\cite[Corollary~3.3.16]{FHM1}}]
	\label{prop:uniqueness-of-inverse}
	Let $X$ be a weak $\omega$-category, $n\geq 1$, $u\colon x\to y$ be an equivalence $n$-cell in $X$, and $v,v'\colon y\to x$ be inverses of $u$. Then we have $v\sim v'$.
\end{proposition}

\section{\texorpdfstring{$\omega$-equifibrations}{ω-equifibrations}}\label{sec:omega-equifibrations}
In this section, we define the class of $\omega$-equifibrations between weak $\omega$-categories and show that it can be characterised by the right lifting property with respect to a set $J$ of morphisms in $\WkCats{\omega}$. 
Whereas formally we will deal with weak $\omega$-categories, 
the corresponding result for strict $\omega$-categories can be obtained by replacing ``weak'' by ``strict'' (and $\WkCats{\omega}$ by $\StrCats\omega$) throughout. (Alternatively, one can apply the strict $\omega$-categorical reflection functor $\WkCats{\omega}\to\StrCats\omega$ to $J$; see \cref{reflecting-En}.)

\subsection{\texorpdfstring{The definition of $\omega$-equifibrations}{The definition of ω-equifibrations}}

The following definition generalises the notion of isofibration between categories to the context of weak $\omega$-categories.
Analogous classes of fibrations between 2-categories \cite{Lack-2-cat}, bicategories \cite{Lack-bicat}, and Gray-categories \cite{Lack-Gray} have been investigated by Lack.

\begin{definition}
	\label{def:equifibration}
	We say that a strict $\omega$-functor $f \colon X \to Y$ between weak $\omega$-categories is an \emph{$\omega$-equifibration}
	if for each $n\geq 1$, $(n-1)$-cell $x$ in $X$, and equivalence $n$-cell $u \colon fx \to y$ in $Y$, there exists an equivalence $n$-cell $\bar u \colon x \to \bar y$ in $X$ such that $f\bar u = u$.
\end{definition}

Although it may seem arbitrary that we consider equivalences $u$ with domain (rather than codomain) $fx$ in the above definition, this asymmetry is only superficial.

\begin{proposition}
\label{prop:fibration-symmetric}
	Let $f\colon X\to Y$ be an $\omega$-equifibration between weak $\omega$-categories.
    Then for each $n\geq 1$, $(n-1)$-cell $x$ in $X$, and equivalence $n$-cell $v\colon y\to fx$ in $Y$, there exists an equivalence $n$-cell $\bar v\colon \bar y\to x$ in $X$ such that $f\bar v=v$.
\end{proposition}
\begin{proof}
	Let $u\colon fx \to y$ be an inverse of $v$.
	Since $u$ is also an equivalence, there exists an equivalence $n$-cell $\bar u\colon x\to\bar y$ in $X$ such that $f\bar u=u$.
	Take an inverse $\bar v'$ of $\bar u$ in $X$. Since $f$ preserves equivalence $(n+1)$-cells by \cref{str-functor-pres-eq}, it sends inverses of $\bar u$ to inverses of $f\bar u=u$.
	Therefore, $f\bar v'$ and $v$ are both inverses of $u$.
	By \cref{prop:uniqueness-of-inverse}, there exists an equivalence $(n+1)$-cell $p\colon f\bar v'\to v$ in $Y$,
	which we can lift to an equivalence $(n+1)$-cell $\bar p\colon \bar v' \to \bar v$ in $X$. In particular, we have $f\bar v=v$.
	Thanks to \cref{prop:invariance-of-equivalence}, $\bar v$ is an equivalence in $X$ since $\bar v'$ is. 
\end{proof}

As a sanity check, we show the expected identity 
\[
\{\,\text{trivial fibrations}\,\}=\{\,\text{weak equivalences}\,\}\cap\{\,\text{fibrations}\,\}.
\]
We first need to recall the notions of trivial fibration and ($\omega$-)weak equivalence.

\begin{definition}\label{def:trivial-fibration}
     A \emph{trivial fibration} $f \colon X \to Y$ between weak $\omega$-categories is a strict $\omega$-functor that has the right lifting property against $\iota^n\colon \partial\C n\to \C n$ (see \cref{subsec:strict-and-weak-omega-cats}) for each $n\geq 0$.
     Equivalently, $f$ is a trivial fibration if
    \begin{itemize}
        \item for each $0$-cell $y$ in $Y$, there exists a $0$-cell $x$ in $X$ such that $fx = y$, and
        \item for each $n \ge 1$, parallel pair of $(n-1)$-cells $x,x'$ in $X$, and $n$-cell $u \colon fx \to fx'$ in $Y$, there exists an $n$-cell $\bar u \colon x \to x'$ such that $f\bar u = u$.\qedhere
    \end{itemize}
\end{definition}

\begin{definition}[{\cite[Definition~3.1.2]{FHM2}, see also \cite[Definition~4.7]{Lafont_Metayer_Worytkiewicz_folk_model_str_omega_cat} for the strict case}]\label{def:weak-equivalence}
    We say that a strict $\omega$-functor $f \colon X \to Y$ between weak $\omega$-categories is an \emph{$\omega$-weak equivalence} if
    \begin{itemize}
        \item for each $0$-cell $y$ in $Y$, there exists a $0$-cell $x$ in $X$ such that $fx \sim y$, and
        \item for each $n \ge 1$, parallel pair of $(n-1)$-cells $x,x'$ in $X$, and $n$-cell $u \colon fx \to fx'$ in $Y$, there exists an $n$-cell $\bar u \colon x \to x'$ such that $f\bar u \sim u$.\qedhere
    \end{itemize}
\end{definition}

\begin{proposition}\label{triv-fib-iff-wk-eq-and-fib}
	For a strict $\omega$-functor $f\colon X\to Y$ between weak $\omega$-categories, the following conditions are equivalent.
	\begin{itemize}
		\item[(1)] $f$ is a trivial fibration.
		\item[(2)] $f$ is both an $\omega$-weak equivalence and an $\omega$-equifibration.
	\end{itemize}
\end{proposition}
\begin{proof}
	If $f$ is a trivial fibration, then it is clearly an $\omega$-weak equivalence. In particular, by \cite[Proposition~3.2.2]{FHM1} or \cite[Proposition~3.1.7]{FHM2}, $f$ reflects equivalence cells. Using this, we can easily see that $f$ is an $\omega$-equifibration. 
	
    Conversely, suppose that $f$ is both an $\omega$-weak equivalence and an $\omega$-equifibration.
    We fix $n\geq 1$ and prove that $f$ has the right lifting property with respect to $\iota^n\colon\partial \C{n}\to \C{n}$; the case $n=0$ can be proven similarly.
    Let $x,x'$ be a parallel pair of $(n-1)$-cells in $X$ and let $u\colon fx\to fx'$ be an $n$-cell in $Y$.
    Then, since $f$ is an $\omega$-weak equivalence, there exist an $n$-cell $\overline u' \colon x\to x'$ in $X$ and an equivalence $(n+1)$-cell $v\colon f\overline u' \to u$ in $Y$.
    Moreover, since $f$ is an $\omega$-equifibration, there exists an equivalence $(n+1)$-cell $\overline v\colon \overline u' \to \overline u$ in $X$ with $f\overline v=v$.
    In particular, we have $f\overline u=u$.
    This completes the proof.
\end{proof}

\subsection{Marked weak \texorpdfstring{$\omega$}{ω}-categories}
\label{subsec:marked-weak-omega-cats}
In the introduction, we argued that the $\omega$-equifibrations are difficult to characterise via the right lifting property because their definition refers to the \emph{property}, rather than a \emph{structure}, of certain cells being equivalences.
It turns out that such a characterisation is easy to obtain if one can encode properties of cells as extra structure.
This is our motivation for introducing the following notion.

\begin{definition}
	A \emph{marked weak $\omega$-category} $(X,tX)$ is a weak $\omega$-category $X$ equipped with distinguished subsets $tX_n \subseteq X_n$ of \emph{marked} cells for each $n \ge 1$.
	A strict $\omega$-functor $f\colon X\to Y$ between the underlying weak $\omega$-categories of marked weak $\omega$-categories $(X,tX)$ and $(Y,tY)$ is \emph{marking-preserving} if for each $n\ge 1$ and each $x\in tX_n$, we have $fx\in tY_n$ 
\end{definition}
Let $\mWkCats{\omega}$ be the category of marked weak $\omega$-categories and marking-preserving strict $\omega$-functors between them. 
The forgetful functor $\mWkCats{\omega}\to \WkCats{\omega}$ mapping $(X,tX)$ to $X$ has both left and right adjoints. In particular, the left adjoint $(-)^\flat\colon \WkCats{\omega}\to \mWkCats{\omega}$ maps each weak $\omega$-category $X$ to $X^\flat=(X,\emptyset)$. 
(The right adjoint $(-)^\sharp\colon \WkCats{\omega}\to \mWkCats{\omega}$ maps $X$ to $X^\sharp=(X,\coprod_{n\geq 1}X_n)$, but we shall not use this.)
Thanks to \cref{str-functor-pres-eq}, we have another functor $(-)^\natural\colon \WkCats{\omega}\to \mWkCats{\omega}$ mapping each weak $\omega$-category $X$ to the marked weak $\omega$-category $X^\natural=(X,eX)$, where $eX$ is the set of all equivalence cells in $X$.

\begin{remark}
    In the above definition of marked weak $\omega$-category $(X,tX)$, we do not require the set $tX$ to contain the identity cells nor be closed under compositions.
    Thus, even if the underlying weak $\omega$-category $X$ happens to be a strict $\omega$-category, the pair $(X,tX)$ may not be an $\infty$-marked $\infty$-category in the sense of \cite[Definition 2.15]{Henry_Loubaton_inductive}.
\end{remark}

Recall from \cref{subsec:strict-and-weak-omega-cats} that $\C{n}\in \WkCats{\omega}$ is the weak $\omega$-category freely generated by the representable globular set $G^n$ for each $n\ge 0$.

\begin{definition}
	For each $n\ge 1$, we write $\mC{n}\in \mWkCats{\omega}$ for the marked weak $\omega$-category $(\C{n},\{c_n\})$ where $c_n$ is the image of the unique $n$-cell in $G^n$ under the inclusion $G^n\to \C{n}$.
    We refer to $c_n$ as the \emph{fundamental $n$-cell} of $\C{n}$.
\end{definition}
For each $n\ge 1$, the strict $\omega$-functor $\sigma^{n-1}\colon \C{n-1}\to \C n$ as in \cref{eqn:sigma-sigma-prime-cat} gives rise to the marking-preserving strict $\omega$-functor
$\sigma^{n-1}\colon (\C{n-1})^\flat\to \mC{n}$.
The following is the easy characterisation of $\omega$-equifibrations alluded to at the beginning of this subsection.

\begin{proposition}
	\label{prop:equifibrations-via-marked-RLP}
	A strict $\omega$-functor $f\colon X\to Y$ between weak $\omega$-categories is an $\omega$-equifibration if and only if the morphism $f^\natural \colon X^\natural \to Y^\natural$ in $\mWkCats{\omega}$ has the right lifting property with respect to the morphism  $\sigma^{n-1}\colon (\C{n-1})^\flat\to\mC{n}$ for each $n\ge 1$.
\end{proposition}

We conclude this subsection by establishing the following result, which will be used later when we perform the small object argument in $\mWkCats{\omega}$.

\begin{proposition}
	\label{mwkcats-lfp}
	The category $\mWkCats{\omega}$ is locally finitely presentable.
\end{proposition}
\begin{proof}
	Observe that $\mWkCats{\omega}$ fits into the following pullback square:
	\[
		\begin{tikzcd}[column sep=large]
			\mWkCats\omega
			\ar[d]
			\ar[rr]
			\ar[drr,phantom,"\lrcorner"very near start]
				&
					&
					\mathbf{Sub}(\Set^\N)
					\ar[d]
			\\
			\WkCats\omega
			\ar[r]
				&
				\GSet
				\ar[r,"{X\mapsto(n\mapsto X_{n+1})}"']
					&
					\Set^{\N}
		\end{tikzcd}
	\]
    (The functor from $\GSet$ to $\Set^\N$ forgets $X_0$; recall that, by definition, marked cells in a marked weak $\omega$-category have positive dimension.)
	Here, the right vertical functor is the subobject fibration of $\Set^\N$, which is in particular an isofibration.
	Therefore the pullback square is a bipullback square \cite{Joyal-Street-bipullback}.
        By \cref{Wk-omega-Cat-LFP},
	all functors in the cospan defining the pullback are finitely accessible right adjoint functors between locally finitely presentable categories.
	This concludes that $\mWkCats\omega$ is locally finitely presentable by \cite[Theorem~2.17]{Bird-thesis}.
\end{proof}

\subsection{Characterising \texorpdfstring{$\omega$}{ω}-equifibrations via the right lifting property in \texorpdfstring{$\WkCats{\omega}$}{Wk-ω-Cats}}
\label{subsec:omega-equifib-via-RLP}
In this subsection, we turn \cref{prop:equifibrations-via-marked-RLP} into a characterisation of $\omega$-equifibrations via the right lifting property in $\WkCats{\omega}$ (rather than $\mWkCats{\omega}$) assuming the existence of a (small) set $K$ of morphisms in $\mWkCats{\omega}$ satisfying conditions \ref{K1}--\ref{K3} stated below.
We will present an example of such a set $K$ in \cref{subsec:example-of-K}, and another example in \cref{appendix-on-halfadjoint-lift-one-step}. 

First we introduce the notation used in the statements of the conditions. 
Let $K$ be a class of morphisms in a category $\mathbf{C}$ with a terminal object $1$. 
\begin{itemize}
	\item We write $K^\pitchfork$ (resp.~$^\pitchfork K$) for the class of all morphisms in $\mathbf{C}$ having the right (resp.~left) lifting property with respect to $K$.
	\item We write $\cof{K}$ for the class $^\pitchfork(K^\pitchfork)$.
	\item We say that an object $X$ in $\mathbf{C}$ is \emph{$K$-injective} if the unique morphism $X\to 1$ is in $K^\pitchfork$, and define $\Inj{K}$ to be the full subcategory of $\mathbf{C}$ consisting of all $K$-injective objects. 
\end{itemize}

Now let $K$ be a set of morphisms in $\mWkCats{\omega}$. We consider the following conditions on $K$.
\begin{itemize}
	\labeleditem{(K1)}\label{K1} For any $\omega$-equifibration $f\colon X\to Y$ between weak $\omega$-categories, the morphism $f^\natural\colon X^\natural \to Y^\natural$ in $\mWkCats{\omega}$ is in $K^\pitchfork$. 
\end{itemize}
Note that the unique strict $\omega$-functor $X\to 1$ to the terminal weak $\omega$-category $1$ is an $\omega$-equifibration for any weak $\omega$-category $X$, and $1^\natural=1^\sharp$ is the terminal object in $\mWkCats{\omega}$.
Thus, if \ref{K1} holds, then the functor $(-)^\natural\colon \WkCats{\omega}\to\mWkCats{\omega}$ factors through the full subcategory $\Inj{K}$ of $\mWkCats{\omega}$. The following condition \ref{K2} is stated under this assumption.
\begin{itemize}
	\labeleditem{(K2)}\label{K2} The functor $(-)^\natural\colon \WkCats{\omega}\to\Inj{K}$ is right adjoint to the forgetful functor $\Inj{K}\to \WkCats{\omega}$.
	\labeleditem{(K3)}\label{K3} For each $n \ge 0$, the object $(\C{n})^\flat$ in $\mWkCats{\omega}$ is in $\Inj{K}$.
\end{itemize}

For the rest of this subsection, we fix a (small) set $K$ of morphisms in $\mWkCats{\omega}$ satisfying the conditions \ref{K1}--\ref{K3}. 
For each $n\geq 1$, we further fix an object $\E{n}\in \Inj{K}$ and a morphism 
$\igen{n}\colon \mC{n}\to  \E{n}$ in $\cof{K}$. 
Such data can be obtained, for example, by applying the 
small object argument with respect to $K$ to the unique morphism $\mC{n}\to 1$ in $\mWkCats{\omega}$, where $1$ is the terminal object of $\mWkCats{\omega}$.
(\cref{mwkcats-lfp} ensures that we may indeed apply the small object argument in this category.)
Define a set $J$ of morphisms in $\WkCats{\omega}$ by 
\[
J=\{\,\C{n-1}\xrightarrow{\sigma^{n-1}}\C{n}\xrightarrow{\igen{n}}\E{n}\mid n\geq 1\,\}.
\]

\begin{theorem}
	\label{RLP-iff-fibration}
	Let $f\colon X\to Y$ be a strict $\omega$-functor between weak $\omega$-categories. Then $f$ is an $\omega$-equifibration if and only if $f$ has the right lifting property with respect to $J$. 
\end{theorem}
\begin{proof}
	We show that the following three conditions on $f$ are equivalent.
	\begin{itemize}
		\item[(a)] The morphism $f^\natural \colon X^\natural \to Y^\natural$ in $\mWkCats{\omega}$ has the right lifting property (in $\mWkCats{\omega}$) with respect to $\sigma^{n-1}\colon (\C{n-1})^\flat\to \mC{n}$ for each $n\geq 1$. 
		\item[(b)] The morphism $f^\natural \colon X^\natural \to Y^\natural$ in $\mWkCats{\omega}$ has the right lifting property (in $\mWkCats{\omega}$) with respect to $\igen{n}\sigma^{n-1}\colon (\C{n-1})^\flat\to \E{n}$ for each $n\geq 1$.
		\item[(c)] The morphism $f \colon X\to Y$ in $\WkCats{\omega}$ has the right lifting property (in $\WkCats{\omega}$) with respect to $\igen{n}\sigma^{n-1}\colon \C{n-1}\to \E{n}$ for each $n\geq 1$.
	\end{itemize}
	Notice that this suffices because $f$ is an $\omega$-equifibration if and only if $f$ satisfies (a) by \cref{prop:equifibrations-via-marked-RLP}. 
	The equivalence of (b) and (c) follows from \ref{K2}, because the object $(\C{n-1})^\flat$ is in $\Inj{K}$ by \ref{K3}, whereas the object $\E{n}$ is in $\Inj{K}$ by construction. Therefore it remains to show the equivalence of (a) and (b).

	First suppose that (a) holds. Take any $n\geq 1$ and any  commutative diagram 
	\[
	\begin{tikzcd}[row sep = small]
		(\C{n-1})^\flat
		\arrow [r, "x"]
		\arrow [d, "\sigma^{n-1}", swap]
			&
			X^\natural
			\arrow [dd, "f^\natural"]
		\\
		\mC{n}
		\ar[d, "\igen{n}"']
			&
		\\
		\E{n}
		\arrow [r,"u"']
			&
			Y^\natural
	\end{tikzcd}
	\]
	in $\mWkCats{\omega}$. By (a), we obtain a morphism $v\colon \mC{n}\to X^\natural$ making the diagram 
	\[
	\begin{tikzcd}[row sep = small]
		(\C{n-1})^\flat
		\arrow [r, "x"]
		\arrow [d, "\sigma^{n-1}", swap]
			&
			X^\natural
			\arrow [dd, "f^\natural"]
		\\
		\mC{n}
		\ar[d, "\igen{n}"']
		\ar[ur, dotted, "v"]
			&
		\\
		\E{n}
		\arrow [r,"u"']
			&
			Y^\natural
	\end{tikzcd}
	\]
	commute. By $\igen{n}\in \cof{K}$ and $f^\natural\in K^\pitchfork$ (\ref{K1}), we obtain a morphism $w\colon \E{n}\to X^\natural$ making the diagram 
	\[
	\begin{tikzcd}[row sep = small]
		(\C{n-1})^\flat
		\arrow [r, "x"]
		\arrow [d, "\sigma^{n-1}", swap]
			&
			X^\natural
			\arrow [dd, "f^\natural"]
		\\
		\mC{n}
		\ar[d, "\igen{n}"']
		\ar[ur, "v"]
			&
		\\
		\E{n}
		\arrow [r,"u"']
		\ar[uur, dotted, "w"']
			&
			Y^\natural
	\end{tikzcd}
	\]
	commute. Therefore (b) holds.

	Now suppose that (b) holds and take any $n\geq 1$ and any commutative square 
	\begin{equation}
	\label{eqn:s-fnatural}
		\begin{tikzcd}
			(\C{n-1})^\flat
			\ar[r,"x"]
			\ar[d,"\sigma^{n-1}"']
				&
				X^\natural
				\ar[d,"f^\natural"]
			\\
			\mC{n}
			\ar[r,"u"]
				&
				Y^\natural
		\end{tikzcd}
	\end{equation}
	in $\mWkCats{\omega}$.
	By $\igen{n}\in \cof{K}$ and $Y^\natural\in \Inj{K}$, there exists a morphism $v$ in $\mWkCats{\omega}$ making the following triangle commute.
	\[
		\begin{tikzcd}
			\mC{n}
			\ar[r,"u"]
			\ar[d,"\igen{n}"']
				&
				Y^\natural
			\\
			\E{n}
			\ar[ru,dotted,"v"']
				&
		\end{tikzcd}
	\]
	Thus we obtain the commutative square
	\[
		\begin{tikzcd}
			(\C{n-1})^\flat
			\ar[r,"x"]
			\ar[d,"\igen{n}\sigma^{n-1}"']
				&
				X^\natural
				\ar[d,"f^\natural"]
			\\
			\E{n}
			\ar[r,"v"]
				&
				Y^\natural
		\end{tikzcd}
	\]
	in $\mWkCats{\omega}$, which admits a diagonal filler $w\colon \E{n}\to X^\natural$ by (b). 
	The morphism $w\igen{n}\colon \mC{n}\to X^\natural$ provides a diagonal filler for \cref{eqn:s-fnatural}, and hence (a) holds.
\end{proof}

\subsection{An example of \texorpdfstring{$K$}{K}}
\label{subsec:example-of-K}
In this subsection, we construct a set $K$ of morphisms in $\mWkCats{\omega}$ satisfying conditions \ref{K1}--\ref{K3}.
\begin{definition}
	\label{def:Fn}
	For each $n \ge 1$, we define a marked weak $\omega$-category $(\FF^n,t\FF^n)$.
    (The letter F stands for ``flat (equivalence)''.)
	The underlying weak $\omega$-category $\FF^n$ is constructed as follows:
	\begin{itemize}
		\item start with $\C{n}$, and call its fundamental $n$-cell $\uF \colon \xF \to \yF$,
		\item freely adjoin $n$-cells $\vF \colon \yF \to \xF$ and $\wF \colon \yF \to \xF$, that is,
		take the colimit of the solid part of the following diagram in $\WkCats{\omega}$:
		\[
		\begin{tikzcd}
			\partial\C{n}
			\ar[r,"\ppair{\yF,\xF}"]
			\ar[d,hook]
				&
				\C{n}
				\ar[d,dashed,"\uF"]
					&
					\partial\C{n}
					\ar[l,"\ppair{\yF,\xF}"']
					\ar[d,hook]
			\\
			\C{n}
			\ar[r,dashed,"\vF"]
				&
				\FFp^n
					&
					\C{n}
					\ar[l,dashed,"\wF"']
		\end{tikzcd}
		\]
        and
		\item freely adjoin $(n+1)$-cells $\pF \colon \uF\comp{n-1}{}\vF\to  \id{n}{}{\xF}$ and $\qF\colon \wF \comp{n-1}{} \uF \to \id{n}{}{\yF}$.
	\end{itemize}
	The marked cells are $\uF$, $\pF$, and $\qF$.
	Notice that there exists a unique marking-preserving strict $\omega$-functor $\kF^n\colon \mC{n}\to \FF^n$, which maps the fundamental $n$-cell $c_n$ of $\mC{n}$ to $\uF$.
\end{definition}

We claim that the set 
$\KF=\{\kF^n\colon \mC{n}\to\FF^n\,|\,n\ge1\}$
satisfies conditions \ref{K1}--\ref{K3}.
We first verify \ref{K1}.

\begin{proposition}
    \label{K-satisfies-K1-flat}
	For any $\omega$-equifibration $f \colon X \to Z$ between weak $\omega$-categories,
	the marking-preserving strict $\omega$-functor $f^\natural\colon X^\natural\to Z^\natural$ is in $\KF^\pitchfork$.
\end{proposition}

\begin{proof}
	Consider a commutative square in $\mWkCats{\omega}$, with $n\geq 1$:
	\[
	\begin{tikzcd}[row sep = large]
		\mC{n}
		\arrow [d,"\kF^n"']
		\arrow [r,"u"] &
		X^\natural
		\arrow [d, "f^\natural"] \\
		\FF^n
		\arrow [r] &
		Z^\natural
	\end{tikzcd}
	\]
	Such data are equivalent to:
	\begin{itemize}
		\item an equivalence $n$-cell $u \colon x \to y$ in $X$,
		\item $n$-cells $v \colon fy \to fx$ and $w\colon fy\to fx$ in $Z$, and
		\item equivalence $(n+1)$-cells $p \colon fu\comp{n-1}{Z}v\to  \id{n}{Z}{fx}$ and $q \colon w \comp{n-1}{Z} fu \to \id{n}{Z}{fy}$ in $Z$.
	\end{itemize}
	We wish to extend $u$ to a quintuple $(u,\bar v, \bar w, \bar p, \bar q)$ of cells of suitable types in $X$
	(so that it corresponds to a marking-preserving map $\FF^n \to X^\natural$)
	such that $f$ sends it to $(fu,v,w,p,q)$. Observe that, since $u\colon x\to y$ is an equivalence in $X$, we can choose 
	\begin{itemize}
		\item $n$-cells $\bar v' \colon y \to x$ and $\bar w'\colon y\to x$ in $X$, and
		\item equivalence $(n+1)$-cells $\bar p' \colon u\comp{n-1}{X}\bar v'\to  \id{n}{X}{x}$ and $\bar q' \colon \bar w' \comp{n-1}{X} u \to \id{n}{X}{y}$ in $X$.
	\end{itemize}
    
    Of course there is no guarantee that $f \bar v' = v$ (or any of the analogous equations for the other cells) holds, which is why we are calling it $\bar v'$ rather than $\bar v$.
    Nevertheless, we can rectify this defect and turn it into a proper lift $\bar v$ as follows; part of the following argument is visualised in \cref{lifting-v}.
    
	Since $f\bar p'\colon fu\comp{n-1}{Z}f\bar v'\to \id{n}{Z}{fx}$ (which indeed has this type by \cref{str-fun-pres-comp-and-id}) is an equivalence in $Z$ by \cref{str-functor-pres-eq}, (the dual of) \cref{whiskering-ess-0-surj} (1) implies that there exists an equivalence $(n+1)$-cell
	\[
	v^\dag \colon fu \comp{n-1}{Z} v \to fu \comp{n-1}{Z} f \bar v'
	\]
	in $Z$ satisfying
    \begin{equation}\label{v-dag}
        v^\dag \comp{n}{Z} f \bar p' \sim p.
    \end{equation}
	Similarly, since $fu\colon fx\to fy$ is an equivalence in $Z$, \cref{whiskering-ess-0-surj} (2) implies that there exists an equivalence $(n+1)$-cell
	\[
	v^\ddag \colon v\to f \bar v'
	\]
	in $Z$ satisfying
    \begin{equation}\label{v-ddag}
        fu \comp{n-1}{Z} v^\ddag \sim v^\dag.
    \end{equation}
	Since $f$ is an $\omega$-equifibration, \cref{prop:fibration-symmetric} allows us to lift $v^\ddag$ to an equivalence $(n+1)$-cell in $X$, which we denote as $\bar v^\ddag \colon \bar v \to \bar v'$.
	This completes the lifting of $v$.
	
	\begin{figure}
	$\begin{tikzpicture}[baseline = 40]
		\node (1ul) at (0,3) {$fx$};
		\node (1ur) at (3,3) {$fy$};
		\node (1lr) at (3,0) {$fx$};

		\draw[->] (1ul) to node [labelsize, auto] {$fu$} (1ur);
		\draw[->] (1ul) to node [labelsize, auto, swap] {$\id{n}{Z}{fx}$} (1lr);
		\draw[->, bend right = 30] (1ur) to node [labelsize, auto, swap] {$f\bar v'$} (1lr);
		\draw[->, bend left = 30] (1ur) to node [labelsize, auto] {$v$} (1lr);
		\draw[2cell] (2.2,2.5) to node [labelsize,swap, auto] {$f\bar p'$} (1.7,2);
		\draw[2cell, dashed] (3.3,1.5) to node [labelsize, auto,swap] {$\exists v^\ddag$} node [labelsize, auto] {$\sim$} (2.7,1.5);
	\end{tikzpicture}
	\qquad\sim\qquad
	\begin{tikzpicture}[baseline = 40]
		\node (1ul) at (0,3) {$fx$};
		\node (1ur) at (3,3) {$fy$};
		\node (1lr) at (3,0) {$fx$};

		\draw[->] (1ul) to node [labelsize, auto] {$fu$} (1ur);
		\draw[->] (1ul) to node [labelsize, auto, swap] {$\id{n}{Z}{fx}$} (1lr);
		\draw[->, bend left = 30] (1ur) to node [labelsize, auto] {$v$} (1lr);
		\draw[->, 2cell] (2.6,2.4) to node [labelsize, swap,auto] {$p$} (2.1,1.9);
	\end{tikzpicture}$
		\caption{Lifting \texorpdfstring{$v$}{v}}
		\label{lifting-v}
	\end{figure}

	Next, we rectify the $(n+1)$-cell $\bar p'$.
	Observe that we have an equivalence $(n+1)$-cell
	\[
	(u \comp{n-1}{X} \bar v^\ddag)\comp{n}{X} \bar p' \colon u \comp{n-1}{X} \bar v\to \id{n}{X}{x}
	\]
	in $X$, and moreover \eqref{v-ddag} and \eqref{v-dag} provide equivalence $(n+2)$-cells
	\[
	f\bigl((u \comp{n-1}{X} \bar v^\ddag)\comp{n}{X} \bar p'\bigr) =(fu \comp{n-1}{Z} v^\ddag)\comp{n}{Z} f\bar p' \sim v^\dag \comp{n}{Z} f \bar p' \sim p
	\]
	in $Z$.
	Since $f$ is an $\omega$-equifibration, we can lift the composite of the latter to an equivalence $(n+2)$-cell
	\[
	(u \comp{n-1}{X} \bar v^\ddag)\comp{n}{X} \bar p' \sim \bar p
	\]
	in $X$.
	The resulting $(n+1)$-cell $\bar p$ is an equivalence in $X$ by \cref{prop:invariance-of-equivalence}, and this completes the lifting of $p$.

    In a similar manner, one can lift $w$ and $q$. 
\end{proof}

Next we verify \ref{K2}. 

\begin{proposition}
	\label{injectivity-detects-equivalences}\leavevmode
    \begin{enumerate}
        \item If a marked weak $\omega$-category $(X,tX)$ is $\KF$-injective, then all cells in $tX$ are equivalences in the underlying weak $\omega$-category $X$.
        \item For any weak $\omega$-category $X$, the marked weak $\omega$-category $X^\natural$ is $\KF$-injective.
    \end{enumerate}
\end{proposition}
\begin{proof}
	For any marked weak $\omega$-category $(X,tX)$,  the containment $tX\subseteq\Psi(tX)$ holds if and only if $(X,tX)$ is $\KF$-injective, which implies both statements.
\end{proof}

Finally, \ref{K3} is obvious because, for each $n\ge 1$, $\mC{n}$ has a marked cell. In fact, for any weak $\omega$-category $X$, the object $X^\flat$ in $\mWkCats{\omega}$ is in $\Inj{\KF}$ simply because there is no morphism $\mC{n}\to X^\flat$ in $\mWkCats{\omega}$.

\section{Presentations of \texorpdfstring{$\E{n}$}{En}}\label{presentation-of-En}

Let $\KF=\{\kF^n\colon \mC{n}\to\FF^n\,|\,n\ge1\}$ be the set of morphisms in $\mWkCats{\omega}$ constructed in \cref{subsec:example-of-K}.
Then \cref{RLP-iff-fibration} provides a characterisation of $\omega$-equifibrations in terms of a family of $\KF$-injective marked weak $\omega$-categories $\E{n}$, but we said almost nothing about the structure of $\E{n}$ in \cref{subsec:omega-equifib-via-RLP} aside from pointing out that they can be obtained using the small object argument.
The aim of this section is to provide an explicit presentation of a certain universal model of $\E{n}$, which we denote by $\EF{n}$.

In fact, we provide \emph{two} presentations of $\EF{n}$.
The first presentation (given in \cref{def:EF-n-m,def:E-n-omega} and justified in \cref{presentation-1-justification}) essentially traces the steps of the algebraic small object argument outlined in \cite[Section 6]{Garner_understanding} (although we do not explicitly prove that the individual steps of our construction actually coincide with those in \cite{Garner_understanding}) whereas the second (given in \cref{def:system-of-witnesses-as-cocone,def:C-n-F} and justified in \cref{presentation-2-justification}) constructs everything in one go.
We will also deduce from the latter presentation that the strict $\omega$-categorical reflection of $\EF{1}$ is isomorphic to the \emph{coherent walking $\omega$-equivalence} $\ORst$ constructed in \cite{OR} (which was later shown in \cite{HLOR} to be indeed weakly equivalent to the terminal strict $\omega$-category in the folk model structure \cite{Lafont_Metayer_Worytkiewicz_folk_model_str_omega_cat}).

Let us first make precise the sense in which our $\EF{n}$ is universal.

\begin{definition}
    An \emph{algebraically $\KF$-injective object} is a pair $(X,\kappa^X)$ consisting of
    a marked weak $\omega$-category $X$ and a family $\kappa^X=\bigl(\kappa^X(n;x)\colon \FF^n\to X\bigr)_{n\geq1,x\in tX_n}$
    of marking-preserving strict $\omega$-functors such that, for each $n\geq1$ and each $x\in tX_n$, the following diagram commutes:
    \[
    	\begin{tikzcd}
    		\mC{n}
    		\ar[r,"c_n\mapsto x"]
    		\ar[d,"\kF^n"']
    			&
    			X
    		\\
    		\FF^n
    		\ar[ur,dotted,"\kappa^X(n;x)"']
    	\end{tikzcd}
    \]
    A morphism of algebraically $\KF$-injective objects $(X,\kappa^X)\to(Y,\kappa^Y)$ is a marking-preserving strict $\omega$-functor $f\colon X\to Y$ satisfying $f\circ\kappa^X(n;x)=\kappa^Y(n;fx)$.
    We write $\KF\mhyphen\AInj$ for the category of algebraically $\KF$-injective objects and morphisms between them.
\end{definition}

By virtue of \cref{mwkcats-lfp}, we can apply the algebraic small object argument
(see \cite[Theorem~4.4]{Garner_understanding} or \cite[Proposition~16]{Bourke_Garner_1}) to the set $\KF$,
and obtain the algebraic weak factorization system on $\mWkCats{\omega}$ cofibrantly generated by $\KF$ (in the sense of \cite[5.2]{Bourke_Garner_1}).
An algebraically $\KF$-injective object is precisely an \emph{algebraically fibrant object} in the sense of \cite[2.3]{Bourke_Garner_1}
with respect to this algebraic weak factorization system.
In particular, the forgetful functor $\KF\mhyphen\AInj\to\mWkCats{\omega}$ has a left adjoint, which gives \emph{algebraically $\KF$-injective replacements}
of marked weak $\omega$-categories.

\begin{definition}\label{def:E-F-n}
    For each $n \ge 1$, we write $(\EF{n},\kappa)$ for the algebraically $\KF$-injective replacement of $\mC{n}$, and $\iF{n} \colon \mC{n} \to \EF{n}$ for the unit map.
\end{definition}

\begin{lemma}
    For each $n \ge 1$, the unit map $\iF{n}\colon\mC{n}\to\EF{n}$ is in $\cof{\KF}$.
\end{lemma}
\begin{proof}
    The monad $R$ on $\mWkCats{\omega}$ induced by the forgetful functor $\KF\mhyphen\AInj\to\mWkCats{\omega}$ and its left adjoint is the dual of the \emph{cofibrant replacement comonad} in the sense of \cite[3.1]{Bourke_Garner_2},
    obtained by suitably restricting the monad $\mathsf{R}$ (on the $\mWkCats{\omega}^\mathbf{2}$) associated with the algebraic weak factorisation system $(\mathsf{L},\mathsf{R})$ cofibrantly generated by $\KF$.
    By the definition of algebraic weak factorisation system,
    each component of the unit of the monad $R$ lies in the  image of the comonad $\mathsf{L}$; see e.g.\ \cite[2.1]{Bourke_Garner_1}.
    In particular, it lies in the left class of the \emph{underlying weak factorisation system} (in the sense of \cite[2.6]{Bourke_Garner_1}) of $(\mathsf{L},\mathsf{R})$, which in this case coincides with $\cof{\KF}$.

    An alternative proof can be obtained at the end of \cref{subsec:presentation-1}: \cref{presentation-1-justification} shows that $i^n_\mathcal{F}=\lambda^0$, the latter being constructed from members of $K_\mathcal{F}$ by pushouts and transfinite composition.
\end{proof}

Combining \cref{RLP-iff-fibration} and the content of \cref{subsec:example-of-K}, we obtain the following.

\begin{proposition}
	\label{RLP-iff-fibration-version-F}
	Let $f\colon X\to Y$ be a strict $\omega$-functor between weak $\omega$-categories. Then $f$ is an $\omega$-equifibration if and only if $f$ has the right lifting property with respect to
    \[
        \JF=\{\,\C{n-1}\xrightarrow{\sigma^{n-1}}\C{n}\xrightarrow{\iF{n}}\EF{n}\mid n\geq 1\,\}.
    \]
\end{proposition}

Since we will be dealing with various colimits in $\mWkCats{\omega}$, let us record the following observation.

\begin{remark}
	\label{rem:colimits-in-m-WkCats-omega}
    The forgetful functor $\mWkCats\omega\to\WkCats\omega$ admits a right adjoint (see \cref{subsec:marked-weak-omega-cats}), and thus preserves colimits.
    In other words, for any small diagram $X\colon I\to\mWkCats\omega$,
    the underlying weak $\omega$-category of $\colim X$ may be computed in $\WkCats{\omega}$.
    The marking $t(\colim X)$ is given by
    $\bigcup_{i\in I}\lambda^i(tX(i))$
    where $\lambda^i\colon X(i)\to\colim X$ is the leg of the colimiting cocone at $i\in I$.
\end{remark}

\subsection{Presenting \texorpdfstring{$\EF{n}$}{EFn} by unravelling the algebraic small object argument}\label{subsec:presentation-1}
    Fix $n\geq1$.
    In this subsection, we present the marked weak $\omega$-category $\EF{n}$ as the colimit of an infinite sequence in the manner outlined in \cite[Section 6]{Garner_understanding}.
    This sequence comprises of the following family of marked weak $\omega$-categories.
    
\begin{definition}\label{def:EF-n-m}
	For each $m\in\N$, we define
	\begin{itemize}
		\item a marked weak $\omega$-category $\EF{n,m}=(\EF{n,m},t\EF{n,m})$, whose set of marked $(n+m)$-cells is denoted by $\marking{n,m} = (t\EF{n,m})_{n+m}$,
		\item a marking-preserving strict $\omega$-functor $e^{m+1}\colon\EF{n,m}\to\EF{n,{m+1}}$, and
		\item a marking-preserving strict $\omega$-functor $f^{m+1}\colon\coprod_{u\in \marking{n,m}}\FF^{n+m}\to\EF{n,{m+1}}$
	\end{itemize} 
	inductively as follows. 
	\begin{itemize}
		\item %
			Set $\EF{n,0}=\mC{n}$ (so that $\marking{n,0} = \{c_n\}$).
		\item %
			For $m \ge 0$, the marked weak $\omega$-category $\EF{n,{m+1}}$ is defined by the following pushout in $\mWkCats{\omega}$:
			\begin{equation}
				\label{E-n-m-pushout}
				\begin{tikzcd}[column sep = huge]
					\displaystyle\coprod_{u\in \marking{n,m}}\mC{n+m}
					\ar[r,"{(u,c_{n+m})\mapsto u}"]
					\ar[d]
					\ar[dr,phantom,"\ulcorner"very near end]
						&
						\EF{n,m}
						\ar[d,"{e^{m+1}}"]
					\\
					\displaystyle\coprod_{u\in \marking{n,m}}\FF^{n+m}
					\ar[r,"f^{m+1}"']
						&
						\EF{n,m+1}
				\end{tikzcd}
			\end{equation}
            Thus the restriction of $f^{m+1}$ to the $u$-th summand, which we denote by $f^{m+1}_u$, sends $\uF$ to $e^{m+1}(u)$. \qedhere
	\end{itemize}
\end{definition}

\begin{remark}\label{marked-cells-in-E-n-m}
By \cref{rem:colimits-in-m-WkCats-omega}, the underlying weak $\omega$-category of $\EF{n,m+1}$ for $m \ge 0$ can be obtained by computing the pushout \eqref{E-n-m-pushout} in $\WkCats{\omega}$, and a cell in $\EF{n,m+1}$ is marked if and only if it is the image of a marked cell under either $f^{m+1}$ or $e^{m+1}$.
In fact, one can easily check by induction on $m$ that
\begin{itemize}
    \item $\EF{n,m+1}$ has no marked $k$-cells for $k < n$,
    \item a $k$-cell in $\EF{n,m+1}$ with $n \le k < n+m$ is marked if and only if it is the image of a marked $k$-cell under $e^{m+1}$ if and only if it is the image of a marked $k$-cell under $e^{m+1} \circ \dots \circ e^{k-n+1}$,
    \item an $(n+m)$-cell in $\EF{n,m+1}$ is marked if and only if it is the image of a marked $(n+m)$-cell under $f^{m+1}$ if and only if it is the image of a marked $(n+m)$-cell under $e^{m+1}$,
    \item an $(n+m+1)$-cell in $\EF{n,m+1}$ is marked if and only if it is the image of a marked $(n+m+1)$-cell under $f^{m+1}$, and
    \item $\EF{n,m+1}$ has no marked $k$-cells for $k > n+m+1$. \qedhere
\end{itemize}
\end{remark}

The following lemma provides an even more explicit description of the marked cells in $\EF{n,m+1}$.

\begin{lemma}
	\label{lem:E-n-m-pushout-markings}
	For any $m\geq 0$ and $n\leq k\leq n+m$, the functions 
    \begin{equation}\label{eqn:maps-between-marked-cells}
        (te^{m+1})_k\colon (t\EF{n,m})_k\to (t\EF{n,m+1})_k \qquad \text{and}\qquad
    (tf^{m+1})_{n+m+1}\colon \coprod_{u\in \marking{n,m}}(t\FF^{n+m})_{n+m+1}\to (t\EF{n,m+1})_{n+m+1}
    \end{equation}
    (between sets of marked $k$- and $(n+m+1)$-cells)
    are bijective.
\end{lemma}
\begin{proof}
    First observe that the pushout square \eqref{E-n-m-pushout} induces the commutative square
    \begin{equation}\label{eqn:E-n-m+1-k}
		\begin{tikzcd}[column sep = huge]
				\displaystyle\coprod_{u\in \marking{n,m}}\bigl(t(\mC{n+m})\bigr)_{k}
				\ar[r]
				\ar[d]
						&
						(t\EF{n,m})_k
						\ar[d,"{(te^{m+1})_k}"]
				\\
				\displaystyle\coprod_{u\in \marking{n,m}}(t\FF^{n+m})_k
				\ar[r,"(tf^{m+1})_k"']
						&
						(t\EF{n,m+1})_k
		\end{tikzcd}
	\end{equation}
    of functions between sets of marked $k$-cells for each $k\geq 1$, in which the functions $(te^{m+1})_k$ and $(tf^{m+1})_k$ are \emph{jointly} surjective by \cref{rem:colimits-in-m-WkCats-omega} (or by \cref{marked-cells-in-E-n-m}).
    
    We begin with the proof of the surjectivity of the functions \eqref{eqn:maps-between-marked-cells}.
    When $k<n+m$, the function $(te^{m+1})_k$ is surjective because the set $\coprod_{u\in \marking{n,m}}(t\FF^{n+m})_k$ is empty. 
    Similarly, the function $(tf^{m+1})_{n+m+1}$ is surjective because the set $(t\EF{n,m})_{n+m+1}$ is empty. Finally, when $k=n+m$, both the top horizontal arrow and the left vertical arrow in \eqref{eqn:E-n-m+1-k} are bijections, and hence both $(te^{m+1})_{n+m}$ and $(tf^{m+1})_{n+m}$ are surjective.

    To show the injectivity of the functions \eqref{eqn:maps-between-marked-cells}, we introduce the following strict $\omega$-categories.
    
\begin{definition}
	For $k \ge 1$, we write $\cosep{k}$ for the following strict $\omega$-category.
	Its cells are given by
	\[
	\cosep{k}_l = \begin{cases}
		\{\star\} & l < k,\\
		\{w_k,\star\} & l=k,\\
		\{w_k,\star\}^2 & l > k,
	\end{cases}
	\]
	and the source and target operations are given by the identity except for
	\begin{gather*}
	s_{k-1}(w_k) = s_{k-1}(\star) = t_{k-1}(w_k) = t_{k-1}(\star) = \star,\\
	s_k = \pi_1 \quad \text{and} \quad t_k = \pi_2.
	\end{gather*}
    In other words, the underlying globular set of $\cosep{k}$ is the one that classifies ``sets of $k$-cells'', in the sense that for any globular set $X$, giving a globular map $g \colon X \to \cosep{k}$ is equivalent to giving a subset of $X_k$ corresponding to $g^{-1}(w_k)$. This follows from the fact that the underlying globular set of $\cosep{k}$ is obtained by applying the right adjoint of the ``evaluation at $k$'' functor $\GSet\to \Set$, mapping each $X\in\GSet$ to $X_k\in \Set$, to the two-element set.

    The strict $\omega$-category structure on $\cosep{k}$ is given as follows.
	Each $l$-cell shares its name with the identity $(l+1)$-cell on it except for
	\[
	\id{k+1}{}{w_k} = (w_k,w_k) \quad \text{and} \quad \id{k+1}{}{\star} = (\star, \star).
	\]
	Note that we have only identity cells in all positive dimensions but $k$ and $k+1$.
	The compositions of $k$-cells are given by
	\[
	x \comp{l}{} y = \begin{cases}
		\star & x=y=\star,\\
		w_k & \text{otherwise}
	\end{cases}
	\]
	for all $0 \le l < k$.
        The compositions of $(k+1)$-cells are given by
        \[
        (x_1,x_2) \comp{l}{} (y_1,y_2) = \begin{cases}
            (x_1 \comp{l}{} y_1, x_2 \comp{l}{} y_2) & 0 \le l < k,\\
            (x_1, y_2) & l=k \text{ (and hence $x_2 = y_1$).}
        \end{cases}
        \]
        The compositions of $l$-cells for $l>k$ are determined by these because all $l$-cells are identity cells.
\end{definition}
    Observe that the unique strict $\omega$-functor $\cosep{k}\to 1$ from $\cosep{k}$ to the terminal weak $\omega$-category $1$ is a trivial fibration in the sense of \cref{def:trivial-fibration}.
    This can be seen by transposing lifting problems along suitable adjunctions as follows. 
    Suppose that we are given the lifting problem in $\WkCats{\omega}$ as in the left below. 
    Using the adjunction $F^L\dashv U^L$ and the fact that $\iota_n\colon\partial\C{n}\to \C{n}$ is in the image of $F^L$, we see that this lifting problem is equivalent to the lifting problem in $\GSet$ as in the right below. 
    \[
		\begin{tikzcd}
			\partial\C{n}
			\ar[r]
			\ar[d,"\iota^n"']
				&
				\cosep{k}
				\ar[d]
			\\
			\C{n}
			\ar[r]
				&
				1
		\end{tikzcd}\qquad\qquad
        \begin{tikzcd}
			\partial \OO^{n}
			\ar[r]
			\ar[d,"\iota^n"']
				&
				U^L\cosep{k}
				\ar[d]
			\\
			\OO^{n}
			\ar[r]
				&
				1
		\end{tikzcd}
    \]
    Using the adjunction consisting of the ``evaluation at $k$'' functor $\GSet\to \Set$ and its right adjoint, and the fact that  $U^L\cosep{k}\to 1$ is in the image of this right adjoint, we see that the latter lifting problem is equivalent to a trivial lifting problem in $\Set$, and thus can be solved.

    \begin{claim}
        \label{distinguishing-a-in-tF}
		Let $m \ge 0$, and let $a,b$ be distinct marked $(n+m+1)$-cells in $\coprod_{u\in \marking{n,m}}\FF^{n+m}$.
		Then there exists a strict $\omega$-functor $g \colon \EF{n,m+1} \to \cosep{n+m+1}$ such that $gf^{m+1}(a) = w_{n+m+1}$ and $gf^{m+1}(b) = \star$.
    \end{claim}
    
    \begin{proof}[Proof of \cref{distinguishing-a-in-tF}]
	Observe that (the underlying weak $\omega$-category of) $\EF{n,m+1}$ can be written as the iterated pushout
	\[
	\begin{tikzcd}[column sep = huge]
		\displaystyle\coprod_{u\in \marking{n,m}}\C{n+m}
		\ar[r,"{(u,c_{n+m})\mapsto u}"]
		\ar[d]
		\ar[dr,phantom,"\ulcorner"very near end] &
		\EF{n,m}
		\ar[d] & \\
		X
		\ar[r] &
		Y
		\arrow [d] &
		\partial\C{n+m+1}
		\arrow [d]
		\arrow [l]
		\ar[dl,phantom,"\urcorner"very near end] \\
		& \EF{n,m+1} &
		\C{n+m+1}
		\arrow [l, "a"]
	\end{tikzcd}
	\]
	where $X$ is ``$\coprod_{u\in \marking{n,m}}\FF^{n+m}$ with $a$ removed''.
	The desired $g$ can be thus obtained as the unique strict $\omega$-functor induced by $Y \xrightarrow{!}\Cst{0}\xrightarrow{\star}\cosep{n+m+1}$ and $\C{n+m+1}\xrightarrow{w_{n+m+1}}\cosep{n+m+1}$.
	This completes the proof of \cref{distinguishing-a-in-tF}.
    \end{proof}
    
\begin{claim}
	\label{distinguishing-a-b-in-tE}
	For any $m \ge 0$ and any distinct marked $k$-cells $a,b$ in $\EF{n,m}$ with $k \ge 1$, there exist
    \begin{enumerate}
        \item a strict $\omega$-functor $g \colon \EF{n,m} \to \cosep{k}$ such that $g(a) = w_k$ and $g(b) = \star$, and
        \item a strict $\omega$-functor $h \colon \EF{n,m+1} \to \cosep{k}$ such that $he^{m+1}(a) = w_k$ and $he^{m+1}(b) = \star$.
    \end{enumerate}
    \end{claim}\begin{proof}[Proof of \cref{distinguishing-a-b-in-tE}]
        We proceed by induction on $m$.
	The base case $m=0$ is vacuous since $\EF{n,0}=\mC{n}$ has only one marked cell.
	For the inductive step, let $m \ge 0$ and let $a,b$ be distinct marked $k$-cells in $\EF{n,m+1}$.
    Note that we have $n\leq k\leq n+m+1$.

    We first prove (1) for $m+1$, that is, the existence of a suitable strict $\omega$-functor $g \colon \EF{n,m+1} \to \cosep{k}$ (note the index).
    In the case $k\leq n+m$, there exist (necessarily distinct) marked $k$-cells $a',b'$ in $\EF{n,m}$ with $e^{m+1}(a')=a$ and $e^{m+1}(b')=b$.
    By the inductive hypothesis (2), we have a strict $\omega$-functor $h' \colon \EF{n,m+1} \to \cosep{k}$ such that $h'(a') = w_k$ and $h'(b') = \star$, and we can simply take $g = h'$.
    In the case $k=n+m+1$, there exist (necessarily distinct) marked $k$-cells $a',b'$ in $\coprod_{u\in \marking{n,m}}\FF^{n+m}$ such that $f^{m+1}(a') = a$ and $f^{m+1}(b') = b$.
    The desired $g$ can be thus obtained by applying \cref{distinguishing-a-in-tF} to the pair $a',b'$.

    Next we prove (2) for $m+1$, that is, the existence of a suitable strict $\omega$-functor $h \colon \EF{n,m+2} \to \cosep{k}$.
    Since we have already constructed $g \colon \EF{n,m+1} \to \cosep{k}$, it suffices to construct a strict $\omega$-functor $h_u \colon \FF^{n+m+1} \to \cosep{k}$ satisfying $h_u(\uF) = g(u)$ for each $u \in \marking{n,m+1}$.
    This can be achieved by solving the following lifting problem: 
    		\[\begin{tikzcd}
			\C{n+m+1}
			\ar[r,"g(u)"]
			\ar[d,"u_\mathcal{F}"']
				&
				\cosep{k}
				\ar[d]
			\\
			\FF^{n+m+1}
			\ar[r]
				&
				1
		\end{tikzcd}\]
        in which the strict $\omega$-functor $u_\mathcal{F}\colon \C{n+m+1}\to \FF^{n+m+1}$ is a \emph{cofibration} (i.e., is in $\cof{I}$ with $I=\{\iota^n\colon\partial \C{n}\to \C{n}\mid n\geq 0\}$) and $\cosep{k}\to 1$ is a trivial fibration (i.e., is in $I^\pitchfork$).
    This completes the proof of \cref{distinguishing-a-b-in-tE}.
    \end{proof}
    The injectivity of the functions \eqref{eqn:maps-between-marked-cells} follows from (2) of \cref{distinguishing-a-b-in-tE} and \cref{distinguishing-a-in-tF}. 
    This completes the proof of \cref{lem:E-n-m-pushout-markings}.
\end{proof}

Now we consider the colimit of the marked weak $\omega$-categories $\EF{n,m}$, corresponding to the final result of the algebraic small object argument.

\begin{definition}\label{def:E-n-omega}
    We write $\EF{n,\omega}$ for the colimit of the sequence
    \begin{equation}
\label{dgm:EF-sequence}
    \begin{tikzcd}
        \EF{n,0}
        \ar[r,"{e^{1}}"]
        &
        \EF{n,1}
        \ar[r,"{e^{2}}"]
        &
        \EF{n,2}
        \ar[r,"{e^{3}}"]
        &
        \cdots
    \end{tikzcd}
\end{equation}
    in $\mWkCats\omega$, and write $(\eomega{m} \colon \EF{n,m} \to \EF{n,\omega})_{m\geq0}$ for the colimiting cocone.
\end{definition}

Again by \cref{rem:colimits-in-m-WkCats-omega}, the underlying weak $\omega$-category of $\EF{n,\omega}$ can be computed as the colimit of the sequence in $\WkCats\omega$, and the marking $t\EF{n,\omega}$ is given by the union $\bigcup_{m \ge 0}\eomega{m}(t\EF{n,m})$.
In particular, all marked cells in $\EF{n,\omega}$ have dimension at least $n$,
and $\eomega{0}(c_n)$ is the only marked $n$-cell in $\EF{n,\omega}$.
In fact, we can say the following about the marked cells in $\EF{n,\omega}$.

\begin{lemma}
    \label{lem:ti-omega-m-monomorphism}
    The (graded) function $t\eomega{m}\colon t\EF{n,m}\to t\EF{n,\omega}$
    is a monomorphism for each $m\geq0$.
    Consequently, $\eomega{m} \colon \EF{n,m} \to \EF{n,\omega}$ induces a bijection $\marking{n,m} \cong (t\EF{n,\omega})_{n+m}$ for each $m \ge 0$.
\end{lemma}
\begin{proof}
    By \cref{Wk-omega-Cat-LFP}, the (filtered) colimit defining $\EF{n,\omega}$ is created by the forgetful functor to $\GSet$.
    Therefore we have $(\EF{n,\omega})_k=\coprod_{m\in\N}(\EF{n,m})_k/{\sim}$,
    where $(m,x)\sim(m',x')$ if and only if there exists $\bar m\geq \max(m,m')$ such that $e^{\bar m/m}(x)=e^{\bar m/m'}(x')$ where $e^{\bar m/m}\colon\EF{n,m}\to\EF{n,\bar m}$ is the composite $e^{\bar m}\circ\cdots\circ e^{m+1}$.
    Hence, given $x,x'\in t\EF{n,m}$ with $\eomega{m}(x)=\eomega{m}(x')$, we can pick $\bar m\geq m$ such that $e^{\bar m/m}(x)=e^{\bar m/m}(x')$.
    Since $te^{\bar m/m}$ is a monomorphism by \cref{lem:E-n-m-pushout-markings}, we can conclude $x=x'$.
    This completes the proof of the first assertion, and the second assertion follows from \cref{marked-cells-in-E-n-m}.
\end{proof}

Finally we are ready to prove the main result of this subsection.

\begin{proposition}\label{presentation-1-justification}
	There is a structure $\kappa$ of algebraically $\KF$-injective object on $\EF{n,\omega}$	such that 
    \[
    \mC{n} = \EF{n,0} \xrightarrow{\eomega{0}} \EF{n,\omega}
    \]
    exhibits $(\EF{n,\omega},\kappa)$ as the algebraically $\KF$-injective replacement of $\mC{n}$. 
    In particular, we have $\EF{n,\omega}\cong \EF{n}$.
\end{proposition}
\begin{proof}
    We must specify, for each marked $(n+m)$-cell $u_\omega$ in $\EF{n,\omega}$ with $m \ge 0$, a strict $\omega$-functor
    \[
    \kappa(n+m;u_\omega) \colon \FF^{n+m} \to\EF{n,\omega}
    \]
    that sends $\uF$ to $u_\omega$.
    Given such $u_\omega$, \cref{lem:ti-omega-m-monomorphism} provides unique $u \in \marking{n,m}$ such that $\eomega{m}(u)=u_\omega$.
    We set
    \[
    \kappa(n+m;u_\omega)=\eomega{m+1}\circ f^{m+1}_{u},
    \]
    and this defines an algebraically $\KF$-injective object $(\EF{n,\omega},\kappa)$.

    We now prove that $(\EF{n,\omega},\kappa)$ enjoys the desired universal property.
    Let $(Z,\kappa^Z)$ be an algebraically $\KF$-injective object and $z$ be a marked $n$-cell in $Z$.
    We wish to show that there exists a unique morphism
    \[
    g \colon (\EF{n,\omega},\kappa)\to(Z,\kappa^Z)
    \]
    in $\KF\mhyphen\AInj$ such that $g \circ \eomega{0}(c_n) = z$.
    Define a cocone $(g^m)_{m\in\N}$ under the sequence \cref{dgm:EF-sequence} in $\mWkCats{\omega}$ inductively as follows.
	\begin{itemize}
		\item %
			$g^{0}\colon\mC{n}\to Z$ is the unique morphism that sends $c_n$ to $z$.
		\item %
			Define $\bar g^{m+1}\colon\coprod_{u\in \marking{n,m}}\FF^{n+m}\to Z$ so that its restriction to the $u$-th summand is
            \[
            \kappa^Z\bigl(n+m;g^m(u)\bigr)\colon\FF^{n+m}\to Z,
            \]
		and define $g^{m+1}$ as the morphism induced by $g^m$ and $\bar g^{m+1}$ via the universality of the pushout square \eqref{E-n-m-pushout}.
	In other words, the strict $\omega$-functor $g^{m+1}$ is the unique one satisfying $g^{m+1}\circ e^{m+1}=g^{m}$ and	$g^{m+1}\circ f^{m+1}_u=\kappa^Z\bigl(n+m;g^m(u)\bigr)$ for all $u \in \marking{n,m}$.
	\end{itemize}
    This cocone induces a marking-preserving strict $\omega$-functor $g \colon \EF{n,\omega} \to Z$, and we claim that this is the desired morphism in $\KF\mhyphen\AInj$.

    For $g$ to be a morphism in $\KF\mhyphen\AInj$, we must have
    \[
    g\circ \kappa(n+m;u_\omega)=\kappa^Z\bigl(n+m;g(u_\omega)\bigr)
    \]
    for each marked $(n+m)$-cell $u_\omega$ in $\EF{n,\omega}$ with $m \ge 0$.
    Let $u \in \marking{n,m}$ be the unique $(n+m)$-cell provided by \cref{marked-cells-in-E-n-m} such that $\eomega{m}(u) = u_\omega$.
    The desired equality can be verified as
    \begin{align*}
		g\circ\kappa(n+m;u_\omega)
		&= g\circ \eomega{m+1}\circ f^{m+1}_{u}
		\tag*{\text{(definition of $\kappa$)}}
		\\
		&= g^{m+1}\circ f^{m+1}_{u}
		\tag*{\text{(definition of $g$)}}
		\\
		&= \kappa^Z\bigl(n+m;g^{m}(u)\bigr)
		\tag*{\text{(definition of $g^{m+1}$)}}
		\\
		&= \kappa^Z\bigl(n+m;g \circ \eomega{m}(u)\bigr)
		\tag*{\text{(definition of $g$)}}
		\\
		&= \kappa^Z\bigl(n+m;g(u_\omega)\bigr).
    \end{align*}

    It remains to prove the uniqueness of $g$.
    To this end, let $g' \colon (\EF{n,\omega},\kappa)\to(Z,\kappa^Z)$ be a morphism in $\KF\mhyphen\AInj$ such that $g' \circ \eomega{0}(c_n) = z$.
    We will prove by induction on $m \ge 0$ that we necessarily have $g' \circ \eomega{m} = g^m$.
    The base case $g'\circ \eomega{0} = g^0$ is precisely the assumption $g \circ \eomega{0}(c_n) = z$.
    For the inductive step, let $m \ge 0$ and suppose that $g' \circ \eomega{m} = g^m$ holds.
    Since $\EF{n,m+1}$ is defined as the pushout \eqref{E-n-m-pushout}, in order to prove $g' \circ \eomega{m+1} = g^{m+1}$, it suffices to show that
    \begin{itemize}
        \item $g' \circ \eomega{m+1} \circ f^{m+1}_u = g^{m+1} \circ f^{m+1}_u$ holds for each $u \in \marking{n,m}$, and
        \item $g' \circ \eomega{m+1} \circ e^{m+1} = g^{m+1} \circ e^{m+1}$ holds.
    \end{itemize}
    These conditions can be verified as
	\begin{align*}
		g'\circ \eomega{m+1}\circ f^{m+1}_u
		&= g'\circ \kappa\bigl(n+m;\eomega{m}(u)\bigr)
		\tag*{\text{(definition of $\kappa$)}}
		\\
		&= \kappa^Z\bigl(n+m;g' \circ \eomega{m}(u)\bigr)
		\tag*{\text{($g'$ is a morphism in $\KF\mhyphen\AInj$)}}
		\\
		&= \kappa^Z\bigl(n+m;g^m(u)\bigr)
		\tag*{(inductive hypothesis)}
            \\
            &= g^{m+1} \circ f^{m+1}_u
            \tag*{\text{(definition of $g^{m+1}$)}}
	\end{align*}
    and
	\begin{align*}
		g' \circ \eomega{m+1}\circ e^{m+1}
		&= g' \circ \eomega{m}
            \tag*{(cocone condition)}
		\\
		&= g^m
		\tag*{(inductive hypothesis)}
            \\
            &= g^{m+1}\circ e^{m+1}
            \tag*{(cocone condition)}.
	\end{align*}
    Thus we can conclude $g' = g$, and this completes the proof.
\end{proof}

\subsection{Suspension construction}\label{subsec:suspension}
Before giving another presentation of $\EF{n}$, we need to introduce and analyse a \emph{suspension} construction for (marked) weak $\omega$-categories.
In order to reassure that it is compatible with the familiar suspension construction for strict $\omega$-categories, we first recall the following construction from \cite{Cottrell_Fujii_hom}.

\begin{remark}\label{hom-valued-in-graphs}
    Given a weak $\omega$-category $X$ and a pair of 0-cells $x,y$ in $X$, we can take the \emph{hom weak $\omega$-category} $X(x,y)$.
    Varying $x$ and $y$, we obtain a $\WkCats{\omega}$-enriched graph, and this extends to a functor which fits into middle row of the following diagram:
    \begin{equation}
		\label{dgm:suspension-compatibility-enriched-graph}
		\begin{tikzcd}
			\StrCats\omega
			\ar[r]
			\ar[d]
			\ar[rd,draw=none,"\cong"description]
				&
				\enGph{(\StrCats\omega)}
				\ar[d]
			\\
			\WkCats\omega
			\ar[r]
			\ar[d]
			\ar[rd,draw=none,"\cong"description]
				&
				\enGph{(\WkCats\omega)}
				\ar[d]
			\\
			\GSet
			\ar[r,"\sim"']
				&
				\enGph{\GSet},
		\end{tikzcd}
	\end{equation}
    where the other two horizontal maps are the more familiar hom functors.
    This is essentially the main result of \cite[Section 3]{Cottrell_Fujii_hom}, but a clarification is in order.

    The authors of \cite{Cottrell_Fujii_hom} use the category $\enGph{\omega}$ of $\omega$-graphs (informally speaking, ``$\omega$-fold'' enriched graphs) instead of $\GSet$ as the underlying category of the monads $T$ and $L$ (so that, in particular, the bottom horizontal map in \eqref{dgm:suspension-compatibility-enriched-graph} becomes an isomorphism).
    We can obtain \eqref{dgm:suspension-compatibility-enriched-graph} by transferring the corresponding (strictly commutative) diagram for $\enGph{\omega}$ along an equivalence $\GSet\simeq\enGph{\omega}$ (cf.~\cite[Appendix~F]{Leinster_book}).
    Such a diagram can be extrapolated from \cite{Cottrell_Fujii_hom} as follows.
    The perimeter of \cite[(19)]{Cottrell_Fujii_hom} is a commutative square
    \[
    \begin{tikzcd}
        \lceil L \rceil
        \arrow [r]
        \arrow [d] &
        L
        \arrow [d] \\
        (\enGph{T^{(\omega)}})1
        \arrow [r] &
        T^{(\omega)}1
    \end{tikzcd}
    \]
    of \emph{globular operads} \cite{Batanin_98,Leinster_book} (or \emph{$T^{(\omega)}$-operads} in \cite{Cottrell_Fujii_hom}), which can be thought of as representing certain monads on $\GSet$ or $\enGph{\omega}$ (see \cite[Subsection 2.1]{Cottrell_Fujii_hom} for a summary or \cite{Leinster_book} for full details).
    Unravelling their notations \cite[Subsections 2.2, 2.4, and Section 3]{Cottrell_Fujii_hom}, one can check that
    \begin{itemize}
        \item $T^{(\omega)}1$ corresponds to the monad for strict $\omega$-categories,
        \item $(\enGph{T^{(\omega)}})1$ corresponds to the monad for $\StrCats{\omega}$-enriched graphs,
        \item $L$ corresponds to the monad for weak $\omega$-categories, and
        \item $\lceil L \rceil$ corresponds to the monad for $\WkCats{\omega}$-enriched graphs.
    \end{itemize}
    Taking the categories of algebras thus yields the upper half of \eqref{dgm:suspension-compatibility-enriched-graph}, and the lower half simply witnesses that these functors commute with the forgetful functors to the base category of the monads.

    According to the lower half of \eqref{dgm:suspension-compatibility-enriched-graph}, the $n$-cells in the hom weak $\omega$-category $X(x,y)$ can be identified with the $(n+1)$-cells in $X$ whose 0-dimensional boundary is given by the pair $(x,y)$.
    As we observed in \cite[Corollary~2.5.6]{FHM1}, the weak $\omega$-category structure on $X(x,y)$ is such that
    \begin{itemize}
        \item $\id{n}{X(x,y)}{z} = \id{n+1}{X}{z}$ for any $(n-1)$-cell $z$ in $X(x,y)$, and
        \item $u \comp{n-1}{X(x,y)} v = u \comp{n}{X} v$ for any composable pair of $n$-cells $u,v$ in $X(x,y)$.\qedhere
    \end{itemize}
\end{remark}

\begin{proposition}
	\label{suspension-compatibility}
	Taking each bipointed weak $\omega$-category $(x,y) \colon \C{0}+\C{0} \to X$ to the hom weak $\omega$-category $X(x,y)$ extends to a functor $\C{0}+\C{0}/\WkCats\omega\to\WkCats\omega$ which fits into the following diagram:
    \begin{equation}
    \label{dgm:sliced-suspension-compatibility-enriched-graph}
		\begin{tikzcd}
			\Cst{0}+\Cst{0}/\StrCats\omega
			\ar[r,"\mathrm{hom}"]
			\ar[d]
			\ar[rd,draw=none,"\cong"description]
				&
				\StrCats\omega
				\ar[d]
			\\
			\C{0}+\C{0}/\WkCats\omega
			\ar[r,"\mathrm{hom}"]
			\ar[d]
			\ar[rd,draw=none,"\cong"description]
				&
				\WkCats\omega
				\ar[d]
			\\
			G^0+G^0/\GSet
			\ar[r,"\mathrm{hom}"']
				&
				\GSet
		\end{tikzcd}
    \end{equation}
	Moreover, all functors in this diagram are right adjoints.
\end{proposition}
\begin{proof}
    The first assertion follows from \cref{hom-valued-in-graphs}, so we only need to prove the second assertion.
    We know that all categories in \eqref{dgm:sliced-suspension-compatibility-enriched-graph} are cocomplete and moreover the functors
    \[
    \StrCats{\omega} \to \GSet, \quad \WkCats{\omega} \to \GSet
    \]
    and their sliced variants are all monadic.
    Hence, by \cite[Theorem~3.7.3.(b)]{Barr_Wells_TTT}, it suffices to show that the bottom horizontal functor is a right adjoint.
    It indeed admits a left adjoint, which sends a globular set $A$ to the bipointed globular set $(\star,\star') \colon G^0+G^0 \to A'$ given by
    \[
    A'_n = \begin{cases}
        \{\star,\star'\} & n=0,\\
        A_{n-1} & n \ge 1
    \end{cases}
    \]
    where $s_0(a)=\star$ and $t_0(a)=\star'$ for all $a \in A'_1$.
\end{proof}

\begin{remark}
    We do not know whether the hom-suspension adjunction for weak $\omega$-categories in \cref{suspension-compatibility} coincides with the version constructed by Benjamin and Markakis \cite{Benjamin_Markakis_hom}; see the ``related work'' section of that paper.
\end{remark}

\begin{definition}
	For each weak $\omega$-category $X$, we write $\bigl(\Sigma X, (\star,\star')\bigr)$ for its image under the left adjoint of $\mathrm{hom}\colon\C{0}+\C{0}/\WkCats\omega\to\WkCats\omega$.
    The components of the unit of this adjunction will be denoted by $\Scell \colon X\to\Sigma X(\star,\star')$.
    The \emph{suspension} functor $\Sigma \colon \WkCats{\omega} \to \WkCats{\omega}$ is the composite of this left adjoint with the forgetful functor $\C{0}+\C{0}/\WkCats\omega\to\WkCats\omega$.
\end{definition}

Note that the functor $\Sigma$ preserves all connected colimits.

\begin{proposition}
	\label{suspension-preserves-canonical-compositions}
	For each weak $\omega$-category $X$, the family of functions $(\Scell\colon X_n\to \Sigma X_{n+1})_{n\in\N}$ preserves binary compositions and identities.
    More precisely, we have
    \[
    \Scell\bigl(\id{n+1}{X}{z}\bigr)=\id{n+2}{\Sigma X}{\Scell z}
    \quad\text{and}\quad
    \Scell(u\comp{n}{X}v)=\Scell u\comp{n+1}{\Sigma X}\Scell v
    \]
	for any $n\geq0$, $n$-cell $z$, and composable pair of $n+1$-cells $u,v$ in $X$.
\end{proposition}
\begin{proof}
    This follows from \cite[Corollary~2.5.6]{FHM1} (recalled at the end of \cref{hom-valued-in-graphs}) and the fact that $\Scell\colon X\to\Sigma X(\star,\star')$ is a strict $\omega$-functor.
\end{proof}

For each $n\geq0$, the left adjoint of $\mathrm{hom}\colon G^0+G^0/\GSet\to\GSet$ sends $G^n$ and $\partial G^n$ to $G^{n+1}$ and $\partial G^{n+1}$ (with the obvious distinguished points) respectively.
Therefore, \cref{suspension-compatibility} implies that we have canonical isomorphisms $\Sigma\C{n}\cong\C{n+1}$ and $\Sigma\partial\C{n}\cong\partial\C{n+1}$.
For each weak $\omega$-category $X$ and an $n$-cell $x$ in $X$, the strict $\omega$-functor $\C{n+1}\to\Sigma X$ picking out $\Scell x$ is
the composite
\[
	\C{n+1}\cong\Sigma\C{n}\to\Sigma X
\]
where the first isomorphism is this canonical one and the second factor is obtained by applying $\Sigma$ to the the strict $\omega$-functor $\C{n} \to X$ picking out $x$.

\begin{definition}\label{def:suspension-marked}
    We define the \emph{suspension} of a marked weak $\omega$-category $(X,tX)$ to be the weak $\omega$-category $\Sigma X$ equipped with the marking $t\Sigma X=\{\Scell x\,|\, x\in tX\}$.
    This assignation extends to a functor $\Sigma\colon\mWkCats{\omega}\to\mWkCats{\omega}$ in the obvious manner.
\end{definition}

It easily follows from \cref{rem:colimits-in-m-WkCats-omega} that the suspension functor $\Sigma \colon \mWkCats{\omega} \to \mWkCats{\omega}$ preserves all connected colimits.
In particular, it preserves the marked weak $\omega$-categories $\FF^n$ in the following sense.

\begin{proposition}
	\label{suspension-compatibility-with-Fn}
	For each $n\geq1$,
	there is an isomorphism
	$\FF^{n+1}\cong\Sigma\FF^{n}$ in $\mWkCats{\omega}$ that sends the generating cells $\uF,\vF,\wF,\pF$, and $\qF$ of $\FF^{n+1}$ in \cref{def:Fn}
	to $\Scell \uF$, $\Scell\vF$, $\Scell\wF$, $\Scell \pF$, and $\Scell \qF$ respectively.
\end{proposition}
\begin{proof}
	Recall that $\FF^{n}$ is constructed in the following two steps.
    Firstly we take the colimit of the solid part of
	\[
		\begin{tikzcd}[column sep =12ex]
			(\partial\C {n})^\flat
			\ar[r,"\ppair{\tau^{n-1},\sigma^{n-1}}"]
			\ar[d,hook,"\ppair{\sigma^{n-1},\tau^{n-1}}"']
			\ar[rd,dashed,"\ppair{\yF,\xF}"description]
				&
				\mC{n}
				\ar[d,dashed,"\uF"]
					&
					(\partial\C {n})^\flat
					\ar[l,"\ppair{\tau^{n-1},\sigma^{n-1}}"']
					\ar[d,hook,"\ppair{\sigma^{n-1},\tau^{n-1}}"]
			\\
			(\C {n})^\flat
			\ar[r,dashed,"\vF"]
				&
				\FFp^{n}
					&
					(\C {n})^\flat
					\ar[l,dashed,"\wF"']
		\end{tikzcd}
    \]
     which yields a marked weak $\omega$-category $\FFp^{n}$.
	Then we take the colimit of the solid part of
	\[
		\begin{tikzcd}[column sep = 20ex]
			(\partial\C{n+1})^\flat
			\ar[r,"\ppair{\uF\comp{n-1}{}\vF,\id{n}{}{\xF}}"]
			\ar[d,hook]
				&
				\FFp^{n}
				\ar[d,dashed]
					&
					(\partial\C{n+1})^\flat
					\ar[l,"\ppair{\wF\comp{n-1}{}\uF,\id{n}{}{\yF}}"']
					\ar[d,hook]
			\\
			\mC{n+1}
			\ar[r,dashed,"\pF"]
				&
				\FF^{n}
					&
					\mC{n+1}
					\ar[l,dashed,"\qF"']
		\end{tikzcd}
	\]
    which yields $\FF^{n}$.
    The suspension functor $\Sigma \colon \mWkCats{\omega} \to \mWkCats{\omega}$ preserves connected colimits, and in particular the above defining colimit diagram of $\FFp^{n}$.
    Combining the resulting diagram with the the canonical isomorphisms $\Sigma\partial\C{n}\cong\partial\C{n+1}$ and $\Sigma\C{n}\cong\C{n+1}$ yields the following colimit diagram.
    \[
    		\begin{tikzcd}[column sep =10ex]
			(\partial\C {n+1})^\flat
			\ar[r,"\ppair{\tau^{n},\sigma^{n}}"]
			\ar[d,hook,"\ppair{\sigma^{n},\tau^{n}}"']
				&
				\mC{n+1}
				\ar[d,dashed]
					&
					(\partial\C {n+1})^\flat
					\ar[l,"\ppair{\tau^{n},\sigma^{n}}"']
					\ar[d,hook,"\ppair{\sigma^{n},\tau^{n}}"]
			\\
			(\C {n+1})^\flat
			\ar[r,dashed]
				&
				\Sigma \FFp^{n}
					&
					(\C {n+1})^\flat
					\ar[l,dashed]
		\end{tikzcd}
    \]
    Comparing it with the defining colimit diagram of $\FFp^{n+1}$, we obtain an isomorphism $\Sigma\FFp^{n}\cong\FFp^{n+1}$ that sends $\Scell u,\Scell v,\Scell w$ to $u,v,w$ in $\FFp^{n+1}$.

    The desired isomorphism $\Sigma\FF^{n}\cong\FF^{n+1}$ can be obtained by applying a similar argument to the second diagram.
    The compatibility of the upper horizontal maps and the canonical isomorphisms $\Sigma\partial\C{n+1}\cong\partial\C{n+2}$ follows from \cref{suspension-preserves-canonical-compositions}.
\end{proof}

\subsection{Presenting \texorpdfstring{$\EF{n}$}{EFn} as the free system of witnesses}
    Fix $n \ge 1$.
    In this subsection, we provide another presentation of $\EF{n}$ which uses the suspension construction.
    The starting point is the following observation.

\begin{lemma}
	For each $m \ge 0$, there is a bijection
	\[
		\marking{n,m}
		\times\{p,q\}
		\cong
		\marking{n,m+1}
	\]
	which sends $(u,p)$ to $f^{m+1}_u(\pF)$ and $(u,q)$ to $f^{m+1}_u(\qF)$ where $\pF$ and $\qF$ are the (generating) marked $(n+m+1)$-cells of $\FF^{n+m}$ in \cref{def:Fn} and $f^{m+1}_u \colon \FF^{n+m} \to \EF{n+m+1}$ is the strict $\omega$-functor in \cref{def:EF-n-m}.
\end{lemma}
\begin{proof}
    This follows from \cref{lem:E-n-m-pushout-markings}, as we have $(t\FF^{n+m})_{n+m+1}=\{p_\FF,q_\FF\}$ and $(t\EF{n,m+1})_{n+m+1}=\marking{n,m+1}$.
\end{proof}

Applying this lemma iteratively, we obtain an isomorphism $\marking{n,m} \cong \{p,q\}^m$ and so the pushout square \cref{E-n-m-pushout} can be rewritten as
\[
	\begin{tikzcd}[column sep = large]
		\displaystyle\coprod_{\phi\in\{p,q\}^{m}}\mC{n+m}
		\ar[r,"{g^{m}}"]
		\ar[d]
		\ar[dr,phantom,"\ulcorner"very near end]
			&
			\EF{n,m}
			\ar[d,"{e^{m+1}}"]
		\\
		\displaystyle\coprod_{\phi\in\{p,q\}^{m}}\FF^{n+m}
		\ar[r,"f^{m+1}"']
			&
			\EF{n,m+1}
	\end{tikzcd}
\]
where $g^0\colon\mC{n}\to\EF{n,0}=\mC{n}$ is the identity, and
$g^{m+1}$ is characterised as the unique strict $\omega$-functor such that, for each $\phi \in \{p,q\}^m$, its restrictions
\[
g^{m+1}_{\phi p}, g^{m+1}_{\phi q}\colon \mC{n+m+1} \to \EF{n,m+1}
\]
onto the $\phi p$-th and the $\phi q$-th factors send $c_{n+m+1}$ to $f^m_\phi(\pF)$ and $f^m_\phi(\qF)$ respectively; here the expression $\phi p$ means the sequence obtained from $\phi$ by appending $p$, and similarly for $\phi q$.

It follows that unpacking the data of a cocone under the sequence \cref{dgm:EF-sequence}
with nadir $X$ (we borrowed this terminology from \cite[Definition~3.1.2]{Riehl-context}) yields precisely the following data.

\begin{definition}\label{system-of-witnesses}
    Let $X$ be a marked weak $\omega$-category and let $n\geq1$.
    A \emph{system of witnesses} of dimension $n$ in $X$ consists of
    \begin{itemize}
        \item a marked $(n+m)$-cell $x_\phi$ in $X$, and
        \item $(n+m)$-cells $x_{\phi v}, x_{\phi w}$ in $X$
    \end{itemize}
    for each sequence $\phi \in \{p,q\}^m$ of length $m \ge 0$, whose source and target are given by
    \[
    \begin{split}
        x_{\phi v}, x_{\phi w} & \colon t_{n+m-1}(x_\phi) \to s_{n+m-1}(x_\phi),\\
        x_{\phi p} & \colon x_{\phi}\comp{n+m-1}{} x_{\phi v} \to \id{n+m}{}{s_{n+m-1}(x_\phi)}, \text{ and}\\
        x_{\phi q} & \colon x_{\phi w}\comp{n+m-1}{} x_{\phi}\to \id{n+m}{}{t_{n+m-1}(x_\phi)}.
    \end{split}
    \]
    Equivalently, the quintuple $(x_\phi,x_{\phi v},x_{\phi w},x_{\phi p},x_{\phi q})$ must specify a (marking-preserving) strict $\omega$-functor $\FF^{n+m}\to X$ for each $\phi \in \{p,q\}^m$.
\end{definition}

We now describe a more elementary diagram in $\mWkCats{\omega}$ such that a cocone under that diagram with nadir $X$ is precisely a system of witnesses in $X$.

\begin{definition}
	\label{def:system-of-witnesses-as-cocone}
	We write $\sys$ for the free category generated by the following graph.
	\begin{itemize}
		\item %
            The set of vertices is $\bigl\{\zig{\phi} \mid \phi\in\{p,q\}^*\bigr\} \cup \bigl\{\zag{\phi} \mid \phi\in\{p,q\}^*\bigr\}$.
		\item %
            For each $\phi \in\{p,q\}^*$, there are edges
            \[
            \alpha_\phi \colon \zig{\phi} \to \zag{\phi}, \quad \beta_{\phi p}\colon \zig{\phi p} \to \zag{\phi}, \quad\text{and}\quad \beta_{\phi q} \colon \zig{\phi q}\to \zag{\phi}.\qedhere
            \]
	\end{itemize}
\end{definition}
    The generating graph of the category $\sys$ can be visualised as follows:
	\[\begin{tikzcd}[column sep=12pt]
		&&&&&&& {\zig{\varepsilon}} \\
		&&& {\zig{p}} &&&& {\zag{\varepsilon}} &&&& {\zig{q}} \\
		& {\zig{pp}} && {\zag{p}} && {\zig{pq}} &&&& {\zig{qp}} && {\zag{q}} && {\zig{qq}} \\
		{\zig\cdot} & {\zag\cdot} & {\zig\cdot} && {\zig\cdot} & {\zag\cdot} & {\zig\cdot} && {\zig\cdot} & {\zag\cdot} & {\zig\cdot} && {\zig\cdot} & {\zag\cdot} & {\zig\cdot} \\
		\vdots && \vdots && \vdots && \vdots && \vdots && \vdots && \vdots && \vdots
		\arrow["{\alpha_\varepsilon}", from=1-8, to=2-8]
		\arrow["{\beta_p}", from=2-4, to=2-8]
		\arrow["{\alpha_{p}}", from=2-4, to=3-4]
		\arrow["{\beta_q}"', from=2-12, to=2-8]
		\arrow["{\alpha_{q}}", from=2-12, to=3-12]
		\arrow["{\beta_{pp}}", from=3-2, to=3-4]
		\arrow["{\alpha_{pp}}", from=3-2, to=4-2]
		\arrow["{\beta_{pq}}"', from=3-6, to=3-4]
		\arrow["{\alpha_{pq}}", from=3-6, to=4-6]
		\arrow["{\beta_{qp}}", from=3-10, to=3-12]
		\arrow["{\alpha_{qp}}", from=3-10, to=4-10]
		\arrow["{\beta_{qq}}"', from=3-14, to=3-12]
		\arrow["{\alpha_{qq}}", from=3-14, to=4-14]
		\arrow["{\beta_{ppp}}", from=4-1, to=4-2]
		\arrow[from=4-1, to=5-1]
		\arrow["{\beta_{ppq}}"', from=4-3, to=4-2]
		\arrow[from=4-3, to=5-3]
		\arrow["{\beta_{pqp}}", from=4-5, to=4-6]
		\arrow[from=4-5, to=5-5]
		\arrow["{\beta_{pqq}}"', from=4-7, to=4-6]
		\arrow[from=4-7, to=5-7]
		\arrow["{\beta_{qpp}}", from=4-9, to=4-10]
		\arrow[from=4-9, to=5-9]
		\arrow["{\beta_{qpq}}"', from=4-11, to=4-10]
		\arrow[from=4-11, to=5-11]
		\arrow["{\beta_{qqp}}", from=4-13, to=4-14]
		\arrow[from=4-13, to=5-13]
		\arrow["{\beta_{qqq}}"', from=4-15, to=4-14]
		\arrow[from=4-15, to=5-15]
	\end{tikzcd}\]

	Note that specifying a diagram $D\colon\sys\to\mathbf C$ in a category $\mathbf C$ is equivalent to specifying the image of the full subcategory 
    \begin{equation}
		\label{dgm:basic-part-system}
		\begin{tikzcd}
				&
				\zig\phi
				\ar[d,"\alpha_\phi"]
					&
			\\
			\zig{\phi p}
			\ar[r,"\beta_{\phi p}"]
				&
				\zag\phi
					&
					\zig{\phi q}
					\ar[l,"\beta_{\phi q}"']
		\end{tikzcd}
    \end{equation}
    for each $\phi \in \{p,q\}^m$ in a coherent manner.
    More precisely, for each $\phi \in \{p,q\}^m$ with $m \ge 1$, the object $\zig{\phi}$ appears in exactly two of such full subcategories (once at the top and once in the bottom row), and the value $D(\zig \phi)$ must be consistent between these two appearances.
\begin{definition}\label{def:C-n-F}
	For each $n\geq1$, define a functor $\tw{n}\colon\sys\to\mWkCats\omega$ by sending \eqref{dgm:basic-part-system} to
    \begin{equation}\label{diagram-for-systems}
        \begin{tikzcd}
				&
				\mC{n+m}
                \ar[d,"\uF"]
                    &
            \\
			\mC{n+m+1}
            \ar[r,"\pF"]
                &
				\FF^{n+m}
                    &
					\mC{n+m+1}
                    \ar[l,"\qF"']\text{,}
        \end{tikzcd}
    \end{equation}
    for each $\phi \in \{p,q\}^m$
    where $\uF,\pF,\qF$ are the generating marked cells of $\FF^{n+m}$ in \cref{def:Fn}.
\end{definition}

The admittedly strange notation $\tw{n}$ is motivated by the following mnemonics.
\begin{itemize}
    \item The relative position $\substack{C\\F}$ suggests that the vertical map in \eqref{diagram-for-systems} has domain $\mC{n+m}$ and codomain $\FF^{n+m}$.
    \item The upper half $\frac{C}{}$ suggests that underlined objects $\zig{\phi}$ are sent to $\mC{n+m}$.
    \item The lower half $\frac{}{F}$ suggests that overlined objects $\zag{\phi}$ are sent to $\FF^{n+m}$.
\end{itemize}

We have defined the functor $\tw{n}$ so that the following holds.

\begin{proposition}\label{presentation-2-justification}
    The category of cocones under $\tw{n}$ is isomorphic to
    the category of cocones under the sequence \cref{dgm:EF-sequence}, and this isomorphism commutes with the evident forgetful functors to $\mWkCats{\omega}$.
    In particular, $\EF{n}$ is a colimit of $\tw{n}$.
\end{proposition}

\begin{corollary}
	\label{suspension-compatibility-with-En}
	For each $n\geq1$,
	the isomorphisms $(\FF^{n+1+m}\cong\Sigma\FF^{n+m})_{m\in\N}$ and $(\mC{n+1+m}\cong\Sigma\mC{n+m})_{m\in\N}$ in \cref{suspension-compatibility-with-Fn} induce an isomorphism $\EF{n+1}\cong\Sigma\EF{n}$.
\end{corollary}
\begin{proof}
    Since the category $\sys$ is connected, we have
    \[
    \Sigma\EF{n} \cong \Sigma (\colim \tw{n}) \cong \colim (\Sigma \circ \tw{n})
    \]
    by the comment just below \cref{def:suspension-marked}.
    By \cref{suspension-compatibility-with-Fn}, the diagram $\Sigma \circ \tw{n}$ is isomorphic to $\tw{n+1}$.
\end{proof}

\subsection{Comparison to the Ozornova--Rovelli model}\label{subsec:comparison-to-OR}
In this subsection, we prove that reflecting (the underlying weak $\omega$-category of) $\EF{1}$ to $\StrCats{\omega}$ yields the strict $\omega$-category $\ORst$, which was first introduced in \cite{OR} and later shown to be weakly contractible (in the folk model structure) in \cite{HLOR}.
We first define a weak $\omega$-category $\OR$ by imitating the original construction of $\ORst$.
\begin{definition}
	\label{defining-ORn}
	Let $\ORm{0} = \C{0} \amalg \C{0}$, whose $0$-cells we call $\ORx$ and $\ORy$.
	Let $\ORm{1}$ be the weak $\omega$-category obtained from $\ORm{0}$ by freely adjoining $1$-cells
	\[
	\ORu \colon \ORx \to \ORy \quad \text{and} \quad \ORv, \ORw \colon \ORy \to \ORx.
	\]
	Write $\ORe{1} \colon \ORm{0} \to \ORm{1}$ for the inclusion, and write
    $\ORp{1},\ORq{1}\colon\Sigma\ORm{0}\to\ORm{1}$ for the strict $\omega$-functors corresponding via the adjunction $\Sigma\dashv\mathrm{hom}$ to
	\[
		\bigl\langle \ORu \comp{0}{} \ORv, \id{1}{}{\ORx}\bigr\rangle\colon\ORm{0}\to\ORm{1}(\ORx,\ORx)
		\quad \text{and} \quad
		\bigl\langle \ORw\comp{0}{} \ORu, \id{1}{}{\ORy}\bigr\rangle\colon\ORm{0}\to\ORm{1}(\ORy,\ORy)
        \text{,}
	\]
    respectively,
    where we identify morphisms from $\ORm{0}$ with pairs of $0$-cells.
	For $m \ge 1$, define the weak $\omega$-category $\ORm{m+1}$ and strict $\omega$-functors
	\[
	\ORe{m+1} \colon \ORm{m} \to \ORm{m+1} \quad \text{and} \quad \ORp{m+1}, \ORq{m+1} \colon \Sigma\ORm{m} \to \ORm{m+1}
	\]
    to be the legs of the colimiting cocone
    \begin{equation}
		\label{dgm:OR-inductive-step}
		\begin{tikzcd}[column sep = huge]
			\Sigma\ORm{m-1}
			\arrow [r, "{\ORp{m}}"]
			\arrow [d, "\Sigma\ORe{m}"']
				&
				\ORm{m}
				\arrow [d, "\ORe{m+1}",dashed]
					&
					\Sigma\ORm{m-1}
					\arrow [d, "\Sigma\ORe{m}"]
					\arrow [l, "{\ORq{m}}", swap] 
			\\
			\Sigma\ORm{m}
			\arrow [r, "{\ORp{m+1}}"',dashed]
				&
				\ORm{m+1}
					&
					\Sigma\ORm{m}
					\arrow [l, "{\ORq{m+1}}",dashed] 
		\end{tikzcd}
    \end{equation}
	in $\WkCats{\omega}$.
	We write $\OR$ for the colimit of the sequence
	\begin{equation}
		\label{dgm:OR-sequence}
        \ORm{0} \xrightarrow{\ORe{1}} \ORm{1} \xrightarrow{\ORe{2}} \ORm{2} \xrightarrow{\ORe{3}} \dots
	\end{equation}
	in $\WkCats{\omega}$.
\end{definition}

\begin{remark}
    Although we have intentionally chosen the notations
	\[
	\ORe{m+1} \colon \ORm{m} \to \ORm{m+1} \quad \text{and} \quad \ORp{m+1}, \ORq{m+1} \colon \Sigma\ORm{m} \to \ORm{m+1}
	\]
    which resemble 
    \[
    e^{m+1} \colon \EF{n,m} \to \EF{n,m+1}
    \quad\text{and}\quad
    f^{m+1}_u \colon \FF^{n+m} \to \EF{n,m+1}
    \]
    from \cref{subsec:presentation-1}, we are not claiming any rigorous relationship between them beyond their playing similar roles.
    In particular, we are giving the same name $e^{m+1}$ to two different strict $\omega$-functors, but this should cause no serious confusion because $e^{m+1} \colon \EF{n,m} \to \EF{n,m+1}$ never appears in this subsection outside of this remark.
\end{remark}

\begin{proposition}\label{reflectiong-OR}
    The strict $\omega$-categorical reflection of $\OR$ is isomorphic to the coherent walking $\omega$-equivalence $\ORst$ constructed in \cite{OR}.
\end{proposition}
\begin{proof}
	By virtue of \cref{suspension-compatibility}, the strict $\omega$-categorical reflection functor $\WkCats\omega\to\StrCats\omega$ sends the defining colimits of each $\ORm{m}$ and $\OR$ to those of $\ORstm{m}$ and $\ORst$ in \cite[Constructions 1.5.12 and 1.5.13]{OR} respectively.
\end{proof}

We now prove that $\OR$ is in fact isomorphic to $\EF{1}$. 
To this end, let $\twp\colon\sys\to\WkCats\omega$ be the unique diagram that sends \cref{dgm:basic-part-system} to
    \begin{equation}
		\label{dgm:tw'-basic-part}
        \begin{tikzcd}
                &
				\Sigma^m\ORm{0}
				\ar[d,"\Sigma^m\ORe{1}"]
                    &
            \\
			\Sigma^{m+1}\ORm{0}
			\ar[r,"\Sigma^m\ORp{1}"]
                &
                \Sigma^m\ORm{1}
                    &
					\Sigma^{m+1}\ORm{0}
                    \ar[l,"\Sigma^m\ORq{1}"']
        \end{tikzcd}
    \end{equation}
	for each $\phi\in\{p,q\}^m$.

\begin{lemma}
    \label{OR-as-colimit-of-tw'}
    The category of cocones under $\twp$ is isomorphic to the category of cocones under the sequence \cref{dgm:OR-sequence}, and this isomorphism commutes with the evident forgetful functors to $\WkCats{\omega}$.
    In particular, $\OR$ is a colimit of $\twp$.
\end{lemma}
\begin{proof}
	For each $m\geq0$, we write $\sysm{m}$ for the full subcategory of $\sys$ spanned by the subset
    \[
    {\bigl\{\zig\phi\,|\,\phi\in\{p,q\}^{\leq m}\bigr\}\sqcup\bigl\{\zag\phi\,|\,\phi\in\{p,q\}^{<m}\bigr\}},
    \]
	and $\twpn{m}$ for the restriction of $\twp$ to $\sysm{m}$.
	The inclusion $\sysm{m}\hookrightarrow\sysm{m+1}$ induces a strict $\omega$-functor
    \[
    \ORep{m+1}\colon\colim(\twpn{m})\to\colim(\twpn{m+1})
    \]
    on their colimits.
	This defines a sequence
	\[
		\begin{tikzcd}
			\colim(\twpn{0})
			\ar[r,"\ORep{1}"]
			&
			\colim(\twpn{1})
			\ar[r,"\ORep{2}"]
			&
			\colim(\twpn{2})
			\ar[r,"\ORep{3}"]
			&
			\cdots
		\end{tikzcd}
	\]
    and clearly the category of cocones under this sequence is isomorphic to that of cocones under $\twp$.
    Therefore it suffices to show that this sequence is naturally isomorphic to the sequence \cref{dgm:OR-sequence}.
    We will do so by constructing suitable strict $\omega$-functors $\ORpp{m+1}, \ORqp{m+1} \colon \Sigma \colim(\twpn{m}) \to \colim(\twpn{m+1})$ and comparing the diagram
	\begin{equation}
		\label{dgm:OR-colimiting-cocone}
		\begin{tikzcd}
			\Sigma(\colim\twpn{m-1})
			\ar[r,"\ORpp{m}"]
			\ar[d,"\Sigma\ORep{m}"']
				&
				\colim\twpn {m}
				\ar[d,"\ORep{m+1}",dashed]
					&
					\Sigma(\colim\twpn{m-1})
					\ar[l,"\ORqp{m}"']
					\ar[d,"\Sigma\ORep{m}"]
			\\
			\Sigma(\colim\twpn {m})
			\ar[r,"\ORpp{m+1}"',dashed]
				&
				\colim\twpn {m+1}
					&
					\Sigma(\colim\twpn {m})
					\ar[l,"\ORqp{m+1}",dashed]
		\end{tikzcd}
	\end{equation}
    to \eqref{dgm:OR-inductive-step}.

    For each $r\in\{p,q\}$, the function $r \concat \colon\{p,q\}^*\to\{p,q\}^*$ given by $r \concat \phi = r\phi$ defines a fully faithful functor $\sysm{m}\to\sysm{m+1}$.
    We write $r\sysm{m}$ for its image, which is precisely the full subcategory of $\sysm{m+1}$ spanned by those objects of the form $\zig{r\phi}$ or $\zag{r\phi}$.
    Moreover, we have a commutative square
    \[
    \begin{tikzcd}
        \sysm{m}
        \arrow [r, "\twpn{m}"]
        \arrow [d, "r\concat"'] &
        \WkCats{\omega}
        \arrow [d, "\Sigma"] \\
        \sysm{m+1}
        \arrow [r, "\twpn{m+1}"'] &
        \WkCats{\omega}
    \end{tikzcd}
    \]
    which induces a map between colimits $\colim(\Sigma\circ\twpn{m})\to\colim(\twpn{m+1})$.
    Since $\sysm{m}$ is connected, the functor $\Sigma$ preserves $\sysm{m}$-shaped colimits, and we obtain a morphism
    \[
    \ORrp{m+1}\colon\Sigma\colim(\twpn{m})\to\colim(\twpn{m+1}).
    \]
    Clearly these strict $\omega$-functors fit into the commutative diagram \eqref{dgm:OR-colimiting-cocone}.

    Observe that the category $\sysm{m+1}$ can be written as the union of three full subcategories
    \[
    \sysm{m+1}=p\sysm{m}\cup q\sysm{m}\cup\sysm{m}
    \]
    and the intersection
	$r\sysm{m}\cap\sysm{m}$ is precisely $r\sysm{m-1}$ for each $r\in\{p,q\}$ whereas $p\sysm{m} \cap q\sysm{m}$ is empty.
    It follows that \eqref{dgm:OR-colimiting-cocone} presents $\colim(\twpn{m+1})$ as the colimit of the solid part.
    Since $\ORrp{1}$ and $\ORep{1}$ are isomorphic to $\ORr{1}$ and $\ORe{1}$ respectively, we can now prove by induction on $m$ that the diagram \cref{dgm:OR-colimiting-cocone} is isomorphic to \cref{dgm:OR-inductive-step} for each $m\geq1$.
    This completes the proof.
\end{proof}

Now we are ready to prove the desired isomorphism.

\begin{theorem}\label{OR-EF-coincidence}
    The weak $\omega$-category $\OR$ is isomorphic to (the underlying weak $\omega$-category of) the algebraically $\KF$-injective replacement $\EF{1}$ of $\mC{1}$.
    More generally, we have $\Sigma^n\OR \cong \EF{n+1}$ for each $n\geq0$.
\end{theorem}
\begin{proof}
    Since we have $\ORm{0}=\partial\C{1}$ and $\ORm{1}=\FFp^1$, \cref{suspension-preserves-canonical-compositions} implies that we can identify the diagram \cref{dgm:tw'-basic-part} with the following diagram
    \[
	\begin{tikzcd}[column sep=17ex]
			&
			\partial\C{m+1}
			\ar[d,"\ppair{\xF,\yF}"]
				&
		\\
		\partial\C{m+2}
		\ar[r,"\ppair{\uF\comp{m}{}\vF,\id{m+1}{}{\xF}}"']
			&
			\FFp^{m+1}
				&
				\partial\C{m+2}
				\ar[l,"\ppair{\wF\comp{m}{}\uF,\id{m+1}{}{\yF}}"]
	\end{tikzcd}
    \]
    in $\WkCats{\omega}$.
    It is straightforward to see from this description that a cocone under $\twp$ with nadir $X$ is precisely a system of witnesses of dimension $1$ in $X$.
    Thus the universal property of $\OR \cong \colim(\twp)$ exhibited in \cref{OR-as-colimit-of-tw'} coincides with that of $\EF{1}$.
    The last assertion follows from \cref{suspension-compatibility-with-En}.
\end{proof}

\section{\texorpdfstring{$\omega$}{ω}-equifibrations and folk fibrations}\label{sec:fibrations-between-strict}
Let us recall the folk model structure on $\StrCats{\omega}$ constructed in \cite[Theorem 4.39]{Lafont_Metayer_Worytkiewicz_folk_model_str_omega_cat}.
As mentioned in \cref{subsec:marked-weak-omega-cats},
we will write $\Cst{n} = F^TG^n$ for the free strict $\omega$-category generated by a single $n$-cell, and call the generating $n$-cell $c_n$.
The cofibrations in the folk model structure are cofibrantly generated by
\[
\bigl\{\iota^n\colon \partial\Cst{n} \to \Cst{n} \mid n \ge 0\bigr\}.
\]
It follows that the trivial fibrations are precisely those defined in \cref{def:trivial-fibration} (but between strict $\omega$-categories).
The weak equivalences are the $\omega$-weak equivalences of \cref{def:weak-equivalence} (but again between strict $\omega$-categories).
The fibrations can be characterised as follows.
For $n \ge 1$, define the \emph{collapse map} $\partial\Cst{n} \to \Cst{n-1}$ by
\[
\begin{tikzcd}
    \partial \Cst{n-1}
    \arrow [r, hook]
    \arrow [d, hook]
    \arrow [dr, phantom, "\ulcorner" very near end] &
    \Cst{n-1}
    \arrow [d]
    \arrow [ddr, bend left, equal] & & \\
    \Cst{n-1}
    \arrow [r]
    \arrow [rrd, bend right, equal] &
    \partial\Cst{n}
    \arrow [dr, dashed] & \\
    & & \Cst{n-1}.
\end{tikzcd}
\]
Recall the notation \eqref{eqn:sigma-sigma-prime-cat}.

\begin{proposition}[{\cite[Corollary 4.38]{Lafont_Metayer_Worytkiewicz_folk_model_str_omega_cat}}]\label{generating-folk-trivial-cofibrations}
    Suppose that, for each $n \ge 1$, we are given a factorisation
    \[
    \partial\Cst{n} \xrightarrow{\igenb{n}} \Est{n} \xrightarrow{\rgen{n}} \Cst{n-1}
    \]
    of the collapse map such that $\igenb{n}$ is a folk cofibration and $\rgen{n}$ is a folk trivial fibration.
    Then
    \[
    \bigl\{\Cst{n-1} \xrightarrow{\sourceb{n-1}} \partial\Cst{n} \xrightarrow{\igenb{n}} \Est{n} \mid n \ge 1 \bigr\}
    \]
    is a generating set of folk trivial cofibrations.
\end{proposition}

The purpose of this section is to give a much more explicit characterisation of the folk fibrations, namely that they are precisely the $\omega$-equifibrations between strict $\omega$-categories.

\subsection{Folk fibrations are \texorpdfstring{$\omega$}{ω}-equifibrations}

For each $n \ge 1$, let 
\[
\Cst{n} \xrightarrow{\ipt n} \Ept{n} \xrightarrow{\rpt n} \Cst{n-1}
\]
be a factorisation of the map picking out $\id{n}{}{c_{n-1}}$ into a folk cofibration $\ipt n$ followed by a folk trivial fibration $\rpt n$, obtained by using e.g.\ the small object argument.
(The notation is intended to suggest that $\Ept{n}$ is ``pointed'', i.e., it comes equipped with a distinguished $n$-cell $\ipt n$.)
Then the composite $\partial \Cst n \xrightarrow{\iota^n} \Cst n \xrightarrow{\ipt n} \Ept{n}$ is clearly a folk cofibration, so
\[
\{\Cst{n-1} \xrightarrow{\sigma^{n-1}} \Cst{n} \xrightarrow{\ipt{n}} \Ept{n} \mid n \ge 1\}
\]
is a generating set of folk trivial cofibrations by \cref{generating-folk-trivial-cofibrations}.
Note that, since the trivial fibration $\rpt{n} \colon \Ept{n} \to \Cst{n-1}$ reflects equivalences by \cite[Lemma~4.9]{Lafont_Metayer_Worytkiewicz_folk_model_str_omega_cat}, the $n$-cell $\ipt{n}(c_n)$ (and similarly every $k$-cell in $\Ept{n}$ with $k \ge n$) is an equivalence.

The statement of the following lemma appears in \cite[Proposition 4.22]{Henry_Loubaton_inductive}.
Although the authors of \cite{Henry_Loubaton_inductive} state that it is merely a reformulation of \cite[Lemma 4.36]{Lafont_Metayer_Worytkiewicz_folk_model_str_omega_cat}, we believe that this is not quite true, so we provide a separate proof.

\begin{lemma}\label{making-equivalence-LMW-coherent}
	Let let $u \colon x \to y$ be an $n$-cell in a strict $\omega$-category $X$ with $n \ge 1$.
	Then $u$ is an equivalence if and only if the map $\Cst{n} \to X$ picking out $u$ can be extended along $\ipt{n} \colon \Cst{n} \to \Ept{n}$.
\end{lemma}
\begin{proof}
    The ``if'' direction follows from \cref{str-functor-pres-eq}.
    The ``only if'' direction is deferred to \cref{appendix-on-cylinders}.
\end{proof}

\begin{proposition}\label{folk-fib-is-equifib}
	Any folk fibration is an $\omega$-equifibration.
\end{proposition}
\begin{proof}
    Let $f\colon X\to Y$ be a folk fibration between strict $\omega$-categories.
    Let $n \ge 1$ and suppose that we are given an equivalence $n$-cell of the form $u \colon fx \to y$ in $Y$.
    Then by \cref{making-equivalence-LMW-coherent}, the map $\Cst{n} \to Y$ picking out $u$ can be extended to $\Ept{n}$, which fits into the commutative square
    \[
    \begin{tikzcd}[row sep = large]
		\Cst{n-1}
		\arrow [r, "x"]
		\arrow [d, "\ipt n\sigma^{n-1}", swap] &
		X
		\arrow [d, "f"] \\
		\Ept{n}
		\arrow [r] &
		Y.
    \end{tikzcd}
    \]
    Since the left vertical map is a folk trivial cofibration, this square admits a lift $\Ept{n} \to X$.
    In particular, it sends the $n$-cell picked out by $\ipt{n} \colon \Cst{n} \to \Ept{n}$ to an equivalence $n$-cell of the form $\bar u \colon x \to \bar y$ with $f\bar u = u$.
\end{proof}

\subsection{\texorpdfstring{$\omega$}{ω}-equifibrations are folk fibrations}
Recall from \cref{presentation-of-En} that, for each $n \ge 1$, $\iF{n} \colon \C{n} \to \EF{n}$ is (the underlying strict $\omega$-functor of) the unit map associated to the algebraically $\KF$-injective replacement (of $\mC{n}$).
We will write $\iFst{n} \colon \Cst{n} \to \EFst{n}$ for its image under the strict $\omega$-categorical reflection functor $\WkCats{\omega}\to\StrCats\omega$ defined in \cref{subsec:strict-and-weak-omega-cats}; in particular, we have $\EFst{1} \cong \ORst$ by \cref{reflectiong-OR,OR-EF-coincidence}.
The following is an immediate consequence of \cref{RLP-iff-fibration-version-F}.

\begin{lemma}\label{reflecting-En}
    A strict $\omega$-functor between strict $\omega$-categories is an $\omega$-equifibration if and only if it has the right lifting property with respect to
    \[
        \JFst=\{\,\Cst{n-1}\xrightarrow{\sigma^{n-1}}\Cst{n}\xrightarrow{\iFst{n}}\EFst{n}\mid n\geq 1\,\}.
    \]
\end{lemma}

In order to prove that $\omega$-equifibrations are folk fibrations, it thus suffices to show that $\JFst$ is a generating set of folk trivial cofibrations.
The first step is straightforward.

\begin{lemma}[{See \cite[Theorem 1.33]{HLOR} for the second assertion}]\label{hat-trivial-cofibration}
	The map $\iFst{n} \sigma^{n-1} \colon \Cst{n-1} \to \EFst{n}$ is a folk trivial cofibration for each $n \ge 1$.
    In particular, the coherent walking $\omega$-equivalence $\ORst$ is weakly equivalent to the terminal strict $\omega$-category in the folk model structure.
\end{lemma}
\begin{proof}
        \cref{folk-fib-is-equifib,reflecting-En} imply that each $\iFst{n} \sigma^{n-1}$ has the left lifting property with respect to all folk fibrations, which proves the first assertion.
        The second assertion follows from the case $n=1$.
\end{proof}

Our next step is to modify \cref{generating-folk-trivial-cofibrations} using a seemingly trivial observation regarding $\omega$-equifibrations.

\begin{definition}
	We call a strict $\omega$-category $X$ \emph{gaunt} if every equivalence in $X$ is an identity.
\end{definition}

\begin{lemma}
	Let $f \colon X \to Y$ be a morphism in $\StrCats{\omega}$ with $Y$ gaunt.
	Then $f$ is an $\omega$-equifibration.
\end{lemma}
\begin{proof}
	Let $x \in X_{n-1}$ and let $u \colon fx \to y$ be an equivalence $n$-cell in $Y$.
	Then $u$ is the identity on $fx$ by assumption, so we may simply lift it to the identity on $x$.
\end{proof}

\begin{lemma}\label{recipe}
	Suppose that, for each $n \ge 1$, we are given a factorisation
	\[
	\partial\Cst{n} \xrightarrow{\igenb{n}} \Est{n} \xrightarrow{\rgen{n}} \Cst{n-1}
	\]
	of the collapse map such that
	\begin{itemize}
		\item $\igenb{n}$ is a folk cofibration, and
		\item the composite $\Cst{n-1} \xrightarrow{\sourceb{n-1}} \partial \Cst{n} \xrightarrow{\igenb{n}} \Est{n}$ is a folk trivial cofibration.
	\end{itemize}
	Then
	\[
	\bigl\{\Cst{n-1} \xrightarrow{\sourceb{n-1}} \partial \Cst{n} \xrightarrow{\igenb{n}} \Est{n} \mid n \ge 1 \bigr\}
	\]
	is a generating set of folk trivial cofibrations.
\end{lemma}

\begin{proof}
	Since $\Cst{n-1}$ is gaunt, the map $\rgen{n}$ is automatically an $\omega$-equifibration.
	Moreover, since $\rgen{n}$ is a retraction of $\igenb{n} \sourceb{n-1}$, the 2-out-of-3 property implies that $\rgen{n}$ is an $\omega$-weak equivalence.
	It follows that $\rgen{n}$ is a folk trivial fibration.
        Hence the claim follows by \cref{generating-folk-trivial-cofibrations}.
\end{proof}

Now we are ready to prove the desired statement.

\begin{theorem}
	The set $\JFst$ is a generating set of folk trivial cofibrations.
	Consequently, the folk fibrations are precisely the $\omega$-equifibrations between strict $\omega$-categories.
\end{theorem}
\begin{proof}
    We apply \cref{recipe} to the strict $\omega$-functors $\igenb{n}$ and $\rgen{n}$ defined as follows.
    The former is the strict $\omega$-categorical reflection of $\partial\C{n} \xrightarrow{\iota^n} \C{n} \xrightarrow{\iF{n}} \EF{n}$.
    Note that this composite in $\WkCats{\omega}$ can be written as a transfinite composite of pushouts of maps of the form $\partial\C{m} \hookrightarrow \C{m}$, so its image $\igenb{n}$ is a transfinite composite of pushouts of $\partial\Cst{m} \hookrightarrow \Cst{m}$, and in particular it is a folk cofibration.
    Moreover, the composite
    \[
    (\Cst{n-1} \xrightarrow{\sourceb{n-1}} \partial \Cst{n} \xrightarrow{\igenb{n}} \EFst{n}) = 
    (\Cst{n-1} \xrightarrow{\sigma^{n-1}} \Cst{n} \xrightarrow{\iFst{n}} \EFst{n})
    \]
    is a folk trivial cofibration by \cref{hat-trivial-cofibration}.
    
    Recall that the map $\iF{n}\colon \mC{n} \to \EF{n}$ in $\mWkCats{\omega}$ is in $\cof{\KF}$.
    It follows that the map $\mC{n} \to (\C{n-1})^\natural$ picking out $\id{n}{}{c_{n-1}}$ can be extended along $\iF{n} \colon \mC{n} \to \EF{n}$.
    The strict $\omega$-functor $\rgen{n} \colon \EFst{n} \to \Cst{n-1}$ is the strict $\omega$-categorical reflection of the underlying map of this extension.
    Clearly the composite $\rgen{n}\igenb{n}$ is the collapse map by construction, and this completes the proof.
\end{proof}

\subsection*{Acknowledgements}
The first-named author is supported by ASPIRE Grant No.\ JPMJAP2301, JST.
The second-named author is supported by JSPS Research Fellowship for Young Scientists and JSPS KAKENHI Grant Number JP23KJ1365.
The third-named author is grateful to the JSPS for the support of KAKENHI Grant Numbers JP23K12960 and JP24KJ0126.

\appendix
\section{Proof of \texorpdfstring{\cref{making-equivalence-LMW-coherent}}{TheTheoremName}} % Theorem number manually set
\label{appendix-on-cylinders}

In this appendix, we complete the proof of \cref{making-equivalence-LMW-coherent}.
Our proof below is essentially a ``proof-relevant'' version of that of \cite[Lemma 4.36]{Lafont_Metayer_Worytkiewicz_folk_model_str_omega_cat}, and in particular, we will use their notion of \emph{cylinder}.
(In the following definition of cylinder, we adopt (and slightly modify) the notations $(-)^\flat$, $(-)^\sharp$, and $(-)^\natural$ from \cite{Lafont_Metayer_Worytkiewicz_folk_model_str_omega_cat}, but note that they are unrelated to the functors $\WkCats{\omega}\to \mWkCats{\omega}$ introduced in \cref{subsec:marked-weak-omega-cats}.)

\begin{definition}
	Let $X$ be a strict $\omega$-category and let $x$ be a cell of dimension $\geq 1$ in $X$.
	We write $x^\flat = s_0(x)$ and $x^\sharp = t_0(x)$.
\end{definition}
\begin{definition}[{\cite[Definition~4.14]{Lafont_Metayer_Worytkiewicz_folk_model_str_omega_cat}}]\label{cylinder-definition}
	The notion of \emph{$n$-cylinder} $U \colon x \cylarrow y$ in a strict $\omega$-category $X$, where $x$ and $y$ are $n$-cells in $X$, is defined inductively as follows.
	\begin{itemize}
		\item A $0$-cylinder $U \colon x \cylarrow y$ is an equivalence $1$-cell $U \colon x \to y$.
		\item For $n \ge 1$, an $n$-cylinder $U \colon x \cylarrow y$ is a triple $U = (U^\flat,U^\sharp,U^\natural)$ where
		\begin{itemize}
			\item $U^\flat \colon x^\flat \to y^\flat$ is an equivalence $1$-cell,
			\item $U^\sharp \colon x^\sharp \to y^\sharp$ is an equivalence $1$-cell, and
			\item $U^\natural \colon (x \comp{X}{0} U^\sharp) \cylarrow (U^\flat \comp{X}{0} y)$ is an $(n-1)$-cylinder in the hom strict $\omega$-category $X(x^\flat,y^\sharp)$.
		\end{itemize}
	\end{itemize}
	For $n \ge 1$, the \emph{source} of such $U$ is the $(n-1)$-cylinder $s_{n-1}(U) \colon s_{n-1}(x) \cylarrow s_{n-1}(y)$ defined inductively as follows.
	\begin{itemize}
		\item If $n = 1$, then $s_0(U) = U^\flat$.
		\item If $n \ge 2$, then $s_{n-1}(U)^\flat = U^\flat$, $s_{n-1}(U)^\sharp = U^\sharp$, and $s_{n-1}(U)^\natural = s_{n-2}(U^\natural)$.
	\end{itemize}
	The \emph{target} $(n-1)$-cylinder $t_{n-1}(U) \colon t_{n-1}(x) \cylarrow t_{n-1}(y)$ is defined similarly.
	These operations make the collection of all cylinders in $X$ into a globular set, which we denote by $\Gamma(X)$.
\end{definition}

\begin{theorem}[{\cite[Theorem~4.21 and Corollary~4.23]{Lafont_Metayer_Worytkiewicz_folk_model_str_omega_cat}}]\label{Gamma}
	The globular set $\Gamma(X)$ can be made into a strict $\omega$-category such that the projections
	\[
	\proj{1}{X}, \proj{2}{X} \colon \Gamma(X) \to X
	\]
	given by $\proj{1}{X}(U \colon x \cylarrow y) = x$ and $\proj{2}{X}(U \colon x \cylarrow y) = y$ are folk trivial fibrations (and in particular they are strict $\omega$-functors).
\end{theorem}

\begin{lemma}
	Let $X$ be a strict $\omega$-category and let
	\[
	\begin{tikzcd}
		x
		\arrow [r, "u"] &
		y
		\arrow [r, "v"] &
		x'
		\arrow [r, "u'"] &
		y'
	\end{tikzcd}
	\]
	be a chain of equivalence $n$-cells in $X$ with $n \ge 1$.
	Then there exists an $n$-cylinder
	\[
	\bridge{n}{X}{u}{v}{u'} \colon u \cylarrow u'.
	\]
\end{lemma}
\begin{proof}
	We proceed by induction on $n$.
	In the case $n=1$, we must provide a $1$-cylinder $u \cylarrow u'$, which amounts to a commutative-up-to-equivalence-$2$-cell square such that
	\begin{itemize}
		\item the top edge is $u$,
		\item the bottom edge is $u'$, and
		\item both vertical edges are equivalences.
	\end{itemize}
	We indeed have such a square, namely:
	\[
	\begin{tikzpicture}
		\node (ll) at (0,0) {$x'$};
		\node (lr) at (3,0) {$y'$};
		\node (ul) at (0,3) {$x$};
		\node (ur) at (3,3) {$y$};

		\draw[->] (ll) to node[labelsize, auto, swap] {$u'$} (lr);
		\draw[->] (ul) to node[labelsize, auto] {$u$} (ur);
		\draw[->] (ul) to node[labelsize, auto, swap] {$u \comp{1}{X} v$} (ll);
		\draw[->] (ur) to node[labelsize, auto] {$v \comp{1}{X} u'$} (lr);

		\draw[->, 2cell] (2.5,2.5) to node[labelsize, fill = white] {$\id{3}{X}{u \comp{1}{X} v \comp{1}{X} u'}$} (0.5,0.5);
	\end{tikzpicture}
	\]

	For $n \ge 2$, we define the $n$-cylinder $\bridge{n}{X}{u}{v}{u'}$ by
	\begin{itemize}
		\item $\bridge{n}{X}{u}{v}{u'}^\flat = \id{1}{X}{x^\flat}$,
		\item $\bridge{n}{X}{u}{v}{u'}^\sharp = \id{1}{X}{x^\sharp}$, and
		\item $\bridge{n}{X}{u}{v}{u'}^\natural = \bridge{n-1}{X(x^\flat,x^\sharp)}{u}{v}{u'}$.\qedhere
	\end{itemize}
\end{proof}

\begin{proof}[Proof of \cref{making-equivalence-LMW-coherent} continued]
    Suppose that $u \colon x \to y$ is an equivalence $n$-cell in $X$.
	Observe that the $n$-cylinder $\bridge{n}{X}{\id{n}{X}{x}}{\id{n}{X}{x}}{u}$ has type $\id{n}{X}{x} \cylarrow u$, so the strict $\omega$-functor $\Cst{n} \to \Gamma(X)$ picking out this cylinder fits into the following commutative diagram:
	\[
	\begin{tikzcd}[column sep = huge]
		\Cst{n}
		\arrow [rr, "\bridge{n}{X}{\id{n}{X}{x}}{\id{n}{X}{x}}{u}"]
		\arrow [rrr, "u", bend left]
		\arrow [d, "\ipt{n}", swap] & &
		\Gamma(X)
		\arrow [r, "\proj{2}{X}"]
		\arrow [d, "\proj{1}{X}"] &
		X \\
		\Ept{n}
		\arrow [r, "\rpt{n}", swap]
		\arrow [rru, dashed]&
		\Cst{n-1}
		\arrow [r, "x", swap] &
		X &
	\end{tikzcd}
	\]
	Since $\ipt{n}$ is a folk cofibration by construction, and $\proj{1}{X}$ is a folk trivial fibration by \cref{Gamma}, we obtain a lift as indicated.
	Composing this lift with $\proj{2}{X}$ yields the desired extension.
\end{proof}

\section{Quotient bicategory of a weak \texorpdfstring{$\omega$}{ω}-category}\label{appendix-on-quotient-bicategory}

Recall from \cref{sim-is-eq-rel} that the binary relation $\sim$ defined by the existence of equivalence cells is an equivalence relation on the set of cells in a given weak $\omega$-category $X$.
The purpose of this appendix is to show that quotienting the $2$-cells by this relation $\sim$ yields a bicategory $\tau_2(X)$.
This construction will be used in \cref{appendix-on-halfadjoint-lift-one-step} in which we provide another example of a set $K$ of maps in $\mWkCats{\omega}$ satisfying \ref{K1}--\ref{K3} based on coinductive half-adjoint equivalences.

In \cref{appendix-on-quotient-bicategory,appendix-on-halfadjoint-lift-one-step}, we will need to open up the black box presented in \cref{subsec:strict-and-weak-omega-cats} and deal with more intricate operations in a weak $\omega$-category.
Since these appendices are somewhat auxiliary and not needed for understanding the main body of this paper, we will refrain from reviewing the preliminaries, and refer the reader to \cite[Section 2]{FHM2} (and only to that section; in particular, the reader should not have to consult \cite{FHM1} unless they are interested in seeing full proofs) as needed.

We will define the quotient \emph{bicategory} construction by applying the quotient \emph{category} construction to the hom weak $\omega$-categories, which is in turn defined by applying the quotient \emph{set} construction to the hom.
(Recall from \cref{subsec:suspension} that we can take hom weak $\omega$-categories of a weak $\omega$-category.)
The seemingly strange choice of letters in the definitions and propositions below is intended to accord with this context.

\begin{definition}
	Let $Z$ be a weak $\omega$-category.
	We write $\tau_0(Z)$ for the quotient set $Z_0/{\sim}$, and $\eqclass{p} \in \tau_0(Z)$ for the equivalence class containing $p \in Z_0$.
\end{definition}

\begin{definition}
	Let $Y$ be a weak $\omega$-category.
	We write $\tau_1(Y)$ for the graph whose vertex set is $Y_0$ and whose edges are given by
	\[
	\tau_1(Y)(u,v) = \tau_0\bigl(Y(u,v)\bigr).\qedhere
	\]
\end{definition}

\begin{proposition}
	Let $Y$ be a weak $\omega$-category.
	Then setting
	\begin{itemize}
		\item $\id{1}{\tau_1(Y)}{u} = \eqclass{\id{1}{Y}{u}}$ for each $0$-cell $u$ in $Y$, and
		\item $\eqclass{p} \comp{0}{\tau_1(Y)} \eqclass{q} = \eqclass{p \comp{0}{Y} q}$ for each composable pair of $1$-cells $p,q$ in $Y$
	\end{itemize}
	makes $\tau_1(Y)$ into a category.
\end{proposition}
\begin{proof}
	The composition is well defined by \cref{sim-is-congruence}, and the category axioms follow from \cref{unit-law,associativity}.
\end{proof}

The following straightforward observation will be useful later.

\begin{proposition}
	\label{equivalences-in-quotient-category}
	Let $Y$ be a weak $\omega$-category and let $p \colon u \to v$ be a $1$-cell in $Y$.
	Then $p$ is an equivalence in $Y$ if and only if $\eqclass{p}$ is an isomorphism in $\tau_1(Y)$.
\end{proposition}
\begin{proof}
	By definition of $\tau_1(Y)$, the morphism $\eqclass{p}$ is an isomorphism if and only if there exists a $1$-cell $q \colon v \to u$ such that $p \comp{0}{Y} q \sim \id{1}{Y}{u}$ and $q \comp{0}{Y} p \sim \id{1}{Y}{v}$.
\end{proof}

Now we can define the underlying category-enriched graph of the desired quotient bicategory.

\begin{definition}
	Let $X$ be a weak $\omega$-category.
	We write $\tau_2(X)$ for the category-enriched graph whose vertex set is $X_0$ and whose hom categories are given by
	\[
	\tau_2(X)(a,b) = \tau_1\bigl(X(a,b)\bigr).\qedhere
	\]
\end{definition}

Before describing the entire bicategory structure, we need to give names to certain instances of what we call \emph{coherence} \cite[Proposition 2.4.6]{FHM2}.

\begin{definition}
Given an $n$-cell $u \colon a \to b$ in a weak $\omega$-category $X$, we write:
\begin{itemize}
	\item $\lunit{n+1}{X}{u} \colon \id{n}{X}{a} \comp{n-1}{X} u \to u$
	\item $\lunitinv{n+1}{X}{u} \colon u \to \id{n}{X}{a} \comp{n-1}{X} u$
	\item $\runit{n+1}{X}{u} \colon u \comp{n-1}{X}\id{n}{X}{b} \to u$
	\item $\runitinv{n+1}{X}{u} \colon u \to u \comp{n-1}{X}\id{n}{X}{b}$
\end{itemize}
for equivalence $(n+1)$-cells in $X$ witnessing the unit law (\cref{unit-law}).

Given a composable triple of $n$-cells $u_i \colon a_{i-1} \to a_i$ for $1 \le i \le 3$, we write
\begin{itemize}
	\item $\assoc{n+1}{X}{u_1}{u_2}{u_3} \colon (u_1 \comp{n-1}{X} u_2) \comp{n-1}{X} u_3 \to u_1 \comp{n-1}{X} (u_2 \comp{n-1}{X} u_3)$
	\item $\associnv{n+1}{X}{u_1}{u_2}{u_3} \colon u_1 \comp{n-1}{X} (u_2 \comp{n-1}{X} u_3) \to (u_1 \comp{n-1}{X} u_2) \comp{n-1}{X} u_3$
\end{itemize}
for equivalence $(n+1)$-cells in $X$ witnessing the associative law (\cref{associativity}).
\end{definition}

The following proposition uses binary compositions of higher codimension (namely composing $2$- and $3$-cells along a $0$-cell) which was introduced in \cite[Definition~2.3.5]{FHM2}.
\begin{proposition}
	Let $(X,\xi \colon LX \to X)$ be a weak $\omega$-category.
        For each $0$-cell $a$ in $X$, define $\id{1}{\tau_2(X)}{a} = \id{1}{X}{a}$.
        For each triple of $0$-cells $a,b,c$ in $X$, define a graph morphism
    \[
    -\comp{0}{\tau_2(X)}-\colon\tau_1\bigl(X(a,b)\bigr)\times\tau_1\bigl(X(b,c)\bigr)\to\tau_1\bigl(X(a,c)\bigr)
    \]
    as follows.
	\begin{itemize}
		\item %
			For $1$-cells $u\colon a\to b$ and $v\colon b\to c$ in $X$, the $1$-cell $u\comp{0}{\tau_2(X)}v\colon a\to c$ is $u\comp{0}{X}v$.
		\item %
			For $2$-cells $p\colon u\to u'\colon a\to b$ and $q\colon v\to v'\colon b\to c$ in $X$,
			the equivalence class
            \[
            \eqclass{p}\comp{0}{\tau_2(X)}\eqclass{q}\in\tau_1\bigl(X(a,c)\bigr)(u\comp{0}{X}v,u'\comp{0}{X}v')
            \]
            is $\eqclass{p\comp{0}{X}q}$.
	\end{itemize}
	This is well defined and moreover defines a functor $-\comp{0}{\tau_2(X)}-\colon\tau_1\bigl(X(a,b)\bigr)\times\tau_1\bigl(X(b,c)\bigr)\to\tau_1\bigl(X(a,c)\bigr)$.
	
	Furthermore, the morphisms
	\[
		\alpha_2^{\tau_2(X)}(u,v,w)=\eqclass{\alpha_2^X(u,v,w)}\colon\left(u\comp{0}{\tau_2(X)}v\right)\comp{0}{\tau_2(X)}w\to u\comp{0}{\tau_2(X)}\left(v\comp{0}{\tau_2(X)}w\right)
		\text{,}
	\]
	\[
		\lambda_2^{\tau_2(X)}(u)=\eqclass{\lambda_2^X(u)}\colon\id{1}{\tau_2(X)}{x}\comp{0}{\tau_2(X)}u\to u
		\text{, and}
	\]
	\[
		\rho_2^{\tau_2(X)}(u)=\eqclass{\rho_2^X(u)}\colon u\comp{0}{\tau_2(X)}\id{1}{\tau_2(X)}{y}\to u
	\]
	in the hom-categories of $\tau_2(X)$ form suitable natural isomorphisms, making $\tau_2(X)$ into a bicategory.
\end{proposition}
\begin{proof}
    We first verify that $-\comp{0}{\tau_2(X)}-$ is well defined.
    Suppose we are given
    \begin{itemize}
        \item $1$-cells $u,u'\colon a\to b$ and $v,v'\colon b\to c$,
        \item $2$-cells $p_1,p_2\colon u\to u'$ and $q_1,q_2\colon v\to v'$, and
        \item  equivalence $3$-cells $h\colon p_1\to p_2$ and $k\colon q_1\to q_2$
    \end{itemize}
    in $X$. 
    We must provide an equivalence $3$-cell $p_1\comp{0}{X}q_1\to p_2\comp{0}{X}q_2$.
    A $3$-cell of correct type can be constructed as $h\comp{0}{X}k$, and it is moreover an equivalence in $X$ by \cite[Theorem~2.4.8]{FHM2}.

    To see that $-\comp{0}{\tau_2(X)}-$ is a functor, suppose we are given
    \begin{itemize}
        \item $1$-cells $u,u',u''\colon a\to b$ and $v,v',v''\colon b\to c$, and
        \item $2$-cells $p\colon u\to u'$, $p'\colon u'\to u''$, $q\colon v\to v'$, and $q'\colon v\to v'$
    \end{itemize}
    in $X$.
    We wish to prove that
    \[
        \id{1}{X(a,b)}{u}\comp{0}{X}\id{1}{X(b,c)}{v} \sim \id{1}{X(a,c)}{u\comp{0}{X}v}
    \]
    and
    \[
	\left(p\comp{0}{X(a,b)}p'\right)\comp{0}{X}\left(q\comp{0}{X(b,c)}q'\right) \sim (p\comp{0}{X}q)\comp{0}{X(a,c)}(p'\comp{0}{X}q')
    \]
    hold.
    By the last paragraph of \cref{hom-valued-in-graphs}, we can rewrite these expressions as
    \begin{equation}\label{horizontal-composition-preserves-units}
		\id{2}{X}{u}\comp{0}{X}\id{2}{X}{v}\sim\id{2}{X}{u\comp{0}{X}v}
	\end{equation}
    and
	\[
		(p\comp{1}{X}p')\comp{0}{X}(q\comp{1}{X}q')
		\sim
		(p\comp{0}{X}q)\comp{1}{X}(p'\comp{0}{X}q').
	\]
    Because their proofs are similar, we will only spell out the former.
    
    The left-hand side of \eqref{horizontal-composition-preserves-units} can be rewritten in terms of the corresponding operations in $L1$ (as opposed to $X$) as follows; see \cite[Subsection 2.2 and 2.3]{FHM2} for the notation.
    \begin{align*}
        \id{2}{X}{u}\comp{0}{X}\id{2}{X}{v} &= \xi\left(\widetilde e_2 \comp{0}{L1} \widetilde e_2, \begin{bmatrix}
            \id{2}{X}{u} & & \id{2}{X}{v} \\ & b &
        \end{bmatrix}\right) \tag*{\text{(\cite[Proposition 2.3.7 (2)]{FHM2})}}\\
        &= \xi\left(\widetilde e_2 \comp{0}{L1} \widetilde e_2, \begin{bmatrix}
            \xi\bigl(\id{2}{L1}{\tilde e_1}, [u]\bigr) & & \xi\bigl(\id{2}{L1}{\tilde e_1}, [v]\bigr) \\ & \xi\bigl(\widetilde e_0, [b]\bigr) &
        \end{bmatrix}\right) \tag*{\text{(\cite[Proposition 2.3.7 (1)]{FHM2} and unit law for $L$-algebra)}}\\
        &= \xi \circ L\xi \left(\widetilde e_2 \comp{0}{L1} \widetilde e_2, \begin{bmatrix}
            \bigl(\id{2}{L1}{\tilde e_1}, [u]\bigr) & & \bigl(\id{2}{L1}{\tilde e_1}, [v]\bigr) \\ & \bigl(\widetilde e_0, [b]\bigr) &
        \end{bmatrix}\right) \tag*{\text{(action of $L\xi$)}}\\
        &= \xi \circ \mu^L_X \left(\widetilde e_2 \comp{0}{L1} \widetilde e_2, \begin{bmatrix}
            \bigl(\id{2}{L1}{\tilde e_1}, [u]\bigr) & & \bigl(\id{2}{L1}{\tilde e_1}, [v]\bigr) \\ & \bigl(\widetilde e_0, [b]\bigr) &
        \end{bmatrix}\right) \tag*{\text{(associative law for monad $L$)}}\\
        &= \xi\left(\mu^L_1\left(\widetilde e_2 \comp{0}{L1} \widetilde e_2, \begin{bmatrix} \id{2}{L1}{\tilde e_1} & & \id{2}{L1}{\tilde e_1} \\ & \widetilde e_0 & \end{bmatrix}\right), \mu^T_X \begin{bmatrix} [u] & & [v] \\ & [b] & \end{bmatrix}\right) \tag*{\text{(justified below)}}\\
        &= \xi\left(\id{2}{L1}{\tilde e_1} \comp{0}{L1} \id{2}{L1}{\tilde e_1}, \begin{bmatrix} u & & v \\ & b & \end{bmatrix}\right) \tag*{\text{(\cite[Proposition 2.3.7 (2)]{FHM2} and action of $\mu^T_X$)}}
    \end{align*}
    Here the penultimate equality follows from the fact that $\mu^L_X$ fits into the commutative cube
    \[
    \begin{tikzcd}[row sep = large]
        LLX
        \arrow [rr, "\mu^L_X"]
        \arrow [dr, "LL!"']
        \arrow [dd, "\arity\arity_X"'] & &
        LX
        \arrow [dr, "L!"]
        \arrow [dd, "\arity_X" near end] & \\
        & LL1
        \arrow [rr, "\mu^L_1" near start]
        \arrow [dd, "\arity\arity_1"' near start] & &
        L1
        \arrow [dd, "\arity_1"] \\
        TTX
        \arrow [rr, "\mu^T_X"' near end]
        \arrow [dr, "TT!"'] & &
        TX
        \arrow [dr, "T!"] & \\
        & TT1
        \arrow [rr, "\mu^T_1"'] & & 
        T1
    \end{tikzcd}
    \]
    whose left and right faces are pullbacks.
    Similarly, the right-hand side of \eqref{horizontal-composition-preserves-units} can be rewritten as
    \[
    \id{2}{X}{u\comp{0}{X}v} = \xi\left(\id{2}{L1}{\widetilde{e}_1 \comp{0}{L1} \widetilde{e}_1},
    \begin{bmatrix}
            u & & v \\ & b & 
        \end{bmatrix}\right).
    \]
    Since both $\id{2}{L1}{\widetilde{e}_1}\comp{0}{L1}\id{2}{L1}{\widetilde{e}_1}$ and $\id{2}{L1}{\widetilde{e}_1 \comp{0}{L1} \widetilde{e}_1}$ are $2$-cells in $L1$ of type $\widetilde{e}_1 \comp{0}{L1} \widetilde{e}_1 \to \widetilde{e}_1 \comp{0}{L1} \widetilde{e}_1$ and arity $\begin{bmatrix} 1 & & 1 \\ & 0 & \end{bmatrix}^{(2)}$, the existence of the desired invertible $3$-cell can be deduced from \cite[Proposition~3.2.5]{FHM1}.

        For the second part of the proposition, that they are invertible follow from \cref{equivalences-in-quotient-category}.
	Their naturality, the pentagon identity, and the triangle identity can all be proved using the same strategy as above.
	To give a little more details, by \cite[Corollary~2.3.9]{FHM2}, we can reduce the desired commutativity conditions to the following statements: given
    \begin{itemize}
        \item $0$-cells $a,b,c,d,e$,
        \item $1$-cells $u,u'\colon a\to b$, $v,v'\colon b\to c$, $w,w'\colon c\to d$, $x\colon d\to e$, and
        \item $2$-cells $p\colon u\to u'$, $q\colon v\to v'$, $r\colon w\to w'$
    \end{itemize} 
    in $X$, there exist equivalence $3$-cells
	\[
		\alpha_2^X(u,v,w)\comp{1}{X}\bigl(p\comp{0}{X}(q\comp{0}{X}r)\bigr)
		\sim
		((p\comp{0}{X}q)\comp{0}{X}r)\comp{1}{X}\alpha_2^X(u',v',w'),
	\]
	\[
		\lambda_2^X(u)\comp{1}{X}p
		\sim
		\bigl(\id{1}{X}{a}\comp{0}{X}p\bigr)\comp{1}{X}\lambda_2^X(u'),
	\]
	\[
		\rho_2^X(u)\comp{1}{X}p
		\sim
		\bigl(p\comp{0}{X}\id{1}{X}{b}\bigr)\comp{1}{X}\rho_2^X(u')
	\]
	\[
		\alpha_2^X(u\comp{0}{X}v,w,x)\comp{1}{X}\alpha_2^X(u,v,w\comp{0}{X}x)
		\sim
		\Bigl(\bigl(\alpha_2^X(u,v,w)\comp{0}{X}\id{2}{X}{x}\bigl)\comp{1}{X}\alpha_2^X(u,v\comp{0}{X}w,x)\Bigr)\comp{1}{X}\bigl(\id{2}{X}{u}\comp{0}{X}\alpha_2^X(v,w,x)\bigr),
	\]
	\[
		\alpha_2^X(u,\id{1}{X}{b},v)\comp{1}{X}\bigl(\id{2}{X}{u}\comp{0}{X}\lambda_2^X(v)\bigr)
		\sim
		\rho_2^X(u)\comp{0}{X}v.
	\]
    Each of these can be proved by rewriting both sides in terms of the corresponding operations in $L1$ using \cite[Proposition 2.3.7]{FHM2} and then applying \cite[Proposition 2.4.6]{FHM2}.
\end{proof}
\begin{proposition}\label{equivalences-in-quotient-bicategory}
	Let $X$ be a weak $\omega$-category, and let $u \colon x \to y$ be a $1$-cell in $X$.
	Then $u$ is an equivalence in $X$ if and only if $u$ is an equivalence when regarded as a $1$-cell in $\tau_2(X)$.
\end{proposition}
\begin{proof}
	By definition of $\tau_2(X)$,
	the $1$-cell $u$ is an equivalence in $\tau_2(X)$ if and only if
	there exists a 1-cell $v\colon y \to x$
	such that there are isomorphisms $u\comp{0}{X}v\cong\id{1}{X}{x}$ and $v\comp{0}{X}u\cong\id{1}{X}{y}$ in $\tau_1(X(x,x))$ and $\tau_1(X(y,y))$ respectively,
	which is by \cref{equivalences-in-quotient-category}, equivalent to $u \comp{0}{X} v \sim \id{1}{X}{x}$ and $v \comp{0}{X} u \sim \id{1}{X}{y}$.
\end{proof}

\section{An example of \texorpdfstring{$K$}{K} based on half-adjoint equivalences}
\label{appendix-on-halfadjoint-lift-one-step}

Here we give another example of a set $K$ of morphisms in $\mWkCats{\omega}$ satisfying conditions \ref{K1}--\ref{K3} stated in \cref{subsec:marked-weak-omega-cats}. 
By the results of \cref{subsec:omega-equifib-via-RLP}, it gives rise to a set of morphisms in $\WkCats{\omega}$ characterising $\omega$-equifibrations via the right lifting property.
Whereas the example in \cref{subsec:example-of-K} is based on the notion of flat equivalence, the example given here is based on half-adjoint equivalences. 
See \cref{sec:intro} for the homotopical intuition.

\begin{remark}
	We declare each $\comp{n}{X}$ to be notationally left associative.
	That is, given enough cells $u_i$ of dimension $>n$ in a weak $\omega$-category $X$ with each consecutive pair being composable along the $n$-dimensional boundary, the expression $u_1 \comp{n}{X} u_2 \comp{n}{X} u_3$ is understood as $(u_1 \comp{n}{X} u_2) \comp{n}{X} u_3$, and $u_1 \comp{n}{X} u_2 \comp{n}{X} u_3 \comp{n}{X} u_4$ is understood as $\bigl((u_1 \comp{n}{X} u_2) \comp{n}{X} u_3\bigr) \comp{n}{X} u_4$ etc.
\end{remark}

\begin{definition}
    For each $n \ge 1$, we define a marked weak $\omega$-category $(\HH^n,t\HH^n)$.
    (Here H stands for ``half-adjoint''.)
    The underlying weak $\omega$-category $\HH^n$ is constructed as follows:
    \begin{itemize}
        \item start with $\C{n}$, and call its fundamental $n$-cell $\uH \colon \xH \to \yH$,
        \item freely adjoin an $n$-cell $\vH \colon \yH \to \xH$, that is, take the pushout
        \[
        \begin{tikzcd}[row sep = large]
            \partial\C{n}
            \arrow [r, "{(\yH,\xH)}"]
            \arrow [d, hook] &
            \C{n}
            \arrow [d, "\uH"] \\
            \C{n}
            \arrow [r, "\vH", swap] &
            \bullet
        \end{tikzcd}
        \]
        \item freely adjoin $(n+1)$-cells $\pH \colon \id{n}{}{\xH} \to \uH \comp{n-1}{} \vH$ and $\qH \colon \vH \comp{n-1}{} \uH \to \id{n}{}{\yH}$, and
        \item freely adjoin an $(n+2)$-cell $\rH$ from $\id{n+1}{}{\uH}$ to
        \[
        \lunitinv{n+1}{}{\uH} \comp{n}{} (\pH \comp{n-1}{}\uH) \comp{n}{} \assoc{n+1}{}{\uH}{\vH}{\uH} \comp{n}{} (\uH \comp{n-1}{} \qH) \comp{n}{} \runit{n+1}{}{\uH}.
        \]
    \end{itemize}
    The marked cells are $\uH$, $\pH$, $\qH$, and $\rH$.

    We define the marking-preserving strict $\omega$-functor $\kH^n\colon \mC n\to \HH^n$ as the one picking out $\uH$.
\end{definition}

We claim that the set 
\[
\KH=\{\,\kH^n\colon \mC{n}\to \HH^n\mid n\geq 1\,\}
\]
of morphisms in $\mWkCats{\omega}$ satisfies conditions \ref{K1}--\ref{K3}.

We first show that $\KH$ satisfies \ref{K1}. The proof of the following proposition (especially the use of \cref{whiskering-ess-0-surj} in it) is inspired by \cite[Proof of Proposition~6]{Lack-bicat}. 

\begin{proposition}
	\label{lift-one-step-apx}
    \label{K1-half-adjoint-eq}
	Let $f \colon X \to Z$ be an $\omega$-equifibration between weak $\omega$-categories, and let $n \ge 1$.
	Suppose that we are given a commutative square
	\[
	\begin{tikzcd}[row sep = large]
		\mC{n}
		\arrow [d, "\kH^n"']
		\arrow [r, "u"] &
		X^\natural
		\arrow [d, "f^\natural"] \\
		\HH^n
		\arrow [r] &
		Z^\natural
	\end{tikzcd}
	\]
	in $\mWkCats{\omega}$.
	Then this square admits a lift.
\end{proposition}

\begin{proof}
	The data of the given commutative square are equivalent to:
	\begin{itemize}
		\item an equivalence $n$-cell $u \colon x \to y$ in $X$,
		\item an $n$-cell $v \colon fy \to fx$ in $Z$,
		\item equivalence $(n+1)$-cells $\unit \colon \id{n}{Z}{fx} \to fu \comp{n-1}{Z} v$ and $\counit \colon v \comp{n-1}{Z} fu \to \id{n}{Z}{fy}$ in $Z$, and
		\item an equivalence $(n+2)$-cell $\tri$ from $\id{n+1}{Z}{fu}$ to
		\[
		\lunitinv{n+1}{Z}{fu} \comp{n}{Z} (\unit \comp{n-1}{Z} fu) \comp{n}{Z} \assoc{n+1}{Z}{fu}{v}{fu} \comp{n}{Z} (fu \comp{n-1}{Z} \counit) \comp{n}{Z} \runit{n+1}{Z}{fu}
		\]
		in $Z$.
	\end{itemize}
	We wish to extend $u$ to a quintet $(u,\bar v, \bar \unit, \bar \counit, \bar \tri)$ of suitable type in $X$
	(so that it corresponds to a marking-preserving map $\HH^n \to X^\natural$)
	such that $f$ sends it to $(fu,v,\unit,\counit,\tri)$.

        We lift $v$ and $\unit$ using an argument similar to that in the proof of \cref{K-satisfies-K1-flat}; the details are as follows.

	Since $u$ is an equivalence, we can choose an $n$-cell $\bar v' \colon y \to x$ and an equivalence $(n+1)$-cell $\bar \unit' \colon \id{n}{X}{x} \to u \comp{n-1}{X} \bar v'$ in $X$.
	Since $f\bar \unit'\colon \id{n}{Z}{fx}\to fu\comp{n-1}{Z}f\bar v'$ is an equivalence in $Z$ by \cref{str-functor-pres-eq}, there exists an equivalence $(n+1)$-cell
	\[
	v^\dag \colon fu \comp{n-1}{Z} f \bar v' \to fu \comp{n-1}{Z} v
	\]
	in $Z$ such that $f \bar \unit' \comp{n}{Z} v^\dag \sim \unit$ by (1) of \cref{whiskering-ess-0-surj}.
	Similarly, since $fu$ is an equivalence in $Z$, there exists an equivalence $(n+1)$-cell
	\[
	v^\ddag \colon f \bar v' \to v
	\]
	such that $fu \comp{n-1}{Z} v^\ddag \sim v^\dag$ by (2) of \cref{whiskering-ess-0-surj}.
	Since $f$ is an $\omega$-equifibration, we can lift $v^\ddag$ to an equivalence $(n+1)$-cell in $X$, which we denote as $\bar v^\ddag \colon \bar v' \to \bar v$.
	This completes the lifting of $v$.
	
	\begin{figure}
	$\begin{tikzpicture}[baseline = 40]
		\node (1ul) at (0,3) {$fx$};
		\node (1ur) at (3,3) {$fy$};
		\node (1lr) at (3,0) {$fx$};

		\draw[->] (1ul) to node [labelsize, auto] {$fu$} (1ur);
		\draw[->] (1ul) to node [labelsize, auto, swap] {$\id{n}{Z}{fx}$} (1lr);
		\draw[->, bend right = 30] (1ur) to node [labelsize, auto, swap] {$f\bar v'$} (1lr);
		\draw[->, bend left = 30] (1ur) to node [labelsize, auto] {$v$} (1lr);
		\draw[->, 2cell] (1.7,2) to node [labelsize, auto] {$f\bar \unit'$} (2.2,2.5);
		\draw[->, 2cell, dashed] (2.7,1.5) to node [labelsize, auto] {$\exists v^\ddag$} node [labelsize, auto, swap] {$\sim$} (3.3,1.5);
	\end{tikzpicture}
	\qquad\sim\qquad
	\begin{tikzpicture}[baseline = 40]
		\node (1ul) at (0,3) {$fx$};
		\node (1ur) at (3,3) {$fy$};
		\node (1lr) at (3,0) {$fx$};

		\draw[->] (1ul) to node [labelsize, auto] {$fu$} (1ur);
		\draw[->] (1ul) to node [labelsize, auto, swap] {$\id{n}{Z}{fx}$} (1lr);
		\draw[->, bend left = 30] (1ur) to node [labelsize, auto] {$v$} (1lr);
		\draw[->, 2cell] (2.1,1.9) to node [labelsize, auto] {$\unit$} (2.6,2.4);
	\end{tikzpicture}$
		\caption{Lifting \texorpdfstring{$v$}{v}}
		\label{lifting-v-apx}
	\end{figure}

	Next, we lift the $(n+1)$-cell $\unit$.
	Observe that we have an equivalence $(n+1)$-cell
	\[
	\bar \unit' \comp{n}{X}(u \comp{n-1}{X} \bar v^\ddag) \colon \id{n}{X}{x} \to u \comp{n-1}{X} \bar v
	\]
	in $X$, and equivalence $(n+2)$-cells witnessing
	\[
	f\bigl(\bar \unit' \comp{n}{X}(u \comp{n-1}{X} \bar v^\ddag)\bigr) = f\bar \unit' \comp{n}{Z}(fu \comp{n-1}{Z} v^\ddag) \sim f \bar \unit' \comp{n}{Z} v^\dag \sim \unit
	\]
	in $Z$.
	Since $f$ is an $\omega$-equifibration, we can lift the composite of the latter to an equivalence $(n+2)$-cell
	\[
	\bar \unit' \comp{n}{X}(u \comp{n-1}{X} \bar v^\ddag) \sim \bar \unit
	\]
	in $X$.
	Note that $\bar \unit$ is an equivalence in $X$ by \cref{prop:invariance-of-equivalence}.

	The way we will lift $\counit$ and $\tri$ will be similar to how we lifted $v$ and $\unit$, and \cref{lifting-e-apx} provides its summary.
	Consider the quotient bicategory $\tau_2\bigl(X(s_{n-2}(x),t_{n-2}(x))\bigr)$, where $X(s_{n-2}(x),t_{n-2}(x))$ is the suitable (iterated) hom weak $\omega$-category of $X$. (When $n=1$, we interpret $X(s_{n-2}(x),t_{n-2}(x))$ as $X$.)
	In this bicategory, $u$ and $\bar v'$ are pseudo-inverse equivalence $1$-cells, and moreover we have $\bar v' \cong \bar v$.
	It follows that $\bar v$ is also a pseudo-inverse of $u$, and so the triple $\bigl(u,\bar v, \eqclass{\bar \unit}\bigr)$ can be completed to an adjoint equivalence.
	In other words, we can pick an equivalence $(n+1)$-cell $\bar \counit' \colon \bar v \comp{n-1}{X} u \to \id{n}{X}{y}$ in $X$ and an equivalence $(n+2)$-cell $\bar \tri'$ from $\id{n+1}{X}{u}$ to
	\[
		\lunitinv{n+1}{X}{u} \comp{n}{X} (\bar \unit \comp{n-1}{X} u) \comp{n}{X} \assoc{n+1}{X}{u}{\bar v}{u} \comp{n}{X} (u \comp{n-1}{X} \bar \counit') \comp{n}{X} \runit{n+1}{X}{u}
	\]
	in $X$.
	Making use of \cref{whiskering-ess-0-surj} with various equivalence cells, we can obtain an equivalence $(n+2)$-cell $\counit^\dag \colon f \bar \counit' \to \counit$ such that there is an equivalence $(n+3)$-cell from
	\[
	f \bar \tri' \comp{n+2}{Z}\bigl(\lunitinv{n+1}{Z}{fu} \comp{n}{Z} (\unit \comp{n-1}{Z} fu) \comp{n}{Z} \assoc{n+1}{Z}{fu}{v}{fu} \comp{n}{Z} (fu \comp{n-1}{Z} \counit^\dag) \comp{n}{Z} \runit{n+1}{Z}{fu}\bigr)
	\]
	to $\tri$ in $Z$.
	Since $f$ is an $\omega$-equifibration, we can lift $\counit^\dag$ to an equivalence $(n+2)$-cell $\bar \counit^\dag \colon \bar \counit' \to \bar \counit$.
	Finally, we observe that the domain of the above equivalence $(n+3)$-cell is in the image of $f$, and we take $\bar \tri$ to be the codomain of its lift through $f$.
	\begin{figure}
		$\left[\begin{tikzpicture}[baseline = 40]
			\node at (-4,2) {$\id{n+1}{Z}{fu}$};

			\draw[->, 3cell] (-2.5,2) to node [labelsize, auto] {$f\bar \tri'$} (-0.5,2);
			
			\node (ul) at (0,4) {$fx$};
			\node (um) at (4,4) {$fy$};
			\node (lm) at (4,0) {$fx$};
			\node (lr) at (8,0) {$fy$};
			
			\draw[->] (ul) to node [labelsize, auto] {$fu$} (um);
			\draw[->] (ul) to node [labelsize, auto, swap] {$\id{n}{Z}{fx}$} (lm);
			\draw[->] (um) to node [labelsize, auto, swap] {$v$} (lm);
			\draw[->] (um) to node [labelsize, auto] {$\id{n}{Z}{fy}$} (lr);
			\draw[->] (lm) to node [labelsize, auto, swap] {$fu$} (lr);
			
			\draw[->, 2cell] (2.4,2.4) to node [labelsize, auto] {$\unit$} (3.2,3.2);
			\draw[->, 2cell, bend left = 40] (4.2,1) to node [labelsize, auto, near end] {$f \bar \counit'$} (5.4,2.2);
			\draw[->, 2cell, bend right = 40] (5,0.2) to node [labelsize, auto, swap] {$\counit$} (6.2,1.4);

			\draw[->, 3cell, dashed] (4.8,1.6) to node [labelsize, auto] {$\exists \counit^\dag$} (5.6,0.8);
		\end{tikzpicture}\right]
		\qquad\sim\qquad \tri$
		\caption{Lifting \texorpdfstring{$e$}{\counit} (with \texorpdfstring{$\check \lambda$}{ˇλ}, \texorpdfstring{$\alpha$}{α} and \texorpdfstring{$\rho$}{ρ} omitted)}
		\label{lifting-e-apx}
	\end{figure}
\end{proof}

We then check that $\KH$ satisfies \ref{K2} by showing an analogue of \cref{injectivity-detects-equivalences}. 

\begin{proposition}
    \leavevmode
    \begin{enumerate}
        \item If a marked weak $\omega$-category $(X,tX)$ is $\KH$-injective, then all cells in $tX$ are equivalences in the underlying weak $\omega$-category $X$.
        \item For any weak $\omega$-category $X$, the marked weak $\omega$-category $X^\natural$ is $\KH$-injective.
    \end{enumerate}
\end{proposition}
\begin{proof}
    If a marked weak $\omega$-category $(X,tX)$ is $\KH$-injective, then we have $tX\subseteq \Phi(tX)$, where $\Phi$ is defined in \cref{def:spherical-eq}. This proves (1). 

    (2) follows from \cref{lift-one-step-apx}.
\end{proof}

Finally, that $\KH$ satisfies \ref{K3} is obvious. 

Applying the construction of \cref{subsec:omega-equifib-via-RLP} to $\KH$, we obtain a set $J_\mathcal{H}$ of strict $\omega$-functors characterising the $\omega$-equifibrations.

\bibliographystyle{plain}
\bibliography{mybib}
\end{document}